\documentclass[11pt, leqno]{amsart}
\usepackage{amsthm,amsmath,amssymb,amscd,verbatim}
\usepackage[curve,matrix,arrow,frame,tips]{xy}
\usepackage{graphicx}

\DeclareMathOperator{\Aut}{Aut}

\DeclareMathOperator{\End}{End}

\DeclareMathOperator{\Gal}{Gal}

\DeclareMathOperator{\GL}{GL}
\DeclareMathOperator{\gr}{gr}

\DeclareMathOperator{\Hom}{Hom}
\DeclareMathOperator{\id}{id}

\DeclareMathOperator{\Ind}{Ind}

\DeclareMathOperator{\Lie}{Lie}

\DeclareMathOperator{\Res}{Res}

\DeclareMathOperator{\Sch}{Sch}

\DeclareMathOperator{\Spec}{Spec}
\DeclareMathOperator{\Spf}{Spf}

\DeclareMathOperator{\Stab}{Stab}

\newcommand{\mA}{{\mathbb A}}

\newcommand{\mC}{{\mathbb C}}

\newcommand{\mF}{{\mathbb F}}
\newcommand{\mG}{{\mathbb G}}

\newcommand{\mQ}{{\mathbb Q}}
\newcommand{\mR}{{\mathbb R}}

\newcommand{\mZ}{{\mathbb Z}}

\newcommand{\A}{{\mathcal A}}

\newcommand{\G}{{\mathcal G}}

\newcommand{\fp}{{\mathfrak p}}

\newcommand{\Llrto}{\Longleftrightarrow}
\renewcommand{\to}{\longrightarrow}
\newcommand{\mto}{\longmapsto}
\newcommand{\xto}{\xrightarrow}

\newcommand{\spa}{\ }

\theoremstyle{plain}
\newtheorem{theorem}{Theorem}[section]
\newtheorem{lemma}[theorem]{Lemma}
\newtheorem{cor}[theorem]{Corollary}
\newtheorem{prop}[theorem]{Proposition}

\newtheorem{thm}{Theorem}[subsection]
\newtheorem{coro}[thm]{Corollary}

\newtheorem{propo}[thm]{Proposition}

\theoremstyle{remark}

\theoremstyle{definition}

\newtheorem{defn}[theorem]{Definition}
\newtheorem{Remark}[theorem]{Remark}

\numberwithin{equation}{subsection}
\numberwithin{theorem}{subsection}

\newcommand{\ZZ}{\mathbb Z}
\newcommand{\QQ}{\mathbb Q}
\newcommand{\RR}{\mathbb R}

\newcommand{\FF}{\mathbb F}

\newcommand\hfld[2]{\smash{\mathop{\hbox to 10mm{\rightarrowfill}}
     \limits^{\scriptstyle#1}_{\scriptstyle#2}}}
\newcommand\hflg[2]{\smash{\mathop{\hbox to 10mm{\leftarrowfill}}
     \limits^{\scriptstyle#1}_{\scriptstyle#2}}}

\title{Shimura varieties with $\Gamma_1(p)$-level via Hecke algebra isomorphisms: the Drinfeld case}
\author{Thomas J. Haines and Michael Rapoport}

\begin{document}

\thanks{Research of Haines partially supported by NSF grants FRG-0554254, DMS-0901723, a University of Maryland GRB Semester Award, and by the University of Chicago.  This work was 
also partially supported by SFB/TR 45 "Periods, Moduli Spaces and 
Arithmetic of Algebraic Varieties" of the DFG}

\date{}



\begin{abstract}
We study the local factor at $p$ of the semi-simple zeta function of a Shimura variety of Drinfeld type for a level structure given at $p$ by the pro-unipotent radical of an Iwahori subgroup. Our method is an adaptation to this case of the  Langlands-Kottwitz  counting method. We explicitly determine
the corresponding test functions in suitable Hecke algebras, and show their centrality by determining their images under the Hecke algebra isomorphisms of Goldstein, Morris, and Roche.

\end{abstract}

\maketitle

\markboth{T. Haines and M. Rapoport}
{Shimura varieties with $\Gamma_1(p)$-level via Hecke algebra isomorphisms}

\tableofcontents

\section{Introduction}

\noindent Many authors have studied the Hasse-Weil zeta functions of Shimura varieties, and more generally the relations between the cohomology of Shimura varieties and automorphic forms.  In the approach of Langlands and Kottwitz, an important step is to express the semi-simple Lefschetz number in terms of a sum of orbital integrals that resembles the geometric side of the Arthur-Selberg trace formula.  The geometry of the reduction modulo $\fp$ of 
the Shimura variety determines which functions appear as test functions in the orbital integrals.  Via the Arthur-Selberg trace formula,  the traces of these test functions on automorphic representations intervene in the expression of the semi-simple Hasse-Weil zeta function in terms of semi-simple automorphic $L$-functions.  Thus, at the heart of this method is the precise determination of the test functions themselves, for any given level structure, and this precise determination is a problem that seems of interest in its own right.

Let us describe more precisely what we mean by test functions.  Let $Sh_K = Sh(G, h^{-1}, K)$ denote the Shimura variety attached to a connected reductive $\mathbb Q$-group $G$, a family of Hodge-structures $h$, and a compact open subgroup $K \subset G(\mathbb A_f)$.  We fix a prime number $p$ and assume that $K$ factorizes as $K = K^pK_p \subset G(\mathbb A^p_f) \times G(\mathbb Q_p)$.  Let $E \subset \mathbb C$ denote the reflex field and suppose that for a fixed prime ideal $\mathfrak p \subset \mathcal O_E$ over $p$, $Sh_K$ possesses a suitable integral model, still denoted by $Sh_K$, over the ring of integers $\mathcal O_{E_{\mathfrak p}}$ in the $\mathfrak p$-adic completion $E_{\mathfrak p}$.  Also, for simplicity let us assume that $E_{\frak p} = \mathbb Q_p$ (this will be the case in the body of this paper.)   Write $R\Psi = R\Psi^{Sh_K}(\bar{\mathbb Q}_\ell)$ for the complex of nearby cycles on the special fiber of $Sh_K$, where $\ell \neq p$ is a fixed prime number.  For $r \geq 1$, let $k_r$  denote an extension of degree $r$ of  $\mathbb F_p$, and consider the semi-simple Lefschetz number defined as the element in $\bar{\mathbb Q}_\ell$ given by the sum
\begin{equation*}
{\rm Lef}^{ss}(\Phi^r_{\fp},R\Psi) = \sum_{x \in Sh_K(k_r)} {\rm Tr}^{ss}(\Phi^r_{\fp}, R\Psi_x) ,
\end{equation*}
where  $\Phi_{\fp}$ denotes the Frobenius morphism for 
$Sh_K \otimes_{\mathcal O_{E_\fp}} {\mathbb F_p}$.  (We refer the reader to \cite{HN} for a discussion of semi-simple trace on $\bar{\mathbb Q}_\ell$-complexes being used here.)

The group-theoretic description of ${\rm Lef}^{ss}(\Phi^r_{\fp},R\Psi)$ emphasized by Langlands \cite{L} and Kottwitz \cite{K84,K90} should take the form of a ``counting points formula''
\begin{equation*}
{\rm Lef}^{ss}(\Phi^r_{\fp},R\Psi) = \sum_{(\gamma_0;\gamma,\delta)} c(\gamma_0;\gamma,\delta) \, {\rm O}_\gamma(f^p) \, {\rm TO}_{\delta \sigma}(\phi_r).
\end{equation*}
The sum is over certain equivalence classes of triples $(\gamma_0;\gamma,\delta) \in G(\mathbb Q) \times G(\mathbb A^p_f) \times G(L_r)$, where $L_r$ is the fraction field of the ring of Witt  vectors $W(k_r)$, and where $\sigma$ denotes the Frobenius generator of ${\rm Gal}(L_r/\mathbb Q_p)$.  
We refer to section \ref{counting_pts_sec} for a discussion of the implicit measures and the remaining notation in this formula.  

We take the open compact subgroup $K_p$ of the kind that for any $r\geq 1$ it also defines an open compact subgroup $K_{p^r}$ of $G(L_r)$. In terms of the Haar measure $dx_{K_{p^r}}$ on $G(L_r)$ giving $K_{p^r}$ volume 1, we say $\phi_r = \phi_r(R\Psi)$ is ``the'' {\em test function} if the counting points formula above is valid.  It should be an element in the Hecke algebra
$\mathcal H(G(L_r),K_{p^r})$ of compactly-supported $K_{p^r}$-bi-invariant $\bar{\QQ}_\ell$-valued functions on $G(L_r)$.   The function $\phi_r$ is not uniquely-determined by the formula, but a striking empirical fact is that a very nice choice for $\phi_r$ always seems to exist, namely, {\em we may always find a test function $\phi_r$ belonging to the center ${\mathcal Z}(G(L_r), K_{p^r})$ of $\mathcal H(G(L_r),K_{p^r})$}.  (A very precise conjecture can be formulated using the {\em stable Bernstein center} (cf. \cite{Vo}), but this will be done on another occasion by one of us (T. H.).)

In fact,  test functions seem to be given in terms of the Shimura cocharacter $\mu_h$ by a simple rule.  Let us illustrate this with the known examples.  Assume $G_{\mathbb Q_p}$ is unramified.   Let $\mu = \mu_h$ denote the minuscule cocharacter of $G_{\mathbb Q_p}$ associated to $h$ and the choice of an embedding $\bar{\mathbb Q} \hookrightarrow \bar{\mathbb Q}_p$ extending $\fp$, and let $\mu^*$ denote its dual (i.e. $\mu^* = -w_0(\mu)$ for the longest element $w_0$ of the relative Weyl group).  If $K_p$ is a {\it hyperspecial} maximal compact subgroup, and $Sh_K$ is PEL of type $A$ or $C$, then  Kottwitz \cite{K92} shows that $Sh_K$ has an integral model with good reduction at $\mathfrak p$, and that $\phi_r$ may be taken to be the characteristic function $ 1_{K_{p^r}\mu^*(p)K_{p^r}}$ in the (commutative) spherical Hecke algebra $\mathcal H(G(L_r),K_{p^r})$.  Next suppose that $Sh_K$ is of PEL type, that $G_{\mathbb Q_p}$ is split of type $A$ or $C$, and that $K_p$ is a parahoric subgroup.  Then in \cite{HN}, \cite{H05} it is proved that $\phi_r$ may be taken to be the {\em Kottwitz function} $k_{\mu^*,r} = p^{r \, {\rm dim}(Sh_K)/2} \, z^{K_{p^r}}_{\mu^*,r}$ in  $\mathcal Z(G(L_r),K_{p^r})$, thus confirming a conjecture of Kottwitz.  
Here for an Iwahori subgroup $I_r$ contained in $K_{p^r}$ the symbol $z_{\mu^*,r}$ denotes the {\em Bernstein function} in $\mathcal Z(G(L_r), I_r)$ described in the Appendix, and $z^{K_{p^r}}_{\mu^*,r}$ denotes the image of $z_{\mu^*,r}$ under the canonical isomorphism $\mathcal Z(G(L_r),I_r)~ \widetilde{\rightarrow} ~ \mathcal Z(G(L_r),K_{p^r})$ given by convolution with the characteristic function of $K_{p^r}$ (see \cite{H09}). 

In this article, we study Shimura varieties in the Drinfeld case.   From now on we fix the associated ``unitary'' group $G$, coming from an involution on a semi-simple algebra $D$ with center an imaginary quadratic extension $F/\mathbb Q$ and with ${\rm dim}_F D = d^2$ (see section \ref{moduli_sec}).  We make assumptions guaranteeing that $G_{\mathbb Q_p} = {\rm GL}_d \times \mathbb G_m$, so that we may identify $\mu$ with the Drinfeld type cocharacter $\mu_0 = (1,0^{d-1})$ of ${\rm GL}_d$.  For $K_p$ we take    either the Iwahori subgroup $I$ of ${\rm GL}_d(\mathbb Z_p)\times \mathbb Z_p^\times$ of elements where the matrices in the first factor are congruent modulo $p$ to an upper triangular matrix, or its pro-unipotent radical $I^+$ where the matrices in the first factor are congruent modulo $p$ to a unipotent upper triangular matrix. The first case leads to the  $\Gamma_0(p)$-level moduli scheme, and the second to  the $\Gamma_1(p)$-level moduli scheme. 

The $\Gamma_0(p)$-level moduli scheme $\A_0$ is defined as parametrizing a chain of polarized abelian varieties of dimension $d^2$ with an action of a maximal order in $D$. 
 The moduli problem $\A_1$ is defined using the Oort-Tate theory \cite{OT} of group schemes of order $p$. Namely,  $\A_1 \rightarrow \A_0$ is given by choosing an Oort-Tate generator for each group scheme $G_i\ (i=1,\ldots,d)$ coming from the flag (\ref{chainpdiv}) of $p$-divisible groups attached to a point in $\A_0$.  We use the geometry of the canonical morphism 
$$
\pi: \A_1 \rightarrow \A_0
$$
to study the nearby cycles $R\Psi_1$ on $\A_1 \otimes \mathbb F_p$.  The following theorem summarizes Propositions \ref{central_prop} and \ref{exp_test_fcn}.  It is the analogue for $\mathcal A_1$ of the theorem for  $\mathcal A_0$ mentioned above, which identifies the test function $\phi_r(R\Psi_0)$ with the Kottwitz function $k_{\mu^*,r} $ in the center of the Iwahori Hecke algebra. 

\begin{thm} \label{gamma1_test} Consider the Shimura variety $Sh_K = \A_{1,K^p}$ in the Drinfeld case, where  $K_p = I^+$ is the pro-unipotent radical of an Iwahori subgroup $I \subset G(\mathbb Q_p)$.  Let $I^+_r \subset I_r$  denote the corresponding subgroups of $G_r := G(L_r)$.  Then with respect to the Haar measure on $G_r$ giving $I^+_r$ volume 1, the test function $\phi_{r, 1} = \phi_r(R\Psi_1)$ is an explicit function belonging to the center $\mathcal Z(G_r, I^+_r)$ of $\mathcal H(G_r, I^+_r)$.  
\end{thm}
(See section \ref{sum_up_sec} for the spectral characterization of $\phi_{r,1}$ and an explicit formula for it.)

The centrality of test functions facilitates the pseudo-stabilization of the counting points formula.  The test function $\phi_{r, 1} \in \mathcal Z(G_r, I^+_r)$  is characterized by its traces on depth-zero principal series representations, and a similar statement applies to its image under an appropriate base-change homomorphism
$$
b_r: \mathcal Z(G_r,I^+_r) \rightarrow \mathcal Z(G,I^+)
$$
which we discuss in section \ref{BC_sec}.   As for the spherical case \cite{Cl,Lab} and the parahoric case \cite{H09}, the present $b_r$ yields pairs of associated functions with matching (twisted) stable orbital integrals \cite {H10} (see Theorem \ref{fl} for a statement), and this plays a key role in the pseudo-stabilization, cf. subsection \ref{L_factor_sec}. In our case, the image $b_r(\phi_{r, 1})$ of  $\phi_{r, 1}$ has the following  nice spectral expression  (Cor. \ref{spectral_cor}). 

\begin{propo} \label{spectral_prop}
Write $f_{r, 1} := b_r(\phi_{r, 1}) \in \mathcal Z(G,I^+)$.  Let $(r_{\mu^*},V_{\mu^*})$ denote the irreducible representation  with extreme weight $\mu^*$ of the L-group $^L(G_{\mathbb Q_p})$.  Let $\Phi \in W_{\mathbb Q_p}$ denote a geometric Frobenius element and let $I_p \subset W_{\mathbb Q_p}$ denote the inertia subgroup.   Then for any irreducible smooth representation $\pi_p$ of $G(\mathbb Q_p)$ with Langlands parameter $\varphi'_{\pi_p}: W'_{\mathbb Q_p} \rightarrow \widehat{G}$,  we have
$$
{\rm tr} \, \pi_p(f_{r, 1} ) = {\rm dim}(\pi_p^{I^+}) \, p^{r \langle \rho, \mu^* \rangle} \, {\rm Tr}(r_{\mu^*} \circ \varphi'_{\pi_p}(\Phi^r) , \spa V_{\mu^*}^{I_p}).
$$
\end{propo}
It is expected (cf. \cite{Rp0}) that semi-simple zeta functions can be expressed in terms of semi-simple automorphic $L$-functions.  By their very definition, the semi-simple local $L$-factors $L^{ss}(s, \pi_p, r_{\mu^*})$ involve expressions as on the right hand side above.  In section \ref{L_factor_sec} we use this and Theorem \ref{fl} to describe the semi-simple local factor $Z^{ss}_{\fp}(s, \A_{1,K^p})$ of $\A_1 = \A_{1,K^p}$,  when the situation is also ``simple'' in the sense of Kottwitz \cite{K92b}.

\begin{thm} \label{main_tr_thm}  Suppose $D$ is a division algebra, so that the Shimura variety $\A_1 = \A_{1,K^p}$ is proper over 
$\Spec \mathbb Z_p$ and has ``no endoscopy'' (cf. \cite{K92b}).  Let $n=d-1$ denote the relative dimension of  $\A_{1,K^p}$.  For every integer $r \geq 1$, the alternating sum of the semi-simple traces
$$
\sum_{i = 0}^{2n} (-1)^i \, {\rm Tr}^{ss}(\Phi_{\fp}^r \, , \, {\rm H}^i(\A_{1,K^p} \otimes_E \bar{E}_{\mathfrak p} , \bar{\QQ}_\ell))
$$
equals the trace
$$
{\rm Tr}( 1_{K^p} \otimes f_{r, 1}  \otimes  f_\infty \, , \, {\rm L}^2(G(\QQ) \, A_G(\RR)^\circ \backslash G(\mathbb A))).
$$
Here $A_G$ is the split component of the center of $G$ and $f_\infty$ is the function at the archimedean place defined by Kottwitz \cite{K92b}.  
\end{thm}

In section \ref{L_factor_sec} we derive from the two previous facts the following result.

\begin{coro} \label{zeta_cor}
In the situation above, we have
$$Z^{ss}_{\fp}(s,\A_{1,K^p})=\prod_{\pi_f}L^{ss}(s- {n \over 2},
\pi_p,r_{\mu^*})^{a(\pi_f)\dim(\pi_f^K)} ,$$
where the product runs over all admissible representations $\pi_f$ 
of $G(\mathbb A_f)$. The integer $a(\pi_f)$ is given by
$$a(\pi_f)=\sum_{\pi_\infty\in \Pi_\infty} m(\pi_f\otimes \pi_\infty) {\rm tr}
\,\pi_\infty(f_\infty) ,$$
where $m(\pi_f\otimes \pi_\infty)$ is the multiplicity of 
$\pi_f\otimes \pi_\infty$ in ${\rm L}^2(G(\QQ)A_G(\RR)^0\backslash 
G(\mathbb A))$. Here $\Pi_\infty$ is the set of admissible representations 
of $G(\RR)$ having trivial central character and trivial infinitesimal
character.
\end{coro}
We refer to  Kottwitz's  paper \cite{K92b} for further discussion
of the integer $a(\pi_f)$. Note that in the above product there are only a 
finite number of representations $\pi_f$  of $G(\mathbb A_f)$ such that 
$a(\pi_f)\dim(\pi_f^K)\not=0$. 

\smallskip

Of course, this corollary is in principle a special case of the results of Harris and Taylor in \cite{HT} which are valid for any level structure. However, our emphasis here is on the explicit determination of test functions and their structure. To be more precise, let us  return to the general Drinfeld case, where $D$ is not necessarily a division algebra.    By the compatibility of the semi-simple trace with the finite morphism $\pi$ from $\A_1$ to $\A_0$, we have ${\rm Lef}^{ss}(\Phi^r_{\fp}, R\Psi_1) = {\rm Lef}^{ss}(\Phi^r_{\fp}, \pi_*R\Psi_1)$, cf. \cite{HN}.  In terms of test functions this means 
$$
\phi_r(R\Psi_1) = [I_r:I^+_r]^{-1} \, \phi_r(\pi_*R\Psi_1).
$$
(The factor $[I_r:I^+_r]$ appears since we use $dx_{I^+_r}$ to normalize $\phi_r(R\Psi_1)$, and $dx_{I_r}$ to normalize $\phi_r(\pi_*R\Psi_1)$.)  Thus to establish Theorem \ref{gamma1_test},  it is enough to study $\pi_*R\Psi_1$ on $\A_0 \otimes \mathbb F_p$.  We do this  by decomposing $\pi_*R\Psi_1$ into ``isotypical components'', as follows.

The  covering $\pi$ is finite flat of degree $(p-1)^d$, and its generic fiber $\pi_\eta$ is a connected\footnote{See Corollary \ref{Galoiscov}.} \'etale Galois cover with group $T(\mathbb F_p)$, where $T$ denotes the standard split torus $T = \mathbb G_m^d$.  Following section \ref{A1_sec},  we obtain a decomposition into  $T(\mF_p)$-eigenspaces
\begin{equation*}
\pi_{\eta *}(\bar{\mQ}_\ell)=\bigoplus_{\chi: T(\mF_p)\to \bar{\mQ}_\ell^\times}\bar{\mQ}_{\ell, \chi}\  , 
\end{equation*}
and, setting $R\Psi_\chi = R\Psi(\bar{\mQ}_{\ell, \chi})$,  we also get
\begin{equation*}
\pi_*R\Psi_1=\bigoplus_{\chi: T(\mF_p)\to \bar{\mQ}_\ell^\times}R\Psi_\chi\  .
\end{equation*}
In Theorem \ref{main_geom_thm} we describe the restriction of $R\Psi_\chi$ to any KR-stratum $\A_{0,S}$ (in fact, in Theorem \ref{main_geom_thm} we only describe the inertia invariants in the cohomology sheaves of this restriction, but this is sufficient for our purpose).  Here for a non-empty subset $S \subset \mathbb Z/d$, the stratum $\A_{0,S}$ is the locus of points $(A_\bullet,\lambda, i, \eta) \in \A_0 \otimes \mathbb F_p$ such that the order $p$ group scheme $G_i = {\rm ker}(X_{i-1} \rightarrow X_i)$, defined by (\ref{chainpdiv}), is \'etale if and only 
if $i \notin S$.  
In Corollary \ref{Galoiscov}, we show that $\pi$ restricts over $\A_{0,S}$ to give an \'etale Galois covering
$$
\pi_S: (\A_{1,S})_{\rm red} \rightarrow \A_{0,S} \ ,
$$
whose group of deck transformations is the quotient torus $T^S(\mathbb F_p) = T(\mathbb F_p)/T_S(\mathbb F_p)$.  Here $T_S = \prod_{i \in S} \mathbb G_m$ and $T^S = \prod_{i \notin S} \mathbb G_m$.    

In Theorem \ref{monodr} (which is perhaps of independent interest) we show following the method of 
Strauch \cite{Str} that the monodromy of $\pi_S$ is as large as possible, in other words, that for any $x \in \A_{0,S}(\bar{\mF}_p)$ the induced homomorphism
$$
\pi_1(\A_{0,S},x) \rightarrow T^S(\mF_p)
$$
is {\em surjective}.  This is one ingredient in the proof of the following theorem (Theorem \ref{main_geom_thm}).   Let $i_S: \A_{0,S} \rightarrow \A_0 \otimes \mathbb F_p$ denote the locally closed embedding of a KR-stratum.  

\begin{thm} \label{i*RPsi1}   The action of the inertia group $I_p$ on the $i$-th cohomology sheaf of $i^*_S(R\Psi\chi)$ is through a finite quotient, and 
its  sheaf of invariants $\big(\mathcal H^i\big(i_S^*(R\Psi_\chi)\big)\big)^{I_p}$ vanishes, unless $\chi$ factors through a character $\bar{\chi}: T^S(\mathbb F_p) \rightarrow \bar{\QQ}_\ell^\times$.  In that case, there is a canonical isomorphism, 
$$
\big(\mathcal H^i\big(i_S^*(R\Psi_\chi)\big)\big)^{I_p} \cong (\pi_{S*}(\bar{\QQ}_\ell))_{\bar{\chi}} \otimes \mathcal H^i\big(i^*_S(R\Psi_0)\big).
$$
Here the sheaf $(\pi_{S*}(\bar{\QQ}_\ell))_{\bar{\chi}}$ is a rank one local system (with continuous compatible Galois action) defined on each connected component of $\A_{0,S}$ by the composition
$$
\pi_1(\A_{0, S}\otimes \bar{\mF}_p, x)\twoheadrightarrow T^S(\mF_p)\xto{\bar \chi}\bar{\mQ}_\ell^\times \ .
$$
By $R\Psi_0=R\Psi^{\mathcal A_0}$ we denote  the complex of nearby cycles for $\mathcal A_0$. 
\end{thm}
This sheaf-theoretic product decomposition is mirrored by the corresponding product decomposition of the  test function $\phi_{r, \chi}=\phi_r(R\Psi_\chi)$:  on the elements $tw \in T(k_r) \rtimes \widetilde{W}$ indexing $I^+_r \backslash G_r / I^+_r$ we set
$$
\phi_{r, \chi}(tw) := \delta^1(w^{-1},\chi) \, \chi_r^{-1}(t) \, k_{\mu^*,r} (w).
$$
Here $\delta^1(w^{-1},\chi) \in \{0,1\}$ is defined in (\ref{delta1_defn}), and $\chi_r$ denotes the composition of the norm homomorphism $N_r: T(k_r) \rightarrow T(\mathbb F_p)$ and $\chi: T(\mathbb F_p) \rightarrow \bar{\QQ}_\ell^\times$.  The factor $\delta^1(w^{-1}, \chi) \, \chi_r^{-1}(t)$ comes directly from $(\pi_{S*}(\bar{\QQ}_\ell))_{\bar{\chi}}$ (see Lemmas \ref{sommes_trig} and \ref{t_g_lemma}).   In section \ref{counting_pts_sec} we prove that this function really is a test function for $R\Psi_\chi$. 

In fact, a stronger {\em point-by-point} statement is true. To a point $x\in \mathcal A_0(k_r)$ there is canonically associated an element  $w=w_x\in \widetilde W$  which lies in the $\mu$-{\em admissible set} ${\rm Adm}^G(\mu)$ and which corresponds to the unique  non-empty subset 
$S(w)$ of $\mZ/d$ with $x\in \mathcal A_{0, S(w)}$ (cf. Proposition \ref{critad}). Furthermore, associated to $x$ there is an element $t_x\in T^{S(w)}(\mF_p)$ which measures the action of the Frobenius $\Phi^r_\fp$ on the fiber of $\pi$ over $x$, cf.\ (\ref{t_x_def_eqn}).  The following is Theorem \ref{chi_tr_Fr}.  
\begin{thm} \label{tr_Fr_intro}
For $x\in \mathcal A_0(k_r)$, 
$${\rm Tr}^{ss}(\Phi^r_{\fp}, (R\Psi_{\chi})_x) =\phi_{r, \chi}(t^{-1} w^{-1}_x) ,
$$
where $t\in T(k_r)$ is any element such that $ N_r(t)\in T(\mF_p)$ projects to $t_x$. 
\end{thm}

By its very construction $\phi_{r,\chi}$ belongs to the Hecke algebra $\mathcal H(G_r,I_r, \chi_r)$  (see section \ref{HA_isom_sec}).  The final step is to use the Hecke algebra isomorphisms of Goldstein, Morris, and  Roche (Theorem \ref{GMR}) to prove that $\phi_{r, \chi}$ lies in the center  $\mathcal Z(G_r, I_r, \chi_r)$ of $\mathcal H(G_r,I_r, \chi_r)$, and to identify it explicitly as an element of the Bernstein center of $G_r$.  
In our setting,  the Hecke algebra isomorphisms  send $\mathcal H(G_r, I_r, \chi_r)$ to the Iwahori Hecke algebra $\mathcal H(M_r, I_{M_r})$ for a certain semistandard Levi subgroup $M = M_\chi$ that depends on $\chi$ via the decomposition of $\{1,\ldots, d\}$ defined by $\chi$ (section \ref{deltadef}).  In fact,  there are many Hecke algebra isomorphisms $\Psi_{\breve \chi}$, each determined by the choice of an extension of $\chi_r$ to a smooth character $\breve{\chi}_r: N_\chi(L_r) \rightarrow \bar{\QQ}_\ell^\times$, where $N_\chi$ denotes the $\chi$-fixer in the normalizer $N$ of the standard split torus $T$ in $G$ (cf.\  Lemma \ref{extn_lemma} and Theorem \ref{GMR}).  
There is a particular choice of a Hecke algebra isomorphism $\Psi_{\breve{\chi}^\varpi}$ which is most useful here (see Remark \ref{varpi_can_rem}), and which is used in the following proposition (Prop. \ref{test_fcn_image}). This proposition exhibits  a certain coherency between test functions for  $G$ and for  its Levi subgroups, and is a key ingredient in the proof of Proposition \ref{spectral_prop}.

\begin{propo}  Let $M = M_\chi$ as above.  Let $\varpi = p$, viewed as a uniformizer in $L_r$.  The Hecke algebra isomorphism
$$
\Psi_{\breve{\chi}^\varpi}: \mathcal H(G_r,I_r, \chi_r) ~ \widetilde{\rightarrow} ~ \mathcal H(M_{ r}, I_{M_r})
$$
takes the function $\phi_{r, \chi}$ to a power of $q(=p^r)$ times a particular Kottwitz function $k^M_{\mu^{1*}_\chi}$ in $\mathcal Z(M_r, I_{M_r})$.  In particular, $\phi_{r, \chi}$ is a function in the center $\mathcal Z(G_r,I_r, \chi_r)$ and its traces on all irreducible smooth representations $\pi_p$ can be understood explicitly in terms of the Bernstein decomposition.
\end{propo}

Let us remark that this proposition relies in a crucial way on Proposition \ref{main_comb_prop} (e), a combinatorial miracle which is perhaps special to the Drinfeld case. One might still hope that similar results to the ones obtained here are valid for other Shimura varieties with a level structure at $p$ given by the pro-unipotent radical of an Iwahori subgroup. 

Our results bear a certain resemblance  to the results of Harris and Taylor in \cite{HT2}. They also use Oort-Tate theory to analyze a Galois covering of  $\A_0$, and apply this to the calculation of the complex of nearby cycles, by decomposing it according to the action of the covering group.    Their covering is of degree $p-1$ and is dominated by $\A_1$. However, the analysis of test functions does not play a role in their paper. We also point out the papers \cite{Sc1, Sc2} by P.\ Scholze, in which he shows that the Langlands-Kottwitz counting method applies quite generally to the determination of the local factor of the semi-simple zeta function in the Drinfeld case, for an {\em arbitrary} level structure at $p$. More precisely, in the case of the elliptic moduli scheme,  Scholze \cite{Sc1}  gives an explicit test function for any $\Gamma(p^n)$-level structure, and he also gives a pointwise formula analogous to Theorem \ref{tr_Fr_intro} above.  In  the general Drinfeld case,  but assuming $D$ is a division algebra,  Scholze \cite{Sc2} constructs a test function for arbitrarily deep level structures; however, to obtain an explicit expression of this function in the general case  seems hopeless, and also  the relation between his  test functions and the geometry seems to be less close, since an analogue of Theorem \ref{tr_Fr_intro} is missing. It would be interesting to relate Scholze's approach to ours, and perhaps deduce {\em a posteriori} the explicit expression  of our test functions from his general results. 

We now explain the lay-out of the paper. 

In \S\ref{notation_sec} we list our global notational conventions. In  \S\ref{moduli_sec} we define the moduli schemes $\A_0$ and $\A_1$, and prove basic results concerning their structure, in particular that they have semi-stable reduction and that the morphism $\pi:\A_1\to\A_0$ is finite and flat. In \S\ref{monodromy_sec} we prove that the monodromy of the restriction of $\pi$ to each stratum $\A_{0, S}$ of $\A_0\otimes\mF_p$ is as large as possible.  In \S\ref{nearby_sec} we determine the complexes $R\Psi_0$ and $\pi_*(R\Psi_1)$ along each stratum $\A_{0, S}$. In \S\ref{comb_sec} we discuss the {\it admissible set} ${\rm Adm}(\mu)$ in the case at hand and analyze some combinatorial aspects of the Kottwitz function. In \S\ref{tracefrob_sec} we prove the ``pointwise" Theorem \ref{tr_Fr_intro} above. In \S\ref{counting_pts_sec} we use this theorem to prove the ``counting 
points formula" for the semi-simple Lefschetz number.  In \S\ref{HA_isom_sec} we present recollections on the Hecke algebra isomorphisms of Goldberg, Morris, and Roche, and their relation to the Bernstein decomposition. In \S\ref{BC_sec} we give the facts on the base change homomorphisms in the depth zero case that are used in the sequel. In \S\ref{image_sec} we determine the image of our test functions $\phi_{r, \chi}$ under the Hecke algebra isomorphisms, and show the centrality of these functions. In \S\ref{sum_up_sec} we sum up explicitly all test functions $\phi_{r, \chi}$ to $\phi_{r, 1}$ and show that $\phi_{r, 1}$  is central in the Hecke algebra $\mathcal H(G_r, I_r^+)$.  In \S\ref{pseud_stab_sec} everything comes together, when we show that the Kottwitz method can be applied to yield Theorem \ref{main_tr_thm} and Corollary \ref{zeta_cor} above.   Finally, the Appendix proves in complete generality that the Bernstein function $z_\mu$ is supported 
on the set ${\rm Adm}(\mu)$; a special case of this is used in $\S\ref{comb_sec}$.

We thank G.\ Henniart, R.\ Kottwitz,  S.\ Kudla, G.\ Laumon, and P.\ Scholze for helpful discussions.   We also thank U.\ G\"{o}rtz for providing the figure in section \ref{comb_sec}.

\section{Notation} \label{notation_sec}

Throughout the paper,  $p$ denotes a fixed rational prime,  and $\ell$ denotes a prime different from $p$.  

We list our notational conventions related to $p$-adic groups and buildings.  Outside of section \ref{moduli_sec}, the symbol $F$ will denote a finite extension of $\mathbb Q_p$,  with ring of integers $\mathcal O = \mathcal O_F$, and residue field $k_F$ having characteristic $p$.  We often fix a uniformizer $\varpi = \varpi_F$ for $F$.    

For $r \geq 1$, we write $F_r$ for the unique degree $r$ unramified extension of $F$ contained in some algebraic closure of $F$.  We write $\mathcal O_r$ for its ring of integers and $k_r$ for its residue field.  Usually, $F = \mathbb Q_p$ and then $k_r = \mathbb F_{p^r}$,  and we then write $\mathbb Q_{p^r}$ for $F_r$.  When we work with $W(k_r)$, the ring of Witt vectors associated to $k_r$, we will write $L_r$ (and not $F_r$) for its fraction field. 

If $G$ is a split connected reductive group defined over $\mathcal O_F$, the symbol $T$ will denote a  $F$-split maximal torus in $G$, and $B$  a  $F$-rational Borel subgroup containing $T$.  Then $T$ determines  an apartment $\mathcal A_T$ in the Bruhat-Tits building $\mathcal B(G,F)$.  We may assume that the vertex $v_0$ corresponding to the hyperspecial maximal compact subgroup $G(\mathcal O_F)$ lies in $\mathcal A_T$.  The choice of $B$ determines the dominant and antidominant Weyl chambers in $\mathcal A_T$.  Fix the alcove ${\bf a}$ which is contained in the {\it antidominant} chamber and whose closure contains $v_0$.  Let $N$ denote the normalizer of $T$ in $G$.  
Having fixed $v_0, {\bf a}$ and $\mathcal A_T$, we may identify the {\em Iwahori-Weyl group} $\widetilde{W} := N(F)/T(\mathcal O_F)$ with the extended affine Weyl  group $X_*(T) \rtimes W$, and furthermore the latter decomposes canonically as 
\begin{equation} \label{W_aff_Omega}
\widetilde{W} = W_{\rm aff} \rtimes \Omega ,
\end{equation}
where $W_{\rm aff}$ is the Coxeter group with generators the simple affine reflections $S_{\rm aff}$ through the walls of ${\bf a}$,  and where $\Omega$ is the subgroup of $\widetilde{W}$ which stabilizes ${\bf a}$ (with respect to the obvious action of $\widetilde{W}$ on the set of all alcoves in $\mathcal A_T$).  The Bruhat order $\leq$ on $W_{\rm aff}$ determined by $S_{\rm aff}$ extends naturally to the Bruhat order $\leq$ on $\widetilde{W}$: if $w\tau$ and $w'\tau'$ are elements of $\widetilde{W}$ decomposed according to (\ref{W_aff_Omega}), then we write $w\tau \leq w'\tau'$ if and only if $\tau = \tau'$ and $w \leq w'$.

For any $\mu \in X_*(T)$, we define the {\em $\mu$-admissible set} ${\rm Adm}^G(\mu)$ to be
\begin{equation*} \label{adm(mu)_def}
{\rm Adm}^G(\mu) = \{ w \in \widetilde{W} ~ | ~ w \leq t_\lambda, \,\,\mbox{for some $\lambda \in W(\mu)$} \},
\end{equation*}
cf. \cite{Rp}, \S3. We write ${\rm Adm}(\mu)$ in place of ${\rm Adm}^G(\mu)$ if the group $G$ is understood.

We embed $\widetilde{W}$ into $N_G(T)(F)$, using a fixed choice of $\varpi$, as follows.  Let $w = t_\lambda \bar{w} \in X_*(T) \rtimes W$.  We send $t_\lambda$ to $\lambda(\varpi) \in T(F)$ and $\bar{w}$ to any fixed representative in $N_{G(\mathcal O)}(T(\mathcal O))$.  (If $G = {\rm GL}_d$, we shall always send $\bar{w}$ to the corresponding permutation  matrix in ${\rm GL}_d(\mathcal O)$.)  These choices give us a {\em set-theoretic} embedding
$$
i_{\varpi}: \widetilde{W} \hookrightarrow G(F).
$$

Let $I$ be the Iwahori subgroup corresponding to the alcove  ${\bf a}$, and let $I^+$ denote the pro-unipotent radical of $I$.  We use Teichm\"{u}ller representatives to fix the embedding $T(k_F) \hookrightarrow T(\mathcal O)$.  Then there is an isomorphism
$$
T(k_F)~ \widetilde{\rightarrow} ~ T(\mathcal O)/\big(T(\mathcal O) \cap I^+\big) ~ \widetilde{\rightarrow} ~ I/I^+,
$$
where the second arrow is induced by the inclusion $T(\mathcal O) \subset I$.  

Using this together with $i_\varpi$, we can state as follows the Bruhat-Tits decomposition,  and its variant for $I^+$ replacing $I$:
\begin{align*}
I \backslash G(F) / I &= \widetilde{W} \\
I^+ \backslash G(F) / I^+ &= T(k_F) \rtimes \widetilde{W}.
\end{align*}
For $w \in \widetilde{W}$, the coset $IwI = Ii_\varpi(w)I$ is independent of the choice of $\varpi$; but the second decomposition really does depend on $\varpi$.  

For a character $\chi: T(k_F)\to \overline{\mQ}_\ell^\times$ we denote by $\mathcal H(G, I, \chi)$ the Hecke algebra of smooth compactly supported functions $f: G(F)\to \overline{\mQ}_\ell$ such that $f(i_1gi_2)=\chi^{-1}(i_1)f(g)\chi^{-1}(i_2)$ for $i_1, i_2\in I$, where the character $\chi$ is extended to $I$ via the homomorphism $I\to T(k_F)$. When $\chi={\rm triv}$ is the trivial character, we use the notation $\mathcal H(G, I)$. By $\mathcal Z(G, I, \chi)$, resp. 
$\mathcal Z(G, I)$ we denote the centers of these algebras. Similarly, we have $\mathcal H(G, I^+)$ and $\mathcal Z(G, I^+)$.

Finally, we specify some more general notation.  If $S \subset X$ is a subset, then $1_S$ will denote the characteristic function of $S$.  If $S$ is contained in a group $G$, and $g \in G$, then $^gS := gSg^{-1}$.  For a character $\chi$ on a torus $T$ and an element $w \in W = N_G(T)/T$, we define a new character $^w\chi$ by $^w\chi(t) = \chi(w^{-1}tw)$.  

\section{Moduli problem} \label{moduli_sec}

\subsection{Rational and integral data} Let  $F$ be an imaginary quadratic extension of $\mQ$ and fix an embedding $\epsilon : F\to\mC$. We denote by $x\mapsto \bar x$ the non-trivial automorphism of $F$. Let
$p$ be a prime number that splits as a product of distinct prime ideals in $F$,
\begin{equation}\label{splits}
p\ \mathcal O_F = \fp\cdot\overline{\fp}\spa .
\end{equation}
Let $D$ be a central simple $F$-algebra of dimension $d^2$ and let $\ast$ be an  involution of $D$ inducing the non-trivial
automorphism $\overline{}$ on $F$. According to the decomposition (\ref{splits}), we have 
$F\otimes\mQ_p = F_\fp\times F_{\overline{\fp}}$, with $F_\fp = F_{\overline{\fp}} = \mQ_p$. Similarly, the 
$\mQ_p$-algebra $D\otimes_\mQ \mQ_p$ splits as
\begin{equation}\label{Dsplits}
D\otimes\mQ_p = D_\fp\times D_{\overline{\fp}}\spa ,
\end{equation}
and the involution $\ast$ identifies $D_{\overline{\fp}}$ with $D_\fp^{\rm opp}$. We make the assumption that $D_\fp$
is isomorphic to ${\rm M}_d (\mQ_p)$.   We fix such an isomorphism which has the additional property that the resulting isomorphism $D \otimes \mQ_p \simeq {\rm M}_d(\mQ_p) \times {\rm M}_d(\mQ_p)^{\rm opp}$ transports the involution $\ast$ over to the involution $(X,Y) \mapsto (Y^t,X^t)$, for $X,Y \in {\rm M}_d(\mQ_p)$.

We introduce the algebraic group $G$ over $\mQ$, with values in a $\mQ$-algebra $R$ equal to
\begin{equation*}
G (R) = \{x\in (D\otimes_\mQ R)^\times\mid x\cdot x^\ast\in R^\times\}\spa .
\end{equation*}
Then using the decomposition (\ref{Dsplits}) we have
\begin{equation*}
G\times_{\Spec\mQ}\Spec\mQ_p\simeq\GL_d\times \mG_m\spa .
\end{equation*}
Let $h_0: \mC\to D_\mR$ be an $\mR$-algebra homomorphism such that $h_0(z)^*=h_0(\bar z)$ and such that the involution $x\mapsto h_0(i)^{-1}x^*h_0(i)$ is positive. We recall from \cite {H05}, p.\ 597--600 how to associate to these data a PEL-datum $(B, \iota, V, \psi(\cdot, \cdot), h_0)$. 

Let $B=D^{\rm opp}$, and $V=D$, viewed as a left $B$-module, free of rank 1, using right multiplications. Then $D$ can be identified with $C={\rm End}_B(V)$, and we use this identification to define $h_0: \mC\to C_\mR$.    There exists $\xi\in D^\times$ such that $\xi^*=-\xi$ and such that the involution $\iota$ on $B$ given by $x^\iota=\xi x^*\xi^{-1}$ is positive. We define the non-degenerate alternating pairing $\psi (\cdot, \cdot): D\times D\to \mQ$ by 
\begin{equation*}
\psi(v, w)={\rm tr}_{D/\mQ}(v \xi w^*) . 
\end{equation*}
Then 
\begin{equation*}
\psi(xv, w) = \psi(v, x^\iota w)\ ,\ v, w\in V,\ x\in D\spa ,
\end{equation*}
and, by possibly replacing $\xi$ by its negative, we may assume that $\psi(\cdot, h_0(i)\cdot)$ is positive-definite. 

These are the rational data we need to define our moduli problem. As integral data we denote by $\mathcal O_B$ the unique
maximal $\mZ_{(p)}$-order in $B$ such that under our fixed identification 
$B\otimes_\mQ\mQ_p = {\rm M}_d (\mQ_p)\times {\rm M}_d (\mQ_p)^{\rm opp}$ we have
\begin{equation*}
\mathcal O_B\otimes\mZ_p = {\rm M}_d (\mZ_p)\times {\rm M}_d (\mZ_p)^{\rm opp}\spa .
\end{equation*}
Then $\mathcal O_B$ is stable under $\iota$.

We define the field $E$, which turns out to be the reflex field of our Shimura variety, to be equal to $F$ when $d>2$ and equal to $\mQ$ when $d=2$.  Further let $R=\mathcal O_{E_\fp}$ when $d>2$, where we fix a choice $\fp$ of one of the two prime factors of $p$,   and
$R = \mZ_p$ when $d = 2$. In either case $R = \mZ_p$. Let $S$ be a scheme over $\Spec R$. We will consider the
category $AV_{\mathcal O_B}$ of abelian schemes $A$ over $S$ with a homomorphism
\begin{equation*}
i:\mathcal O_B\to\End_S (A)\otimes\mZ_{(p)}\spa .
\end{equation*}
The homomorphisms between two objects $(A_1, i_1)$ and $(A_2, i_2)$ are defined by
\begin{equation*}
\Hom ((A_1, i_1), (A_2, i_2)) = \Hom_{\mathcal O_B}(A_1, A_2)\otimes\mZ_{(p)}\spa
\end{equation*}
(elements in $\Hom_S (A_1, A_2)\otimes\mZ_{(p)}$ which commute with the $\mathcal O_B$-actions). If $(A, i)$ is an object
of this category, then the dual abelian scheme $\hat{A}$, with $\mathcal O_B$-multiplication 
$\hat{i}: \mathcal O_B\to\End (\hat{A})\otimes\mZ_{(p)}$ given by $\hat{i} (b) = (i(b^\iota))^\wedge$, is again an object
of this category. A {\it polarization} of an object $(A, i)$ is a homomorphism $\lambda : A\to\hat{A}$ in $AV_{\mathcal O_B}$
such that $n\lambda$ is induced by an ample line bundle on $A$, for a suitable natural number $n$, and a polarization $\lambda$ is
called {\it principal}  if $\lambda$ is an isomorphism. A $\mQ${\it -class of polarizations} is a set $\overline{\lambda}$ of
quasi-isogenies $(A, i)\to (\hat{A}, \hat{i})$ such that locally on $S$ two elements of $\overline{\lambda}$ differ
by a factor in $\mQ^\times$ and such that there exists a polarization in $\overline{\lambda}$.

\subsection{Definition of $\A_0$} Let $K^p$ be a sufficiently small open compact subgroup of $G(\mA_f^p)$. Since $K^p$ will be fixed throughout the discussion, we will mostly omit it from the notation. 

\begin{defn}\label{defA_0}
Let $\A_{0,K^p} = \A_0$ be the set-valued functor on $(\Sch/R)$ whose values on $S$ are given by the following 
objects up to isomorphism.
\end{defn}

1.) A commutative diagram of morphisms in the category  $AV_{\mathcal O_B}$
$$\xymatrix{
A_0\ar[d]^{\lambda_0}\ar[r]^\alpha & A_1\ar[d]^{\lambda_1}\ar[r]^\alpha & \ldots\ar[r]^\alpha & A_{d-1}\ar[d]^{\lambda_{d-1}}\ar[r]^\alpha & A_0\ar[d]^{\lambda_0}\\
\hat{A}_0\ar[r]^{\hat\alpha} & \hat{A}_{d-1}\ar[r]^{\hat\alpha} & \ldots\ar[r]^{\hat\alpha} & \hat{A}_1\ar[r]^{\hat\alpha} & \hat{A}_0\spa .
}$$
Here each $\alpha$ is an isogeny of degree $p^{2d}$ and $\alpha^d = p\cdot\id_{A_0}$. Furthermore $\lambda_0$ is a $\mQ$-class of polarizations containing a principal polarization, and the collection 
$\lambda_\bullet$ is given up to an overall scalar in $\mQ^\times$. 

2.) A $K^p$-level-structure, i.e., an isomorphism
\begin{equation*}
  \overline{\eta}^p : V\otimes_\mQ\mA_f^p ~ \widetilde{\rightarrow} ~ H_1 (A_0, \mA_f^p) \mod K^p\spa ,
\end{equation*}

which is $\mathcal O_B$-linear and respects the bilinear forms of both sides (induced by $\lambda_0$ and $\psi$
respectively) up to a constant in $(\mA_f^p)^\times$.

We require of each $A_i$ the Kottwitz condition of type $(1,d-1)$ in the form
\begin{equation}\label{kocon}
{\rm char}(i (x)\mid\Lie A_i) = {\rm char} (x)^{1}\cdot\overline{{\rm char} (x)}^{d-1}\ 
,\  x\in\mathcal O_B\spa .
\end{equation}
Here ${\rm char} (x)\in\mathcal O_F\otimes\mZ_{p}[T]$ denotes the reduced characteristic polynomial of $x$ (a 
polynomial of degree $d$, so that $\dim A_i = d^2$ for all $i$).

This functor is representable by a quasi-projective scheme over $\Spec \mZ_p$, see \cite{K92}. If $D$ is a division algebra,
$\A_{0, K^p}$ is a projective scheme over $\mZ_p$.

\subsection{Definition of $\A_1$} Our next aim is to describe a moduli scheme $\A_1= \A_{1, K^p}$ which is equipped with a morphism to 
$\A_{0}$. We shall need the theory of Oort-Tate which we summarize as  follows, cf.\ \cite{GT}, Thm. 6.5.1, cf.\ also \cite{DR}, V.\ 2.4.
\begin{theorem}
{\rm (Oort-Tate)} Let $OT$ be the $\mZ_p$-stack of finite flat group schemes of order $p$.

\smallskip

\noindent {\rm (i)} $OT$ is an Artin stack isomorphic to
\begin{equation*}
[\big(\Spec\mZ_p [X, Y] / (XY- w_p) \big)/ \mG_m]\spa ,
\end{equation*}
where $\mG_m$ acts via $\lambda\cdot (X, Y) = (\lambda^{p-1} X, \lambda^{1-p} Y)$.  Here $w_p$ denotes an explicit element of $p\mZ^\times_p$.  

\smallskip

\noindent {\rm (ii)} The universal group scheme $\G$ over $OT$ is equal
\begin{equation*}
\G = \big[\big({\it Spec}_{OT} \mathcal O [Z] / (Z^p - XZ)\big)/\mG_m\big]\spa ,
\end{equation*}
(where $\mG_m$ acts via $Z\mapsto \lambda Z$), with zero section $Z=0$.

\smallskip

\noindent {\rm (iii)} Cartier duality acts on $OT$ by interchanging $X$ and $Y$. 
\qed
\end{theorem}
We denote by $\G^\times$ the closed subscheme of $\G$ defined by the ideal generated by $Z^{p-1}-X$. The morphism
$\G^\times\to OT$ is relatively representable and finite and flat of degree $p-1$.  Note that after base change to
the generic fiber $OT\times_{\Spec \mZ_p}{\Spec \mQ_p}$, the group scheme $\G$ is \'etale and the points of $\G^\times$ are the
generators of $\G$. Therefore $\G^\times$ is called {\it the scheme of generators} of $\G$, cf. \cite{KM} and Remark \ref{OT=KM_rem} below. 
The group
$\mF_p^\times = {\mathbb \mu}_{p-1}$ acts on $\G^\times$ via $Z\mto\zeta Z$.  The restriction of $\G^\times$ to  $OT\times_{\Spec \mZ_p}{\Spec \mQ_p}$ is a p.h.s. under $\mF_p^\times$.

If $G$ is a finite flat group scheme over a $\mZ_p$-scheme $S$, it corresponds to a morphism
$\varphi: S\to OT$ such that $G = \varphi^\ast (\G)$. The {\it subscheme of generators} of $G$ is by definition the
closed subscheme $G^\times =\varphi^\ast (\G^\times)$. Note that if $S=\Spec \bar\mF_p$ and $G$ is infinitesimal, the zero section is a generator, i.e., is a section of $G^\times$. 

\begin{Remark} \label{OT=KM_rem}
Let $c \in G(S)$.  We claim that $c \in G^\times(S)$ if and only if $c$ is an element of exact order $p$, in the sense of Katz-Mazur \cite{KM}, (1.4).  (By loc.~cit.~Theorem 1.10.1, this condition on $c$ may be described in terms of Cartier divisors, or alternatively by requiring that the multiples of $c$ (as in (\ref{[m]c}) below) form a ``full set of sections'' in the sense of loc.~cit.~(1.8.2).)   We are grateful to Kottwitz for pointing out how this may be easily extracted from \cite{OT} and \cite{KM}, as follows.

Write $0$ for the zero section.  We need to show that $c \in G^\times(S)$ if and only if
\begin{equation} \label{[m]c}
\{0, c, [2]c, \cdots, [p-1]c\}
\end{equation}
is a full set of sections of $G/S$ in the sense of \cite{KM}, (1.8.2).  By localizing we reduce to the affine case, so that $S = {\rm Spec}\!\  R$ for a $\mZ_p$-algebra $R$, $G = {\rm Spec}\!\  R[Z]/(Z^p - aZ) $ for some $a \in pR$, and 
$c$ is an element of $R$ satisfying $c^p -ac = 0$.  By \cite{OT}, p.~14, (\ref{[m]c}) coincides with the set
\begin{equation} \label{xic}
\{0,c,\xi c, \dots, \xi^{p-2}c \}
\end{equation}
for $\xi \in \mathbb Z^\times_p$ an element of order $p-1$.  By Lemma 1.10.2 of \cite{KM}, (\ref{xic}) is a full set of sections for $G/S$ if and only if we have the equality of polynomials in $R[Z]$
$$
Z^p -aZ = Z \prod_{i=0}^{p-2} (Z-\xi^i c).
$$
Clearly this happens if and only if $a = c^{p-1}$, that is, if and only if $c \in G^\times(S)$.
\end{Remark}

We now return to our moduli problem $\A_0$. Let  $(A_\bullet, \iota_\bullet, \alpha_\bullet, \lambda_\bullet, \overline{\eta}^p)$ be an $S$-valued point of $\A_0$.
Let $A_i (p^\infty)$ be the $p$-divisible group over $S$ attached to $A_i$. Then $\mathcal O_B\otimes\mZ_p$ acts
on $A_i (p^\infty)$ and the decomposition $\mathcal O_B\otimes\mZ_p = \mathcal O_{B_\fp}\times \mathcal O_{B_{\overline{\fp}}}$
induces a decomposition
\begin{equation}
A_i (p^\infty) = A_i (\fp^\infty)\times A_i (\overline{\fp}^\infty)\spa .
\end{equation}
Using (\ref{kocon}), we see that the $p$-divisible group $A_i (\fp^\infty)$ is of height $d^2$ and dimension $d$, and the 
$p$-divisible group $A_i (\overline{\fp}^\infty)$ is of height $d^2$ and dimension $d(d-1)$. On either
$p$-divisible group there is an action of ${\rm M}_d (\mZ_p)$. Using now the idempotent $e_{11}\in {\rm M}_d (\mZ_p)$
we may define
\begin{equation}
X_i = e_{11} A_i ({\fp}^\infty)\spa .
\end{equation}
This is a $p$-divisible group of height $d$ and dimension $1$,  and $A_i ({\fp}^\infty) = X_i^d$
(canonically, compatible with the ${\rm M}_d (\mZ_p)$-action on both sides). By the functoriality of this construction,
we obtain a chain of isogenies
\begin{equation}\label{chainpdiv}
X_0\xto{\alpha} X_1\xto{\alpha}\ldots\xto{\alpha} X_{d-1}\xto{\alpha} X_0\spa ,
\end{equation}
each of which is of degree $p$ and such that $\alpha^d = p\cdot\id_{X_0}$.  Henceforth let $X_d := X_0$.
Let
\begin{equation}
G_i = \ker (\alpha : X_{i-1}\to  X_{i})\qquad ,\quad i= 1, \ldots , d\spa .
\end{equation}
Then each $G_i$ is a finite flat group scheme of rank $p$ over $S$, and the collection $(G_1, \ldots , G_{d})$ defines
a morphism into the $d$-fold fiber product, 
\begin{equation}\label{maptoOT}
\varphi: \A_{0}\to OT\times_{\Spec\mZ_p}\ldots\times_{\Spec\mZ_p} OT\spa .
\end{equation}
Now $\A_{1, K^p}=\A_1$ is defined as the 2-fiber product
$$\xymatrix{
\A_{1}\ar[d]_{\pi}\ar[r] & \G^\times\times\ldots\times\G^\times\ar[d]\\
\A_{0}\ar[r] & OT\times\ldots\times OT\spa .
}$$
The morphism $\pi$ is relatively representable and finite and flat of degree $(p-1)^d$. It is an \'etale
Galois covering in the generic fiber with Galois group $(\mF_p^\times)^d$, which ramifies along the special
fiber, cf. Corollary \ref{Galoiscov}.

Using the terminology and notation introduced above, the scheme $\A_{1, K^p}$ classifies objects
$(A_\bullet, \iota_\bullet, \alpha_\bullet, \lambda_\bullet, \overline{\eta}^p)$ of $\A_{0, K^p}$,
together with generators of the group schemes $G_1, \ldots , G_{d}$ (i.e. sections of
$G_1^\times, \ldots , G_{d}^\times$).

\begin{Remark}\label{nonsymm}
A less symmetric, but equivalent,  formulation of the moduli problems $\A_0$ and $\A_{1}$ is as follows. Let us define a moduli problem 
${\A'}_{0, K^p}$ over $(\Sch/\mZ_p)$ that associates to a $\mZ_p$-scheme $S$ the set of isomorphism classes of objects $(A, i)$ of $AV_{\mathcal O_B}$ over
$S$ satisfying the Kottwitz condition of type $(1, d-1)$ as in (\ref{kocon}), with a $\mQ$-class of polarizations $\overline{\lambda}$ containing a principal polarization, and equipped with a $K^p$-level structure $\bar\eta^p$. Let $X$ be the unique $p$-divisible group over $S$ with $A ({\fp}^\infty) = X^d$ compatibly with the action
of ${\rm M}_d (\mZ_p)$ on both sides. In addition to $(A, i, \overline{\lambda}, \bar\eta^p )$ we are given a filtration
of finite flat group schemes
\begin{equation*}
(0) = H_{0}\subset H_1\subset\cdots\subset H_{d-1}\subset H_{d} = X [p]\spa ,
\end{equation*}
such that $G_i=H_i / H_{i-1}$ is a finite flat group scheme  of order $p$ over $S$, for $i=1, \ldots , d$. Similarly, we can define a moduli problem 
${\A'}_{1, K^p}$, where  one is given in addition a section of $G_i^\times (S)$, for $i=1, \ldots , d$.

The moduli problems $\A'_0$ and $\A_0$, resp.\ $\A'_1$ and $\A_1$, are equivalent. Indeed, it is clear that an object of $\A_{0, K^p}$ defines an object of ${\A'}_{0, K^p}$ (put $H_i=\ker \alpha^{i}: X_0\to X_{i}$). Conversely, starting
from an object of ${\A'}_{0, K^p}$, we obtain as follows an object of $\A_{0, K^p}$. Let 
$\lambda_0\in\overline{\lambda}$ be a principal polarization. Let $X_0''=X^d=A_0({\fp}^\infty)$. The flag of finite flat subgroups of
${X_0''} = A_0 ({\fp}^\infty)$,
\begin{equation*}
(0)= H_0^d\subset H_1^d\subset\cdots\subset H_{d-1}^d\subset X_0'' [p]
\end{equation*}
defines a sequence of isogenies of $p$-divisible groups
\begin{equation}\label{chainpd}
{X_0''}\to {X_1''}\to\ldots\to {X_{d-1}''}\to {X_0''}
\end{equation}
with composition equal to $p\cdot\id$. Furthermore,  $\lambda_0$ defines an isomorphism 
$A_0 ({\fp}^\infty)\simeq A_0 (\overline{\fp}^\infty)^\wedge$. Now one checks that there exists exactly one
commutative diagram of isogenies in $AV_{\mathcal O_B}$ as in Definition \ref{defA_0}, 1.) which induces the chain of isogenies (\ref{chainpd})
on the ${\fp}^\infty$-part of $A_\bullet (p^\infty)$ and such that $\lambda_0$ induces the previous isomorphism between $A_0 ({\fp}^\infty)$ and  $A_0 (\overline{\fp}^\infty)^\wedge$. This shows how to construct an object of $\A_0$ from an object of $\A_0'$, and it is easy to see that this implies that $\A_{0, K^p}$ is equivalent to $\A_{0, K^p}'$. 

The equivalence of $\A_{1, K^p}$ and ${\A'}_{1, K^p}$ is then obvious.
\end{Remark}

\subsection{The structure of $\A_0$ and $\A_1$}\label{local_struc}
 Let
\begin{equation*}
\Lie X_0\to\Lie X_1\to\ldots\to\Lie X_{d-1}\to\Lie X_0
\end{equation*}
be the chain of homomorphisms of invertible sheaves on  ${\A}_0\otimes \mF_p$,  given by the Lie algebras of the universal $p$-divisible
groups $X_i$ over $\A_0\otimes\mF_p$. For $i = 1, \ldots, d$, let $\overline{\A_{0, i}}$ be the vanishing locus of $\Lie X_{i-1}\to\Lie X_{i}$.  More precisely, $\overline{\A_{0,i}}$ is the support of the cokernel of $(\Lie X_{i})^\vee \rightarrow (\Lie X_{i-1})^\vee$.  From the theory of local models we have the following theorem, cf.\ \cite{G1}, Prop.~4.13. 
\begin{theorem}\label{dnc}
The scheme $\A_0$ has strict semi-stable reduction over $\Spec\mZ_p$. More precisely,  the $\overline{\A_{0, i}}$ for $i=1,\ldots , d$ are smooth
divisors in $\A_0$, crossing normally, and ${\A}_0\otimes \mF_p = \bigcup\nolimits_i\overline{\A_{0, i}}$.\hfill\qed
\end{theorem}

As usual, the divisors $\overline{\A_{0,i}}$ define a stratification of $\A_0\otimes \mF_p$.  For $x \in\A_0 (\bar{\mF}_p)$ let
\begin{equation}
S (x) = \{i\in\mZ / d\mid x\in\overline{\A_{0, i}}\}\spa .
\end{equation}
Then $S (x)$ is a non-empty subset of $\mZ / d :=\{1, \ldots, d\}$, called the {\it set of critical indices} of $x$. In terms of the $d$-tuple
$(G_1, \ldots , G_{d})$ of finite group schemes of order $p$ attached to $x$,
\begin{equation*}
S (x) = \{i\in\mZ / d\mid G_i\spa\text{infinitesimal}\spa\}\spa .
\end{equation*}
To any non-empty subset $S$ of $\mZ / d$ we can associate the locally closed reduced subscheme $\A_{0, S}$
of ${\A}_0\otimes \mF_p$, with $\A_{0, S} (\bar{\mF}_p)$ equal to the set of $\bar{\mF}_p$-points with
critical set equal to $S$. From Theorem \ref{dnc} we deduce
\begin{equation}\label{clos}
{\A}_0\otimes\mF_p =\bigcup\nolimits_{S\neq \emptyset}\A_{0, S}\ , \quad 
\overline{\A_{0, S}} = \bigcup\nolimits_{S'\supseteq S}\A_{0, S'}\spa .
\end{equation}
Theorem \ref{dnc} also implies that $\A_{0, S}$ and its closure $\overline{\A_{0, S}}$ are smooth of dimension $d-|S|$.

In order to investigate the structure of $\A_1$, we consider for $i=1,\ldots , d$, the morphism
$\varphi_i : \A_0\to OT$ defined by associating to an $S$-valued point of $\A_0$ the finite flat group scheme
$G_i$ of order $p$, comp. (\ref{maptoOT}).
\begin{lemma} \label{local_param}
Let $x\in\A_{0, S}(\bar{\mF}_p)$.  Then there is an \'{e}tale neighborhood of $x$ in $\A_{0,S}$ which carries an \'{e}tale morphism to an affine scheme $\Spec 
\mathcal O$, where
\begin{equation*}
{\mathcal O}\simeq W (\bar{\mF}_p)[T_1,\ldots , T_{d}][T^{-1}_j, j \notin S] / (\prod\nolimits_{j\in S} T_j - w_p)\ ,\end{equation*}
such that the traces of the divisors $\overline{\A_{0,i}}$ on $\Spec \mathcal O$ are given by $T_i = 0$, 
and such that the morphisms $\varphi_i$ to
$OT = [\big(\Spec\mZ_p[X, Y] / (XY - w_p)\big) / \mG_m]$ are given by
\begin{align} \label{varphi(X)}
\varphi_i^\ast (X) &= u_i T_i\ , \quad\varphi_i^\ast (Y) = u_i^{-1} T_i^{-1} \prod\nolimits_{j\in S} T_j\spa\ , 
\end{align}
where $u_i$ is a unit in $\mathcal O$.
\end{lemma}

Here, if $i \in S$, we understand $T_i^{-1}\prod\nolimits_{j \in S} T_j$ to be $\prod\nolimits_{j \in S, j \neq i} T_j$, with the convention that the empty product is 1. 

\begin{proof}
It is clear from Theorem \ref{dnc} that the ring $\mathcal O$ and the elements $T_1, \ldots, T_d$ can be chosen with the first two properties.  
It remains to show that for each $i$, there exists a unit $u_i\in \mathcal O^\times$ such that (\ref{varphi(X)}) holds.  The group scheme $G_i$ over $\Spec \mathcal O$ is obtained by pullback under $\varphi_i$ of the universal
group scheme $\G$ over $OT$, hence
\begin{equation*}
G_i = \Spec \mathcal O [Z] / (Z^p - \varphi_i^\ast (X)\cdot Z)\spa .
\end{equation*}
The reduced locus in $\Spec \mathcal O$ over which $G_i$ is infinitesimal is defined by the equation
$\varphi_i^\ast (X) = 0$. On the other hand, it is also given by $T_i = 0$. Hence
\begin{equation*}
\varphi_i^\ast (X) = u_i T_i^m\spa ,
\end{equation*}
for some $m\geq 1$, where $u_i\in \mathcal O^\times$. Let $g_i = \varphi_i^\ast (Y)$. Then
\begin{equation*}
u_i T_i^m g_i = w_p = \prod\nolimits_{j \in S} T_j\spa.
\end{equation*}
For $i \notin S$, the elements $\varphi^\ast_i(X)$ and $T_i$ are units, and so by replacing $u_i$ with $u_iT_i^{m-1}$ we may arrange to have $\varphi^\ast_i(X) = u_i T_i$; then (\ref{varphi(X)}) is satisfied.  

Next assume $i \in S$.  
Since $\mathcal O$ is a $UFD$ and the elements $T_j$ ($j \in S$) are pairwise distinct prime elements, it follows that 
$m=1$ and again (\ref{varphi(X)}) holds.
\end{proof}

Let $\A_{1, S}=\pi^{-1}(\A_{0, S})$. Since $\pi$ is proper, this defines a stratification of $ \A_1\otimes\mF_p$ with similar properties as the analogous stratification of $\A_0\otimes\mF_p$. 
\begin{cor}\label{locnormfor}
The scheme $\A_1$ is regular and its special fiber is a divisor with normal crossings, with all branches
of multiplicity $p-1$. More precisely, let $x \in\A_{0, S} (\bar{\mF}_p)$.  Then the covering 
$\pi : \A_1\to\A_0$ in an \'{e}tale local neighborhood of $x$ is isomorphic to the morphism 
\begin{equation*}
\Spec W (\bar{\mF}_p)[T'_1, \ldots , T'_{d}] / \Big((\prod\nolimits_{j\in S} T'_j)^{p-1} - w_p\Big)\to\Spec W (\bar{\mF}_p)[T_1, \ldots , T_{d}]/\Big(\prod\nolimits_{j\in S} T_j - w_p\Big)
\end{equation*}
defined via
\begin{equation*}
T_i\mto T_i'^{p-1}\ , \  i=1, \ldots , d\spa .
\end{equation*}
\end{cor}
\begin{proof}
This follows from Lemma \ref{local_param}, since over $\Spec \mathcal O$, the covering $\A_1$ is  defined by extracting the $(p-1)$-st root of $\varphi_i^\ast (X)$ for $i=1, \dots, d$.  
 
Implicit in our assertion is the claim that the units $u_i$ play no essential role.  We will prove a more precise version of the corollary, which also justifies this.  

The first step is to pass to a smaller \'{e}tale neighborhood of $x$: for each $i$ choose a $v_i \in W(\bar{\mF}_p)[T_1, \ldots, T_d][T_j^{-1}, j \notin S]$ such that its image in $\mathcal O$ is $u_i^{-1}$.  Replace $\Spec \mathcal O$ with a smaller \'{e}tale affine neighorhood of $x$, by replacing the former ring $\mathcal O$ with the ring
$$
\widetilde{\mathcal O} := W(\bar{\mF}_p)[T_i, T^{-1}_j, V_i^{\pm 1}]/\Big(\prod\nolimits_{j \in S} T_j - w_p \spa , \spa V^{p-1}_i - v_i \Big).
$$
Here the index $i$ for $T_i, V_i$ and $v_i$ ranges over $1, \ldots, d$ and the index $j$ for $T_j^{-1}$ ranges over the complement of $S$.

The pull-back to $\Spec \widetilde{\mathcal O}$ of $\pi: \A_1 \rightarrow \A_0$ is given by adjoining to $\widetilde{\mathcal O}$ the indeterminates $t_i$, $i = 1, \ldots, d$, which are subject to the relations $t^{p-1}_i = v^{-1}_i T_i$.  Set $T'_i = V_i t_i$.   Note that $T_i'^{p-1} = T_i$.  It follows that the pull-back of $\pi$ to $\Spec \widetilde{\mathcal O}$ is the morphism taking
\begin{equation*}
\Spec W (\bar{\mF}_p)[T'_i, T'^{-1}_j, V^{\pm 1}_i] / \Big((\prod\nolimits_{j\in S} T'_j)^{p-1} - w_p \spa , \spa V^{p-1}_i - v_i\Big)
\end{equation*}
to
\begin{equation*}
\Spec W (\bar{\mF}_p)[T_i, T^{-1}_j, V^{\pm 1}_i]/\Big(\prod\nolimits_{j\in S} T_j - w_p \spa , \spa V_i^{p-1} - v_i \Big),
\end{equation*}
via $T_i \mapsto T_i'^{p-1}$.    
\end{proof}

To formulate the next result, we introduce the following tori over $\mF_p$:
\begin{equation*}
T = {\prod}_{i\in\mZ / d}\mG_m\quad , \quad T_S = {\prod}_{i\in S}\mG_m\quad , \quad T^S = \prod\nolimits_{i\notin S}\mG_m\spa .
\end{equation*}
Here we consider $T_S$ as a subtorus of $T$ and $T^S$ as $T / T_S$.
\begin{cor}\label{Galoiscov}
The morphism $\pi : \A_1\to\A_0$ is a connected\footnote{Here and throughout this article, we call a finite \'etale map $f: Y \rightarrow X$ an \'etale Galois covering if for every closed point $x \in X$, we have $\#{\rm Aut}(Y/X) = \#f^{-1}(x)$.  We call such a morphism {\em connected} if the inverse image of any connected component of $X$ is connected.} Galois covering in the generic fiber, with group of deck transformations
$T (\mF_p)$. The restriction $\pi_S$ of $\pi$ to $(\A_{1, S})_{\rm red}$ is an \'etale Galois covering over $\A_{0, S}$, with group
of deck transformations $T^S (\mF_p)$.
\end{cor}
\begin{proof}
The statement about $\pi_S$ follows from the description of the morphism $\pi$ in the previous corollary.  Using that every $G_i$ is \'etale in characteristic zero, a similar result describes $\pi$ in the generic fiber and 
implies the remaining assertions.
\end{proof}

Later on in section \ref{monodromy_sec}, we shall prove that $\pi_S$ is in a fact a connected \'etale Galois covering as well.

\subsection{The KR-stratification of $\A_0\otimes\mF_p$}\label{subKR} We now compare the stratification of
$\A_0\otimes\mF_p$ by critical index sets with the KR{\it -stratification}. Recall (\cite{GN}, \cite{H05}) that this stratification
is defined in terms of the {\em local model diagram}, and that its strata are parametrized by the {\em admissible subset}
${\rm Adm} (\mu)$ of the extended affine Weyl group $\widetilde{W}$ of $G_{\mQ_p}$ (see \S\ref{comb_sec} below). Here,
${\mu} = (\mu_0, 1)\in X_\ast (\GL_d)\times X_\ast (\mG_m)$ and $\mu_0 = (1, 0, \ldots , 0)$ is the fixed
minuscule coweight of Drinfeld type. In fact, we may obviously identify 
${\rm Adm} (\mu)$ with the admissible set ${\rm Adm} (\mu_0)$ for $\GL_d$. 

We denote by $\tau$ the unique element of length zero in ${\rm Adm} (\mu_0)$. Also, let ${\bf a}$ be the fundamental alcove. 
\begin{prop}\label{critad}
\noindent {\rm (i)} Let $w\in{\rm Adm} (\mu_0)$. Let
\begin{equation*}
S (w) = \{i\in\{1, \ldots , d\}\mid w{\bf a}\spa\text{and}\spa \tau{\bf a}\spa\text{share a vertex of type}\spa i\}\spa 
\end{equation*}
cf. subsection \ref{crit_def} and Lemma \ref{w_leq_e_j}.  (Here ``type $d$'' corresponds 
to ``type $0$''.)  Then $S(w)\neq\emptyset$ and the association $w\mto S(w)$ induces a bijection between ${\rm Adm} (\mu_0)$ and the
set of non-empty subsets of $\mZ / d$.

\smallskip

\noindent {\rm (ii)} For the Bruhat order on ${\rm Adm} (\mu_0)$,
\begin{equation*}
w\leq w'\Llrto S(w)\supseteq S(w')\spa .
\end{equation*}
In particular, $w=\tau$ if and only if $S (w) = \mZ / d$.
\end{prop}
\begin{proof}
The non-emptiness of $S(w)$ follows from Lemma \ref{w_leq_e_j}.  The bijectivity of $w \mapsto S(w)$ follows by combining Lemma \ref{general_form} with Proposition \ref{main_comb_prop}, (a).  This proves (i).

Part (ii) is Proposition \ref{main_comb_prop}, (d).   
\end{proof}
Suppose now that $x\in \A_0(\bar{\mF}_p)$ belongs to the KR-stratum $\A_{0,w}$ indexed by $w \in {\rm Adm}(\mu_0)$. The chain of isogenies of one-dimensional $p$-divisible groups (\ref{chainpdiv}) defines a chain of homomorphisms of the corresponding (covariant) Dieudonn\'e modules, 
\begin{equation}\label{dieuchain}
M_0\to M_1\to\cdots\to M_{d-1}\ .
\end{equation} This may be viewed as a chain of inclusions of lattices in the common rational Dieudonn\'e module.  Let ${\rm inv}(M_\bullet, VM_\bullet) \in \widetilde{W}$ denote the relative position of the lattice chains $M_\bullet$ and $VM_\bullet$ defined by (\ref{dieuchain}).  By definition, $x \in \mathcal A_{0,w}$ if and only if 
\begin{equation} \label{M:VM}
{\rm inv}(M_\bullet, VM_\bullet) = w
\end{equation}
(cf. \cite{H05}, $\S8.1$).   Now, $i$ is a critical index for the point $x$ if and only if ${\rm Lie}(A_{i-1}) \rightarrow {\rm Lie}(A_i)$ vanishes.  Recall that $M_i/VM_i = {\rm Lie}(A_i)$.  Thus $i$ is a critical index if and only if the lattice $M_{i-1}$ coincides with $VM_{i}$.  In view of (\ref{M:VM}), this holds if and only if $\tau{\bf a}$ and $w{\bf a}$ share a vertex of type $i$.  Thus $S(x) = S(w)$.  Hence (i) of the previous proposition  proves the following result.
\begin{cor} The KR-stratification $\{\A_{0, w}\}_{w\in{\rm Adm} (\mu_0)}$ may be identified with the critical index stratification $\{\A_{0, S}\}_{S\subset\mZ / d}$, via $\A_{0, w}=\A_{0, S(w)}$. 
\end{cor}
\begin{Remark}
The KR-stratification has the property that $\A_{0,w}\subset \overline{\A_{0,w'}}$ if and only if $w\leq w'$ in the Bruhat order. Hence, the item (ii) in Proposition \ref{critad} follows from the above corollary and the closure relation (\ref{clos}) for the critical index stratification. But the proof of (ii) in
\S\ref{comb_sec} is purely combinatorial.
\end{Remark}
In the next section we will need the following refinement of the closure relation (\ref{clos}).
\begin{prop}\label{refclos}
 Let $S\subset S'$ be non-empty subsets of $\mZ/d$. Every connected component of $\A_{0, S}$ meets $\A_{0, S'}$ in its closure. 
\end{prop}
\begin{proof} Since the closures of the strata $\A_{0, S}$ are  smooth, to prove the claim for $S'=\mZ/d$, we may proceed by induction and show that, if $S\neq \mZ/d$, then no connected component of $\A_{0, S}$ is closed, cf.~\cite{GY}, proof of Thm.~6.4. The case of general $S'$ follows then from Theorem \ref{dnc}, since the closure of a connected component of $\A_{0, S}$ is a connected component of the closure of $\A_{0, S}$. 

 Consider the morphism $\rho: \A_0\otimes\mF_p\to \A\otimes\mF_p$ into the absolute moduli scheme. In terms of the non-symmetric formulation of the moduli problem in Remark \ref{nonsymm}, $\rho$ is forgetting the filtration $H_\bullet$, and keeping $(A, i, \lambda, \bar\eta^p)$. Now $\A\otimes\mF_p$ has a stratification by the $p$-rank of the $p$-divisible group $A({\frak p}^\infty)$ of the universal abelian scheme $A$, cf. \cite{Ito}:  a stratification by smooth locally closed subschemes 
$\A^{(i)}$ of pure dimension $i$, for $i=0,\ldots , d-1$ such that the closure $\A^{[i]}$ of $\A^{(i)}$ is the union of the strata $\A^{(j)}$, for $j= 0,\ldots , i$. Furthermore, $\rho(\A_{0, S})=\A^{(i)}$, where $i=d-\vert S\vert$, and $\rho^{-1}(\A^{(i)})=\bigcup\nolimits_{\vert S\vert=d-i}\A_{0, S}$. Since the morphism $\rho$ is finite and flat, the image of a connected component of $\A_{0, S}$ is a connected component of $\A^{(i)}$ for $i=d-\vert S\vert$, and it is closed if and only if its image under $\rho$ is closed. We are therefore reduced to proving, for $i\geq 1$, that no connected component of $\A^{(i)}$ is closed. 

Now, by Ito \cite{Ito},  there exists an ample line bundle $\mathcal L$ on $\A\otimes \mF_p$ such that $\A^{[i-1]}$ is defined in $\A^{[i]}$ by the vanishing of a natural global section of $\mathcal L^{\otimes(p^{d-i}-1)}$ over $\A^{[i]}$, for any $i$ with $1\leq i\leq d-1$ ({\em higher Hasse invariant}). It follows that if   $\A\otimes \mF_p$  is projective, then for any $i\geq 1$, the closure of any connected component of $\A^{(i)}$ meets  $\A^{(i-1)}$ in a divisor, which proves the claim in this case.

When  $\A\otimes \mF_p$  is not projective, we are going to use the minimal compactification $(\A\otimes \mF_p)^*$ of $\A\otimes \mF_p$, in the sense of Lan \cite{Lan}. By \cite{Lan}, Thm.~7.2.4.1, this adds a finite number of points to $\A\otimes \mF_p$ (since this is the case for the minimal compactification of $\A\otimes \mC$). Furthermore, the line bundle $\mathcal L$ is the restriction of an ample line bundle on $(\A\otimes \mF_p)^*$ and the section of $\mathcal L^{\otimes(p^{d-i}-1)}$ on $\A^{[i]}$ above extends over the closure of $\A^{[i]}$ in $(\A\otimes \mF_p)^*$. Hence the previous argument goes through for any connected component of $\A^{(i)}$, as long as $i\geq 2$. Furthermore, since the case $d=2$ is classical, we may assume $d>2$. Then the case $i=1$ will follow if we can show that $\A^{[1]}$ is closed in $(\A\otimes \mF_p)^*$. In fact, we will show that $\A^{[i]}$ is complete for any $i\leq d-2$. For this, consider an object $(A, i, \lambda, \bar\eta^p)$
of $\A\otimes\mF_p$ over the fraction field of a complete discrete valuation ring with algebraically closed residue field, such that $A$ has semi-stable reduction. We need to show that if $(A, i, \lambda, \bar\eta^p)$ lies in $\A^{[d-2]}$, then $A$ has good reduction. Assume not. Now the special fiber of its 
N\'eron model inherits an action of $\mathcal O_B$, and its neutral component is an extension of an abelian variety $A_0$ by a torus $T_0$, both equipped with an action of $\mathcal O_B$. In fact,  the Kottwitz signature condition of type $(1, d-1)$ for $A$  implies that $A_0$ satisfies the Kottwitz signature condition of type $(0, d-2)$. It follows that $\dim T_0=2d$, and that $A_0$ is isogenous to the $d(d-2)$-power of an elliptic curve with CM by $F$. Now the $p$-torsion of the $p$-divisible group $T_0({\frak p}^\infty)$ has \'etale rank $d$, and the $p$-torsion of the $p$-divisible group $A_0({\frak p}^\infty)$ has \'etale rank $d(d-2)$. By the infinitesimal lifting property, this \'etale group scheme lifts to a subgroup of $A({\frak p}^\infty)$ of $p$-rank $d(d-1)$. But this implies that the $p$-rank of $A({\frak p}^\infty)$ is too large for $(A, i, \lambda, \bar\eta^p)$ to be contained in $\A^{[d-2]}$, and the assertion is proved. 

\end{proof}

\section{The monodromy theorem} \label{monodromy_sec}

In this section we modify our notation: we will denote by $\A_0, \A_1,$ etc. the result of the
base change $\times_{\Spec\mZ_p}{\Spec\bar\mF_p}$ of the schemes in previous sections.

\subsection{Statement of the result.} Our aim here is to prove the following theorem.
\begin{thm}\label{monodr}
Let $S\subset\mZ / d\spa ,\spa S\neq\emptyset$.
\smallskip

\noindent {\rm (i)} The lisse $\mF_p$-sheaf $\underset{i\not\in S}{\prod}\G_i$ on $\A_{0, S}$ defines
for any point $x\in\A_{0, S} (\bar\mF_p)$ a homomorphism
$$
\pi_1 (\A_{0, S}, x)\to T^S (\mF_p)\spa .
$$
This homomorphism is surjective.
\smallskip

\noindent {\rm (ii)} Consider the map of sets of connected components, 
\begin{equation*}
\pi_0 (\A_{1, S})\to\pi_0 (\A_{0, S}) ,
\end{equation*}
 defined by the  \'etale
Galois covering $\pi_S: (\A_{1, S})_{\rm red}\to\A_{0, S}$ induced by $\pi$, cf.\ Corollary \ref{Galoiscov}. This map is bijective. Equivalently, the inverse image under $\pi$ of any connected component
of $\A_{0, S}$ is connected.
\end{thm}
Note that, since $T^S (\mF_p)$ is the group of deck transformations  for the covering
$(\A_{1, S})_{\rm red}\to \A_{0, S}$, these two statements are trivially equivalent. We
will prove the theorem in the form (i).

We will in fact prove something stronger. Let $y\in\A_{0, \mZ/d} (\bar\mF_p)$
be a point in the closure of the connected component of $\A_{0, S}$ containing $x$, cf. Proposition \ref{refclos}.
Let $X_0$ be the formal group defined by $y$, and let 
$R\simeq W(\bar\mF_p)[[U_1, \ldots , U_{d-1}]]$ be its universal
deformation ring. We obtain a chain of finite ring morphisms
\begin{equation*}
R\subset R^0\subset R^1\subset\widetilde{R}\spa .
\end{equation*}
Here $R^0$, resp.\  $R^1$, resp.\ $\widetilde{R}$ represent the following functors over $\Spf R$:
\begin{equation*}
\begin{aligned}
R^0 =& \spa\text{flag}\spa H_\bullet \spa\text{of}\spa X_0 [p]\\
R^1 =& \spa\text{flag}\spa H_\bullet \spa\text{and generators of}\spa G_i = H_i / H_{i-1}, i = 1, \ldots , d\\
\widetilde{R}\, =& \spa\text{Drinfeld level structure}\spa \varphi : \mF_p^d\to X_0 [p]\spa .
\end{aligned}
\end{equation*}
  Note that $G_i$ is a subgroup scheme of the formal group $X_i=X_0/H_{i-1}$ of dimension 1.  Hence the notion of  Drinfeld or Katz/Mazur of a generator of $G_i$ makes sense; by Remark \ref{OT=KM_rem} it coincides with  the Oort-Tate  notion.  

Note that the injection $R^0\subset \widetilde{R}$ is induced by the functor morphism which
associates to a Drinfeld level structure $\varphi$ the chain $H_\bullet = H_\bullet (\varphi)$, where $H_i$
is defined by the equality of divisors
\begin{equation*}
[H_i] = \sum_{x\in{\rm span}(e_1, \ldots , e_i)}{[\varphi (x)]}\spa .
\end{equation*}
Similarly, $R^1\subset\widetilde{R}$ is induced by the functor
\begin{equation*}
(X_0, \varphi)\mto (X_0, H_\bullet (\varphi), \spa\text{generator}\spa\varphi (e_i)\spa\text{of}\spa G_i = H_i / H_{i-1}, i = 1, \ldots , d)\spa .
\end{equation*}
Here $e_1, \ldots , e_d\in\mF_p^d$ are the standard generators.

Let $\wp_S^0\subset R^0$ be the ideal defined by $\overline{\A_{0, S}}$.
Let $\kappa_S^0 = {\rm Frac} (R^0/\wp_S^0)$. Then, denoting by $\eta_x$ the generic point of the
connected component of $\A_{0, S}$ containing $x$, we obtain a commutative 
diagram
\begin{equation}\label{oblique}
\xymatrix{
\Gal (\kappa_S^0)\ar[r]\ar[drr] & \Gal (\eta_x)\ar[r] & \pi_1 (\A_{0, S}, x)\ar[d]\\
{\quad} & {\quad} & T^S (\mF_p)\spa . }
\end{equation}
We will prove that the oblique arrow is surjective, which implies assertion (i). 

\subsection{A result of Strauch.}We are going to use the following result of Strauch \cite{Str}.  We consider,  for a fixed integer $h$ with $1\leq h\leq d$, the following two subgroups of $\GL_d (\mF_p)$:
\begin{equation*}
\begin{aligned}
P_h &=\Stab ({\rm span} (e_1, \ldots , e_h))\\
Q_h &=\text{subgroup of}\spa P_h \spa\text{of elements acting trivially on}\spa \mF_p^d / {\rm span} (e_1, \ldots , e_h)\spa .
\end{aligned}
\end{equation*}
Recall (\cite{Dr}) that $\widetilde{R}$ is a regular local ring, with system of regular
parameters $\varphi (e_1), \ldots , \varphi (e_d)$ (which we regard as elements of
the maximal ideal of $\widetilde{R}$ after choosing a formal parameter of the formal group $X_0$).
Let
\begin{equation}
\begin{aligned}
\widetilde{\wp}_h &= \spa\text{prime ideal generated by}\spa\varphi (e_1), \ldots , \varphi (e_h)\spa ,\\
\widetilde{R}_h &= \widetilde{R} / \widetilde{\wp}_h\spa , \spa\widetilde{\kappa}_h ={\rm Frac} (\widetilde{R}_h)\spa .
\end{aligned}
\end{equation}
Similarly, let
\begin{equation}
\wp_h = \widetilde{\wp}_h\cap R\spa ,\spa R_h = R / \wp_h\spa ,\spa\kappa_h = {\rm Frac} (R_h)\spa .
\end{equation}
We also set
\begin{equation*}
\widetilde{\kappa} = {\rm Frac} (\widetilde{R})\spa ,\spa\kappa = {\rm Frac} (R)\spa .
\end{equation*}

\begin{prop}\label{Strauch}
{\rm (Strauch):} \noindent {\rm (i)} The extension $\widetilde{\kappa} / \kappa$ is Galois. The action of $\GL_d (\mF_p)$
on the set of Drinfeld bases induces an isomorphism
\begin{equation*}
\GL_d (\mF_p)\simeq \Gal (\widetilde{\kappa} / \kappa)\spa .
\end{equation*}

\smallskip

\noindent {\rm (ii)}  The extension $\widetilde{\kappa}_h / \kappa_h$ is normal. The decomposition
group of $\widetilde{\wp}_h$ is equal to $P_h$ and the inertia group is equal to $Q_h$.
The canonical homomorphism
\begin{equation*}
P_h / Q_h\to\Aut (\widetilde{\kappa}_h / \kappa_h)
\end{equation*}
is an isomorphism. In particular,
\begin{equation*}
\Aut (\widetilde{\kappa}_h / \kappa_h)\simeq\GL_{d-h} (\mF_p)\spa .
\end{equation*}\qed
\end{prop}
We introduce the following subgroups of $P_h$,
\begin{equation*}
\begin{array}{l}
PB_h = \{g\in P_h\mid g\spa\text{preserves the standard flag}\spa\mathcal F_\bullet\spa\text{on}\spa\mF_p^d / {\rm span} (e_1, \ldots , e_h)\}\\
PU_h = \{g\in PB_h\mid g \spa\text{induces the identity on}\spa \gr_\bullet\mathcal F\}\spa .
\end{array}
\end{equation*}
We thus obtain a chain of subgroups of $\GL_d (\mF_p)$, 

\begin{equation*}
\begin{array}{ccccccc}
Q_h & \subset & PU_h & \subset & PB_h & \subset & P_h\\
\\
|| && || && || && ||\\
\\
{\begin{pmatrix}\xymatrix@R=0.3mm@C=0.3mm@M=0.3mm@W=0.3mm{&&&&\ar@{-}[dddddd]&&&&\\
&&\ast &&&&\ast &&\\
\ar@{-}[rrrrrrrr]&&&&&&&&\\
&&&&&1&&0&\\
&&0&&&&\ddots&&\\
&&&&&0&&1&\\
&&&&&&&&}\end{pmatrix}} &&
{\begin{pmatrix}\xymatrix@R=0.3mm@C=0.3mm@M=0.3mm@W=0.3mm{&&&&\ar@{-}[dddddd]&&&&\\
&&\ast &&&&\ast &&\\
\ar@{-}[rrrrrrrr]&&&&&&&&\\
&&&&&1&&\ast &\\
&&0&&&&\ddots&&\\
&&&&&0&&1&\\
&&&&&&&&}\end{pmatrix}} &&
{\begin{pmatrix}\xymatrix@R=0.3mm@C=0.3mm@M=0.3mm@W=0.3mm{&&&&\ar@{-}[dddddd]&&&&\\
&&\ast &&&&\ast &&\\
\ar@{-}[rrrrrrrr]&&&&&&&&\\
&&&&&\ast &&\ast &\\
&&0&&&&\ddots&&\\
&&&&&0&&\ast &\\
&&&&&&&&}\end{pmatrix}} &&
{\begin{pmatrix}\xymatrix@R=0.3mm@C=0.3mm@M=0.3mm@W=0.3mm{&&&&\ar@{-}[dddddd]&&&&\\
&&\ast &&&&\ast &&\\
\ar@{-}[rrrrrrrr]&&&&&&&&\\
&&&&&\ast &&\ast &\\
&&0&&&&\ddots&&\\
&&&&&\ast &&\ast &\\
&&&&&&&&}\end{pmatrix}}
\end{array}
\end{equation*}

\medskip

Let
\begin{equation}
\begin{aligned}
\wp_h^1 &= \widetilde{\wp}_h\cap R^1,\  &R_h^1 = R^1 / \wp_h^1, \ \ \ &\kappa_h^1 = {\rm Frac} (R_h^1)\\
\wp_h^0 &= \widetilde{\wp}_h\cap R^0,\  &R_h^0 = R^0 / \wp_h^0, \ \ \ &\kappa_h^0 = {\rm Frac} (R_h^0)\spa .
\end{aligned}
\end{equation}
\begin{cor} {\rm (i)} $\wp_h^0$ coincides with the prime ideal $\wp_S^0$ introduced earlier, for $S = S_0=\{1, \ldots , h\}$.

\smallskip

\noindent {\rm (ii)} Under the isomorphism in Proposition \ref{Strauch}, (ii), there are the following identifications,
\begin{equation*}
\begin{aligned}
\Aut (\widetilde{\kappa}_h / \kappa_h^1) &= PU_h / Q_h\\
\Aut (\widetilde{\kappa}_h / \kappa_h^0) &= PB_h / Q_h\\
\Aut (\kappa_h^1 / \kappa_h^0) &= PB_h / PU_h\spa .
\end{aligned}
\end{equation*}
\end{cor}
\begin{proof} Obviously $\widetilde{\wp}_h\cap R^0\subset\wp_{S_0}^0$. Since both prime ideals have height
$h$, this inclusion is an equality, which proves (i). The assertion (ii) follows from the description of the isomorphism
in Proposition \ref{Strauch}, (ii).
\end{proof}

\subsection{End of the proof.}Now we can prove the surjectivity of the oblique arrow in (\ref{oblique}) for $S=S_0= \{1, \ldots , h\}$.
Indeed, since $\kappa_h^1 / \kappa_h^0$ is of degree prime to $p$, this extension is separable
so that $\Aut (\kappa_h^1 / \kappa_h^0) = \Gal (\kappa_h^1 / \kappa_h^0)$. Furthermore, we
have a natural identification
\begin{equation*}
PB_h / PU_h\simeq T^S (\mF_p)\spa ,
\end{equation*}
and the induced homomorphism
\begin{equation*}
\Gal (\kappa_h^0)\twoheadrightarrow \Gal (\kappa_h^1 / \kappa_h^0)\to T^S (\mF_p)
\end{equation*}
is the oblique arrow  in (\ref{oblique}). The claim follows in the case $S=S_0 = \{1, \ldots , h\}$. 

Let now
$S$ be arbitrary with $\vert S\vert = h$. Let $g\in \GL_d (\mF_p)$ be a permutation matrix with $g S_0 = \{1, \ldots , h\}$.
Then $^g\wp_h$ lies over $\wp_S^0$, and the decomposition group resp.\ inertia group of $^g\wp_h$ is the
conjugate by $g$ of $P_h$ resp. $Q_h$. Now the surjectivity of the homomorphism 
$\Gal (\kappa_S^0)\to T^S (\mF_p)$ follows from the surjectivity of $\Gal (\kappa_{S_0}^0)\to T^{S_0} (\mF_p)$.\qed

\section{Nearby cycles} \label{nearby_sec}

\subsection{Recollections on $R\Psi$}
Let $(Z, \eta, s)$ be a henselian triple, with $s$ of characteristic $p$. Let $(\bar Z, \bar {\eta}, \bar s)$ be the strict henselian triple deduced from a geometric point $\bar {\eta}$ over $\eta$. Let $X\to Z$ be a morphism of finite type. We obtain a diagram with cartesian squares, where $\bar X=X\times_Z\bar Z$,
\begin{equation}\label{leveldiag}
\begin{matrix}
& X_{\bar\eta}&\overset{\bar j}{\longrightarrow}& \bar X&\overset{\bar\iota}{\hookleftarrow}& X_{\bar s}\\
{}&\downarrow\,&{}&\downarrow\,&{}&\downarrow\,\\
& \bar\eta& \longrightarrow& \bar Z &{\hookleftarrow}& \bar s
\end{matrix}
\end{equation}
Let $\mathcal F$ be a $\bar{\mQ}_\ell$-sheaf on $X_\eta$, where $\ell\neq p$. Then $R\Psi (\mathcal F)=\bar{\iota}^*R\bar{j}_*(\mathcal F_{\bar{\eta}})$, the {\em complex of nearby cycles} of $\mathcal F$,  is an object of the  triangulated category of $\bar{\mQ}_\ell$-sheaves on $X_{\bar s}$ with continuous action of ${\rm Gal}(\bar \eta/\eta)$ compatible with its action on $X_{\bar s}$,
\begin{equation}
R\Psi (\mathcal F)\in D^b_c(X_s\times_s\eta, \bar{\mQ}_\ell) ,
\end{equation}
cf. \cite{SGA}, XIII, 2.1. The construction extends to the case where $\mathcal F$ is replaced by any bounded complex of $\bar{\mQ}_\ell$-sheaves on $X_\eta$. If $f:X\to X'$ is a morphism of $Z$-schemes of finite type, there are natural functor morphisms
\begin{equation} \label{RPsi_morup}
f^* R\Psi' \to R\Psi f_\eta^*,
\end{equation}
and 
\begin{equation}\label{RPsi_mordown} 
R\Psi' Rf_{\eta *}\to Rf_* R\Psi . 
\end{equation}
The first is an isomorphism when $f$ is smooth, and the second is an isomorphism when $f$ is proper.  

When $\mathcal F=\bar{\mQ}_\ell$, we also write $R\Psi$ for $R\Psi(\bar{\mQ}_\ell)$, and $R^i\Psi$ for $\mathcal H^i(R\Psi)$. We recall the explicit expression for $R\Psi$ in the case when $X$ has semi-stable reduction over $Z$, with all multiplicities of the divisors in the special fiber prime to $p$, cf. \cite{SGA}, I, 3.2,  \cite{I}, 3.5. 

Let $x\in X_{\bar s}$, and let $\bar x$ be a geometric point over $x$. Let $\mathcal C=\{C_i\}_i$ be the set of branches of $X_{\bar s}$ through $x$. Consider the following homological complex concentrated in degrees $1$ and $0$,
\begin{equation}
\begin{aligned}
K_x:\quad & \mZ^{\mathcal C}&\longrightarrow &\ \mZ\\
&(n_i)_i&\mapsto &\sum\nolimits_i n_i d_i 
\end{aligned}
\end{equation}
Here $d_i$ denotes the multiplicity of the branch $C_i$ through $x$. Under the hypothesis that all $d_i$ are prime to $p$, $R\Psi$ is tame, i.e., the action of ${\rm Gal}(\bar\eta/\eta)$ on $R\Psi$ factors through the tame Galois group ${\rm Gal}(\eta_t/\eta)$. Furthermore,
\begin{equation}\label{SGAfor}
\begin{aligned}
R^i\Psi_{\bar x}&=\bigwedge^i\big(H_1(K_x)\otimes\bar{\mQ}_\ell(-1)\big)\otimes R^0\Psi_{\bar x}\ ,\\
R^0\Psi_{\bar x}&= \bar{\mQ}_\ell^{H_0(K_x)} .
\end{aligned}
\end{equation} 
Here there is a transitive action of the inertia subgroup $I={\rm Gal}(\eta_t/\eta_{un})$ on $H_0(K_x)$, making the last identification equivariant. The set $H_0(K_x)$ arises as the set of connected components of $\big(X_{(\bar{x})}\big)_{\eta_t}$, where $X_{(\bar x)}$ denotes the  strict localization of $X$ in $\bar x$.

\subsection{The case $\A_0$}
We now apply these considerations to the scheme $\A_0$ over $Z=\Spec \mZ_p$. We denote the complex $R\Psi(\bar{\mQ}_\ell)\in D^b_c\big((\A_0\otimes\mF_p)\times_{\Spec \mF_p}\Spec \mQ_p, \bar{\mQ}_\ell\big)$ by $R\Psi_0$. Since the multiplicities of all divisors in the special fiber of $\A_0$ are equal to one, we have $R^0\Psi_0=\bar{\mQ}_\ell$. Furthermore, the special fiber of $\A_0$ is {\it globally} a union of smooth divisors. From  (\ref{SGAfor}) we deduce therefore the following formula for the pull-back of $R\Psi_0$ by the locally closed embedding $i_S: \A_{0, S}\hookrightarrow \A_0\otimes \mF_p$, 
\begin{equation}
i_S^*(R^i\Psi_0)=\bigwedge^i\big(\ker(\mZ^S\to\mZ)\otimes\bar{\mQ}_\ell(-1)\big) .
\end{equation}
In particular, $i_S^*(R^i\Psi_0)$ is a constant local system on $\A_{0, S}$ with trivial inertia action, for all $i$.

\subsection{The case $\A_1$} \label{A1_sec}
We next consider the scheme $\A_1$ over $Z=\Spec \mZ_p$. In this case, we use the notation $R\Psi_1$ for $R\Psi(\bar{\mQ}_\ell)$.  Using the formulas (\ref{SGAfor}) and keeping in mind that all multiplicities are the same $(= p-1)$, we obtain from Corollary \ref{locnormfor} (compatibility of the stratifications of $\A_0\otimes\mF_p$ and $\A_1\otimes\mF_p$)
\begin{equation} \label{preproj}
R\Psi_1 = R^0\Psi_1 \otimes \pi^*(R\Psi_0).
\end{equation}
Here $R^0\Psi_1$ is a lisse sheaf of rank $p-1$, and the action of inertia $I_p$ on cohomology sheaves is controlled by $R^0\Psi_1$: by the formulas (\ref{SGAfor}), $I_p$ acts trivially on each $R^i\Psi_0$ and via the regular representation of its natural quotient $\mu_{p-1}$ on the stalks of $R^0\Psi_1$.  

For our purposes it is most convenient to consider the direct image $\pi_*(R\Psi_1)$ on $\A_0$. More precisely, applying the functoriality (\ref{RPsi_mordown}) to the finite morphism $\pi$, we may write 
\begin{equation*}
\pi_*(R\Psi_1)=R\Psi\big(\pi_{\eta *}(\bar{\mQ}_\ell)\big) .
\end{equation*}
Now $\pi_\eta$ is a connected \'etale Galois covering with Galois group $T(\mF_p)$, cf. Corollary \ref{Galoiscov}. Hence we obtain a direct sum decomposition into eigenspaces for the $T(\mF_p)$-action,
\begin{equation}
\pi_{\eta *}(\bar{\mQ}_\ell)=\bigoplus_{\chi: T(\mF_p)\to \bar{\mQ}_\ell^\times}\bar{\mQ}_{\ell, \chi} . 
\end{equation}
Here each summand is a local system of rank one on $\A_{0\, \eta}$. Hence, setting $R\Psi_\chi=R\Psi(\bar{\mQ}_{\ell, \chi})$, we obtain a direct sum decomposition,
\begin{equation} \label{direct_sum}
\pi_*(R\Psi_1)=\bigoplus_{\chi: T(\mF_p)\to \bar{\mQ}_\ell^\times}R\Psi_\chi\  .
\end{equation}
On the other hand, applying the projection formula to (\ref{preproj}), we get
\begin{equation} \label{proj}
\pi_*(R\Psi_1) = \pi_*(R^0\Psi_1) \otimes R\Psi_0.
\end{equation}
Our aim is to understand $i^*_S(R\Psi_\chi)$, the restriction of $R\Psi_\chi$ to the stratum $\A_{0, S}$, for any $\chi$ and any $S$, and especially the inertia and $T(\mathbb F_p)$-actions it carries.  This will be achieved by restricting both (\ref{direct_sum}) and (\ref{proj}) to $\A_{0,S}$, and then comparing the right hand sides.  It is easiest to compare only the inertia invariants of cohomology sheaves (and this is enough for our purposes).  This is facilitated by the following lemma.

\begin{lemma} \label{const}
The sheaf $(R^0\Psi_1)^{I_p}$ of inertia invariants is canonically isomorphic to the constant sheaf $\bar{\mQ}_\ell$ on the special fiber of $\A_1$.     
\end{lemma}

\begin{proof}
The morphism (\ref{RPsi_morup}) gives a canonical $I_p$-equivariant morphism 
$$
\pi^*(R^0\Psi_0) \rightarrow R^0\Psi_1
$$
whose image lies in the inertia invariants since the left hand side is $\bar{\mQ}_\ell$.  It gives an isomorphism onto $(R^0\Psi_1)^{I_p}$ because $I_p$ acts transitively on a basis for any stalk of $R^0\Psi_1$, cf. (\ref{SGAfor}). 
\end{proof}

Now $I_p$ acts on the cohomology sheaves of (\ref{preproj}) and (\ref{proj}) through a finite quotient.  It is easily shown that applying $R^i\pi_*$ (to a complex with finite action of $I_p$ on its cohomology sheaves) commutes with the formation of $I_p$-invariants.  This, along with (\ref{proj}), Lemma \ref{const}, and the triviality of the action of $I_p$ on $R^i\Psi_0$, yields the isomorphism
\begin{equation*}
\big(R^i\pi_*(R\Psi_1)\big)^{I_p} = \pi_*(\bar{\mQ}_\ell) \otimes R^i\Psi_0,
\end{equation*}
for each $i$.  
Moreover, $i_S^*$ commutes with taking $I_p$-invariants in this situation, so applying $i^*_S$ to the above equation yields an isomorphism
\begin{equation} \label{compari}
\big(\mathcal H^i\big(i^*_S(\pi_*(R\Psi_1))\big)\big)^{I_p} = \pi_{S*}(\bar{\mQ}_\ell) \otimes i^*_S(R^i\Psi_0).
\end{equation}
This identification is canonical, and therefore is an equality of sheaves on $\A_{0}\otimes \mF_p$ with continuous ${\rm Gal}(\bar{\mQ}_p/\mQ_p)$-action, and is equivariant for the action of $T(\mF_p)$ which comes via its action by deck transformations in the generic fiber. On the right hand side, this action factorizes over the quotient $T^S(\mathbb F_p)$.  We now compare the eigenspace decompositions of both sides of (\ref{compari}) under the action of $T(\mF_p)$.
                                        
\begin{thm} \label{main_geom_thm} {\rm (i)} For each character $\chi: T(\mF_p)\to \bar{\mQ}_\ell^\times$ and each $S$, the inertia group $I_p$ acts on each cohomology sheaf $\mathcal H^i\big(i^*_S(R\Psi_\chi)\big)$ through a finite quotient.

\smallskip

\noindent{\rm (ii)} If the character $\chi: T(\mF_p)\to \bar{\mQ}_\ell^\times$ does not factor through $T^S(\mF_p)$, then 
$$\big(\mathcal H^i\big(i_S^*(R\Psi_\chi)\big)\big)^{I_p} = 0$$ for each $i$.

\smallskip

{\rm Assume now that $\chi: T(\mF_p)\to \bar{\mQ}_\ell^\times$ is induced by a character $\bar{\chi}: T^S(\mF_p)\to \bar{\mQ}_\ell^\times$. }
 \smallskip
 
\noindent {\rm (iii)} Then 
$$
\big(\mathcal H^i\big(i_{S}^*(R\Psi_\chi)\big)\big)^{I_p} = \bar{\mQ}_{\ell\, \bar{ \chi}} \otimes \mathcal H^i\big(i_S^*(R\Psi_0)\big)\ ,
$$
where $\bar{\mQ}_{\ell\, \bar{ \chi}}$ is the lisse sheaf on $\A_{0, S}$ obtained from $\pi_{S *}(\bar{\mQ}_\ell)$ by push-out via $T^S(\mF_p)\to \bar{\mQ}_\ell^\times$. Furthermore, the local system on $\A_{0, S}\otimes \bar{\mF}_p$ corresponding to  $\bar{\mQ}_{\ell\, \bar{ \chi}}$ is defined on each connected component by the composition 
$$
\pi_1(\A_{0, S}\otimes \bar{\mF}_p, x)\twoheadrightarrow T^S(\mF_p)\xto{\bar \chi}\bar{\mQ}_\ell^\times ,
$$
where the first map is surjective. 
\smallskip

\noindent {\rm (iv)} We have an equality of semi-simple traces of Frobenius, for all $x\in \A_S(\mF_{p^r})$, 
 $$
{\rm Tr}^{ss}\big(\Phi^r_{\fp}, (i^*_S(R\Psi_\chi))_x\big) = {\rm Tr}^{ss}\big(\Phi^r_{\fp}, (\bar{\mQ}_{\ell\, \bar{ \chi}})_x\big) \cdot {\rm Tr}^{ss}\big(\Phi^r_{\fp}, (i^*_S(R\Psi_0))_x\big)\ .
$$ 

\end{thm}
\begin{proof} The assertion (i) is clear from the discussion above, in view of (\ref{direct_sum}).  
The assertion (ii) is proved by substituting (\ref{direct_sum}) into (\ref{compari}) and using the equivariance of the identification (\ref{compari}), since  this action factors on $\pi_{S *}(\bar{\mQ}_\ell)$  through  $T^S(\mF_p)$.  
 For the same reason the first assertion in (iii) holds. The second assertion in (iii) is also clear, and implies by the monodromy Theorem \ref{monodr} the surjectivity of the first arrow.   The assertion (iv) follows using (i), (iii), and the definition of semi-simple trace of Frobenius.


\end{proof}

\begin{Remark} \label{unip}
The complexes $i^*_S(R\Psi_\chi)$ possess unipotent/non-unipotent decompositions in the sense of \cite{GH}, $\S5$:
$$
i^*_S(R\Psi_\chi) = i^*_S(R\Psi_\chi)^{\rm unip} \oplus i^*_S(R\Psi_\chi)^{\rm non-unip}.
$$
If $\mathcal F \in D^b_c(X_s \times_s \eta, \bar{\mQ}_\ell)$ and $I_p$ acts through a finite quotient on $\mathcal H^i\mathcal F$, then $(\mathcal H^i\mathcal F)^{I_p} = (\mathcal H^i\mathcal F)^{\rm unip}$.  Using this and the permanence properties of unipotent/non-unipotent decompositions (cf. loc.cit.), one can show as in the proof of Theorem \ref{main_geom_thm} that $i^*_S(R\Psi_\chi)^{\rm unip} = 0,$ unless $\chi$ factorizes through a character $\bar{\chi}$ on $T^S(\mathbb F_p)$, in which case
\begin{equation} \label{unip_eq}
\big(i^*_S(R\Psi_\chi)\big)^{\rm unip} = \bar{\mQ}_{\ell\, \bar{ \chi}} \otimes R\Psi_0.
\end{equation}
But it is easy to see that for any {\em non-unipotent} complex $\mathcal F$, the semi-simple trace of Frobenius function ${\rm Tr}^{ss}(\Phi^r_{\fp}, \mathcal F)$ is identically zero.   Thus, as far as {\em semi-simple trace of Frobenius functions} is concerned, we may omit the ``unip'' superscript on the left hand side of (\ref{unip_eq}).  This gives another perspective on part (iv) of Theorem \ref{main_geom_thm}.
\end{Remark}

\subsection{Full determination of $R^0\Psi_1$}
To obtain statement (iv) of Theorem \ref{main_geom_thm} above, we only used that the action of 
$I_p$ on $R^0\Psi_1$ factors through a finite quotient,  and that $(R^0\Psi_1)^{I_p}$
is the constant sheaf $\overline{\mQ}_\ell$ on $\A_1\otimes_{\mZ_p}\mF_p$. In this subsection we will show how one 
can in fact determine explicitly $R^0\Psi_1$ with its $I_p$-action (and this implies the
facts just mentioned).  The result will
not be used in the rest of the paper. 

\begin{lemma}\label{conn_comp}
There exists a morphism of finite flat normal $Z$-schemes
\begin{equation*}
\beta : Z_1\to Z_0
\end{equation*}
 with the following properties. 

\noindent{\rm (i)} $Z_1\times_Z\eta$ and $Z_0\times_Z\eta$ are finite disjoint sums of  abelian 
field extensions of $\mQ_p$. In the case of $Z_0$, these extensions are unramified
extensions of $\mQ_p$, and  in the case of $Z_1$ they are totally ramified over $Z_0$.

\smallskip

\noindent {\rm (ii)} The structure morphisms of $\A_0$ resp. $\A_1$ factor  through $Z_0$,  resp.
$Z_1$ and induce  bijections 
$$\pi_0 (\A_0\times_Z\overline{\eta})\xto{\,\,\sim\,\,}\pi_0 (Z_0\times_Z\overline{\eta})\quad 
\mbox{and }\quad  \pi_0 (\A_1\times_Z\overline{\eta})\xto{\,\,\sim\,\,}\pi_0 (Z_1\times_Z\overline{\eta})\ .$$

\smallskip

\noindent{\rm(iii)} Each fiber of the map
\begin{equation*}
\pi_0 (Z_1\times_Z\overline{\eta})\to\pi_0 (Z_0\times_Z\overline{\eta})
\end{equation*}
has $p-1$ elements.
\end{lemma}
\begin{proof}
The generic fibers of $Z_0$, resp. $Z_1$ are  given by the Galois sets $Z_0(\bar\eta)$, resp.\ $Z_1(\bar\eta)$. We define them by the identities
in (ii). Then $Z_0$, resp.\ $Z_1$ are defined as the normalizations of $Z$ in $Z_{0, \eta}$, resp. $Z_{1, \eta}$ and the factorizations in Part (ii)
follow from the normality of $\A_0$, resp.\ $\A_1$. 

It then remains to determine the geometric connected components of $\A_0$ resp. $\A_1$ and to show that the action of ${\rm Gal}(\overline{\mQ}_p/\mQ_p)$ on them factors through a suitable abelian quotient. It 
suffices to do this over $\mC$. More precisely, there is a natural way to formulate
moduli problems $\A_0^\natural$, resp.\ $\A_1^\natural$ over $\Spec\mathcal O_F\otimes\mZ_{(p)}$
which after base change $\mathcal O_F\otimes\mZ_{(p)}\to\mathcal O_{F, \fp}$ give back $\A_0$,
resp.\ $\A_1$.
Then the chosen complex embedding $\epsilon: F\to\mC$ defines $\mC$-schemes
$\A_0^\natural\otimes_{\mathcal O_F}\mC$ and $\A_1^\natural\otimes_{\mathcal O_F}\mC$ which are disjoint
sums\footnote{The argument from section \ref{counting_pts_sec} shows that in our particular case, there is only one copy.} 
of Shimura varieties of the form $S_{K_0}= Sh (G, h^{-1}; K_0)$,  resp.
$S_{K_1}= Sh (G, h^{-1}; K_1)$, where $K_0 = I\cdot K^p$ and
$K_1 = I^+\cdot K^p$, for an Iwahori subgroup $I$ of $G(\mQ_p)$ and its
pro-unipotent radical $I^+$ and some open compact subgroup 
$K^p\subset G(\mA_f^p)$. Now for $\pi_0 (S_{K_i}\otimes_F\mC)$ there is the
expression (cf., e.g., \cite{Kur}, \S 4) 
\begin{equation}
\pi_0 (S_{K_i}\otimes\mC) = \mathcal T (\mQ)^0\backslash \mathcal T (\mA_f)/\nu^0 (K_i)\ ,
\end{equation}
where $\nu^0 : G\to \mathcal T$ is the natural homomorphism to the maximal torus
quotient and where $\mathcal T(\mathbb Q)^0 := \mathcal T(\mathbb Q) \cap \mathcal T(\mathbb R)^0$.  Furthermore, the action of the Galois group $\Gal (\overline{\mQ}/E)$ of the Shimura field $E$ on $\pi_0 (S_{K_i}\otimes\mC)$ factors through its maximal abelian quotient and is given by the reciprocity map 
$\rho: (E\otimes\mA_{ f})^\times\to \mathcal T(\mA_f)$ 
of $S_{K_i}$. Now   $\mathcal T$ is  given as follows (\cite{Kur}, 4.4), 
\begin{equation*}
\mathcal T =
\begin{cases}
\mathcal T^1\times\mG_m
&
\text{ if }\, d\, \text{ is even,}
\\
\Res_{F/\mQ} (\mG_m)
&
\text{ if }\, d\, \text{ is odd,}
\end{cases}
\end{equation*}
where $\mathcal T^1 = \ker (N:\Res_{F/\mQ} (\mG_m)\to\mG_m)$. And $\nu^0$ is given by 
\begin{equation*}
\nu^0(g) = 
\begin{cases} 
(\frac{\det(g)}{\nu(g)^k},\nu(g))&\text{if $d=2k$ is even,}\\
\frac{\det(g)}{\nu(g)^k}&\text{if $d=2k+1$ is odd.}
\end{cases}
\end{equation*}
Here $\det (g)$ denotes the reduced norm of $g$, and $\nu (g)$ the similitude 
factor of $g$. Furthermore, noting that $E=F$ for $d>2$ and $E=\mQ$ for $d=2$, the reciprocity map is given as follows. 

$$\rho(a) = \begin{cases}
(\left(\frac{a}{\bar a}\right)^{k-1}, a\bar a)&\text{if $d=2k>2$ is even,}\\
(1, a)&\text{if $d=2$ }\\
\left(\frac{a}{\bar a}\right)^{k-1}\,a&\text{if $d=2k+1$ is odd.}
\end{cases}
$$
It  follows that $\pi_0 (S_{K_0}\otimes\mC) $, resp. $\pi_0 (S_{K_1}\otimes\mC) $, corresponds to  the disjoint sum of copies of the spectrum of an abelian extension $L_0$, resp. $L_1$, of $E$ with
norm group $\rho^{-1}(\nu^0(K_0))$, resp. $\rho^{-1}(\nu^0(K_1))$, in $(E\otimes\mA_{ f})^\times$. It is now an elementary calculation to check 
 for $d>2$ that $L_0/F$ is unramified at the places $v$ over $\mathfrak p$,  and for $d=2$ that $L_0/\mQ$ is unramified at the places $v$ over $p$,   and that $L_1/L_0$ is totally ramified of degree $p-1$ at all these places, cf. \cite{Ku}. For instance, for $d=2$, the norm group of $L_0/\mQ$ is of the form 
 $\mZ_p^\times\cdot U^p$ for $U^p\subset (\mA_f^p)^\times$, and the norm group of $L_1/\mQ$ is 
 $(1+p\mZ_p)\cdot U^p$. The lemma is proved. 
\end{proof}

Let $\mathcal L = R\Psi (\beta_{\eta_\ast}\overline{\mQ}_\ell)$. Then $\mathcal L$ is a complex concentrated
in degree zero, and is a sheaf with a continuous action of $\Gal (\overline{\mQ}_p/\mQ_p)$ of rank
$p - 1$ on the special fiber of $Z_0$. Since the covering $Z_1/Z_0$ is totally ramified,
the fiber of $\mathcal L$ at each geometric point localized at a point $s_0\in Z_0\times_Z s$ is
isomorphic as a $I_p$-module to $\Ind_{\Gal (Z_1, s_0)}^{\Gal (Z_0, s_0)} (\overline{\mQ}_\ell)$.

The following proposition gives the explicit determination of $R^0\Psi_1$. 
\begin{prop}\label{expl_expr}
There is a natural isomorphism of sheaves with continuous $\Gal (\overline{\mQ}_p/\mQ_p)$-action
\begin{equation*}
\mathcal L_{\A_1}\xto{\,\,\sim\,\,} R^0\Psi_1\, ,
\end{equation*}
where
\begin{equation*}
\mathcal L_{\A_1} = p_0^\ast (\mathcal L)
\end{equation*}
is the inverse image of $\mathcal L$ under the composed morphism
\begin{equation*}
p_0 : \A_1\xto{\,\,p_1\,\,} Z_1\xto{\,\,\beta\,\,}Z_0\, .
\end{equation*}
\end{prop}
\begin{proof}
The functoriality (\ref{RPsi_morup}) induces a homomorphism
\begin{equation}\label{first_mor}
\mathcal L_{\A_1} = p_0^\ast \big(R\Psi (\beta_{\eta_\ast}\overline{\mQ}_\ell)\big)\to R\Psi \big(p_{0, \eta}^\ast\beta_{\eta\ast} (\overline{\mQ}_\ell)\big)\, .
\end{equation}
On the other hand, by adjunction,  there is a natural homomorphism
\begin{equation*}
p_{0, \eta}^\ast\beta_{\eta *}(\overline{\mQ}_\ell) = p_{1, \eta}^\ast  (\beta_\eta^\ast\beta_{\eta *} (\overline{\mQ}_\ell))\to\overline{\mQ}_{\ell {\A_{1, \eta}}}\, .
\end{equation*}
Applying $R\Psi$ and composing with (\ref{first_mor}) defines in degree zero the natural morphism
\begin{equation}\label{ident_sheaf}
\mathcal L_{\A_1} \to R^0\Psi_1\, .
\end{equation}
To prove that this is an isomorphism, it suffices to do this point-wise. Let 
$x\in\A_1 (\overline{\mF}_p)$ with image $y\in\A_0 (\overline{\mF}_p)$. We consider the
commutative diagram
\begin{equation}\label{diagpi0}
\begin{aligned}
\xymatrix{
\pi_0 (\A_{1(x)} \times_Z\overline{\eta})\ar[d]\ar[r] & \pi_0 (\A_1\times_Z\overline{\eta}) = \pi_0 (Z_1\times_\eta\overline{\eta})\ar[d]\\
\pi_0 (\A_{0(y)}\times_Z\overline{\eta})\ar[r] & \pi_0 (\A_0\times_Z\overline{\eta}) = \pi_0 (Z_0\times_\eta\overline{\eta})\, .
 }
 \end{aligned}
\end{equation}
From the local equations for $\A_0$ resp. $\A_1$ in section \ref{local_struc}, we know that
$\pi_0 (\A_{0(y)}\times_Z\overline{\eta})$ consists of one element and
$\pi_0 (\A_{1(x)}\times_Z\overline{\eta})$ consists of $p - 1$ elements. 
On the other hand, the covering $\A_1\times_Z\overline{Z}$ of $\A_0\times_Z\overline{Z}$ is totally ramified 
along its special fiber. Hence, the upper horizontal map in (\ref{diagpi0}) induces a surjection onto the fiber of
the right vertical map over the  element of $\pi_0(\A_0\times_Z\overline{\eta})$  given by the image of the lower horizontal map. 
By part (iii) of
Lemma \ref{conn_comp}, this fiber has $p-1$ elements. This shows that each connected component of 
$\A_{1(x)}\times_Z\overline{\eta}$ lies in a different connected component of
$\A_1\times_Z\overline{\eta}$, all of which map to the same connected component
of $\A_0\times_Z\overline{\eta}$. Hence the stalk in $x$ of the homomorphism
(\ref{ident_sheaf}) is an isomorphism, and this proves the assertion.
\end{proof}
Proposition \ref{expl_expr}, along with  the projection formula, implies,    as  in the proof of (\ref{proj}) and (\ref{compari}), an isomorphism of complexes with $I_p$-action on each stratum $\A_{0, S}$, 
\begin{equation}
i_{S}^*(\pi_*R\Psi_1)\cong \big(\mathcal L_{\A_{0, S}}\otimes \pi_{S *}(\bar{\mQ}_{\ell})\big)\otimes i_S^*(R\Psi_0) , 
\end{equation}
where $\mathcal L_{\A_{0, S}}$ denotes the pull-back of $\mathcal L$ to ${\A_{0, S}}$ under the structure morphism ${\A_{0, S}}\to Z_0$. This then yields the expression 
\begin{equation}
i_{S}^*(R\Psi_\chi)\cong \big(\mathcal L_{\A_{0, S}}\otimes \bar{\mQ}_{\ell\, \bar{ \chi}}\big)\otimes i_S^*(R\Psi_0) ,
\end{equation}
if the character $\chi$ is induced by $\bar{\chi}:T^S(\mF_p)\to \bar{\mQ}_{\ell}^\times$ (and the left hand side is zero if $\chi$ is not of this form).

\section{Combinatorics related to the test function} \label{comb_sec}

\subsection{Critical indices} \label{crit_def}

Recall $\mu_0=(1, 0, \ldots, 0)\in X_*(\GL_d)$ from section \ref{subKR}.  For $w \in {\rm Adm}(\mu_0)$, define the subset $S(w) \subseteq \{ 1, \dots, d \}$ of {\em critical indices} by 
$$
S(w) = \{ j ~ | ~  w \leq t_{e_j} \}.
$$
Recall that the Bruhat order $\leq$ here and below is defined using the base alcove ${\bf a}$, which is contained in the {\em antidominant} Weyl chamber.  The following subsections will help us understand this subset combinatorially.  In particular we will show that this definition coincides with the one appearing in  Proposition \ref{critad}.

\subsection{Lemmas on admissible subsets}

Let $\mu$ be a minuscule dominant coweight for an arbitrary  based root system.  Let $w_0$ denote the longest element in the finite Weyl group $W$.  Define the Bruhat order $\leq$ using the base alcove ${\bf a}$ in the antidominant Weyl chamber.  Given an 
element $w \in {\rm Adm}(\mu)$, when is it true that $w \leq t_{w_0(\mu)}$?  The answer will depend on the decomposition $w = t_{\lambda(w)}\bar w \in X_* \rtimes W$.  

\begin{lemma} \label{w_leq_mu}
Let $w = t_{\lambda(w)} \bar w \in {\rm Adm}(\mu)$.  Then $w \leq t_{w_0(\mu)}$ if and only if $\lambda(w) = \ w_0(\mu)$.
\end{lemma}

\begin{proof}
We give a proof using results about the functions $\Theta^-_{\lambda}$, taken from \cite{HP}.  Recall that $\Theta^-_\lambda$ is the ``antidominant'' variant of the ``usual'' Bernstein function $\Theta_\lambda$.  However, \cite{HP} uses the Bruhat order defined using the base alcove in the {\em dominant} Weyl chamber.  When we recast the results of loc.~cit. using our Bruhat order, we actually recover the results used below about the functions $\Theta_\lambda$, rather than the functions $\Theta^-_{\lambda}$ studied in loc.~cit.. 

Recall that the Bernstein 
function $z_\mu$ may be expressed as
$$
z_\mu = \sum_{\lambda \in W\mu} \Theta_{\lambda}.
$$
In \cite{HP}, Cor. 3.5, it is proved that for $\lambda \in W\mu$,
\begin{equation} \label{eq:HP}
{\rm supp}(\Theta_\lambda) = \{ w \in \widetilde{W} ~ | ~ \lambda(w) = \lambda ~ \text{and} ~ w \leq t_\lambda \}.
\end{equation}
Therefore ${\rm supp}(\Theta_\lambda) \cap {\rm supp}(\Theta_{\lambda'}) = \emptyset$ if $\lambda \neq \lambda'$.

Further, by  \cite{H1b}, Prop. 4.6 (or the Appendix, Proposition \ref{suppz}) we have ${\rm supp}(z_\mu) = {\rm Adm}(\mu)$.  Thus we have
\begin{equation} \label{eq:adm=union}
{\rm Adm}(\mu) = \coprod_{\lambda \in W\mu} \{ w \in \widetilde{W} ~ | ~ \lambda(w) = \lambda ~ \text{and} ~ w \leq t_\lambda \}.
\end{equation}

If $w \leq t_{w_0(\mu)}$, then $w$ lies in the support of $\widetilde{T}^{-1}_{-w_0(\mu)} = \Theta_{w_0(\mu)}$ and (\ref{eq:HP}) implies that $\lambda (w) = w_0(\mu)$.  The reverse implication in the lemma follows easily from (\ref{eq:adm=union}). 
\end{proof}

Next we assume $G = {\rm GL}_d$, and fix $\mu_0 := (1,0^{d-1})$. For each $i$ with $0 \leq i \leq d$ let 
$\omega_i = (1^i,0^{d-i})$ denote the $i$-th fundamental coweight.  For $j$ with $1 \leq j \leq d$ let $e_j = \omega_j - \omega_{j-1}$, the $j$-th standard basis vector.  Write $\bar{\omega}_i = -\omega_i$ for all $i$.  

Lemma \ref{w_leq_mu} says that $w \in {\rm Adm}(\mu_0)$ satisfies $w \leq t_{e_d}$ iff $\lambda(w) = e_d$.   Fix an integer $j$ with $1 \leq j \leq d$.  When is it true that $w \leq t_{e_j}$?

Let ${\bf a}$ denote the base alcove whose vertices are represented by the vectors $\bar{\omega}_0, \dots, \bar{\omega}_{d-1}$.  Let $\Omega \subset \widetilde{W}$ denote the subgroup of elements stabilizing ${\bf a}$ as a set.  Define the element 
$$\tau := t_{e_d} (d \,\, d\!\!-\!\!1 \,\, \cdots 1) \in \Omega,$$
which satisfies $t_{\mu_0} \in W_{\rm aff} \tau$.  We have $\widetilde{W} = W_{\rm aff} \rtimes \Omega$, and conjugation by $\tau$ permutes the simple affine reflections in $W_{\rm aff}$.  Viewed as an automorphism of the base apartment, $\tau$ cyclically permutes the vertices of the base alcove.  In fact, we have 
$$
\tau(\bar{\omega}_i) = \bar{\omega}_{i-1},
$$
for $1 \leq i \leq d$.   Also, we shall need the identity $\tau^{d-j} t_{e_d} \tau^{-(d-j)} = t_{e_{j}}$, for any $j$ with $1 \leq j \leq d$.

\begin{lemma} \label{w_leq_e_j}
Suppose $w \in {\rm Adm}(\mu_0)$.  Write $w = t_{\lambda} \bar w$.  Then the following are equivalent:
\begin{enumerate}
\item[{\rm (i)}] $w \leq t_{e_j}$;
\item[{\rm (ii)}] $\bar{\omega}_{j} - \bar w(\bar{\omega}_{j}) = \lambda - e_j$;
\item[{\rm (iii)}] $\lambda + \bar w(\bar{\omega}_{j}) = \bar{\omega}_{j-1}$;
\item[{\rm (iv)}] $t_\lambda \bar w(\bar{\omega}_{j}) = \tau(\bar{\omega}_{j}),$ i.e., the alcoves $w{\bf a}$ and $\tau{\bf a}$ share the same type $j$ vertex (interpreting $\bar{\omega}_d$ as type $0$).
\end{enumerate}
\end{lemma}

\begin{proof}
Using Lemma \ref{w_leq_mu}, we see that 
\begin{align*} 
w \leq t_{e_j}  &\Leftrightarrow \tau^{-(d-j)} w \tau^{d-j} \leq t_{e_d} \\
&\Leftrightarrow  \lambda(\tau^{-(d-j)} t_\lambda \bar w \tau^{d-j}) = e_d \\
&\Leftrightarrow \tau^{-(d-j)} t_\lambda \bar{w} \tau^{d-j} (\bar{\omega}_0) = t_{e_d}(\bar{\omega}_0) \\
&\Leftrightarrow t_\lambda \bar{w}(0^j,1^{d-j}) = (0^{j-1},1^{d-j+1}) \\
&\Leftrightarrow \lambda + \bar{w}(\bar{\omega}_j) = \bar{\omega}_{j-1}, 
\end{align*}
showing that (i) $\, \Leftrightarrow \,$ (iii).  The equivalence of (ii-iv) is immediate.
\end{proof}

We need to describe explicitly the elements of ${\rm Adm}(\mu_0)$.  An element $w = t_{\lambda}\bar{w} \in \widetilde{W}$ gives rise to a sequence of $d$ vectors $v_i = w(\bar{\omega}_i) \in \mathbb Z^d$ for $i = 0,1, \dots, d-1$.  For a vector $v \in \mathbb Z^d$, let $v(j)$ denote its $j$-th entry.  We say that $w = t_{\lambda} \bar w$ is {\em minuscule} (meaning $(1^r,0^{d-r})$-permissible for some $r$ with $1 \leq r \leq d$, see \cite{KR}) if and only if $\lambda(j) \in \{0,1 \}$ for each $j$, and the following conditions hold:
\begin{equation}
\bar{\omega}_i(j) \leq v_i(j) \leq \bar{\omega}_i(j) + 1
\end{equation}
for each $i = 0,1, \dots, d-1$ and $j =1,2, \dots, d$.
It is easy to see that these conditions are equivalent to the following conditions:  $\lambda(j) \in \{0,1\}$ for all $j$ and
 \begin{align*}
\lambda(j) = 0 &\Rightarrow \bar w^{-1}(j) \geq j \\
\lambda(j) = 1 &\Rightarrow \bar w^{-1}(j) \leq j\ ,
\end{align*}
cf.\ \cite{H1a}, Lemma 6.7.   As a consequence, using the equality ${\rm Perm}(\mu) = {\rm Adm}(\mu)$,  which holds in this setting (\cite{KR},\cite{HN}), we easily derive the following description of ${\rm Adm}(\mu_0)$.

\begin{lemma} \label{general_form}
Let $w = t_{e_m} \bar w$, for some $\bar w \in S_d$.  Then $w$ belongs to ${\rm Adm}(\mu_0)$ if and only if $\bar w$ is a cycle of the form
$$
\bar w = (m_k \, m_{k-1} \, \cdots \, m_1)
$$
where $ m = m_k >  m_{k-1} > \cdots > m_1$.\qed
\end{lemma}

\subsection{Main combinatorial proposition} \label{main_comb_subsec} 

Let $\mu_0 = (1,0^{d-1})$.  The upper triangular Borel subgroup $B \subset {\rm GL}_d$ determines the notion of {\em dominant} coweight, also called
{\em $G$-dominant}.

The semistandard Levi subgroups $M$ in ${\rm GL}_d$ correspond bijectively with disjoint decompositions
$$
\mathcal O_1 \amalg \cdots \amalg \mathcal O_t = \{ 1, 2, \dots, d \}.
$$
Fix such a decomposition and the corresponding Levi subgroup $M$.  Then the Weyl group $W_M$ is the product of the permutation groups on the subsets $\mathcal O_i$.  Each $\mathcal O_i$ can be identified with a $W_M$-orbit inside $W\mu_0$.  Note that  $B_M := B \cap M$ is a Borel subgroup of $M$, hence defines a notion of {\em $M$-dominant} coweight.  Let $\mu_i$ denote the unique $M$-dominant element of $\mathcal O_{i}$, and henceforth set $\mathcal O_{\mu_i} := \mathcal O_i$.

Next we define the following subset of ${\rm Adm}^G(\mu_0)$, for each $i=1, \ldots,t$,
$${\rm Adm}^G(\mathcal O_{\mu_i}) = \{ w \in {\rm Adm}^G(\mu_0) ~ | ~ S(w) \subseteq \mathcal O_{\mu_i} \}.
$$
Also, define
$$
{\rm Adm}^G(\mu_0,M) = \coprod_{i=1,\dots,t} {\rm Adm}^G(\mathcal O_{\mu_i}).
$$

\bigskip

For a dominant minuscule coweight $\mu$ of any based root system, let $k_\mu := q^{\langle \rho, \mu \rangle}z_\mu$ denote the Kottwitz function.  Here we think of $q$ as an indeterminate, or as $p^r$ if working with a group over $\QQ_{p^r}$.  Also, we have fixed a notion of positive root (or a semi-standard Borel subgroup $B$), and we define $\rho$ to be the half-sum of the $B$-positive roots.  In the Drinfeld case, we have from \cite{H1b} the simple formula for each $w \in {\rm Adm}(\mu_0)$:
\begin{equation} \label{Ko_fcn_GLd}
k_{\mu_0}(w) = (1-q)^{\ell(t_{\mu_0})-\ell(w)}.
\end{equation}
Define for $w \in {\rm Adm}(\mu_0)$, 
\begin{equation}
{\rm codim}(w) := \ell(t_{\mu_0}) - \ell(w) .
\end{equation}
Now let $\nu \in W\mu_0$ denote any $M$-dominant element (i.e., one of the elements $\mu_i$ above). We denote by  $k_{\mathcal O_\nu}$ the function on $\widetilde{W}$ which is equal to  $k^G_{\mu_0}(w)$ at $w \in {\rm Adm}^G(\mathcal O_\nu)$ and which vanishes outside ${\rm Adm}^G(\mathcal O_\nu)$, i.e., 
\begin{equation}
k_{\mathcal O_\nu} = k^G_{\mu_0}|_{{\rm Adm}^G(\mathcal O_\nu)}.
\end{equation}


\begin{prop} \label{main_comb_prop} 
\noindent {\rm (a)} Consider an element $w = t_{e_m} \bar w \in {\rm Adm}(\mu_0)$, where $\bar w$ is a cycle of the form $(m_k \, m_{k-1} \, \cdots \, m_1)$ with $m = m_k > \cdots > m_1$ {\em (Lemma \ref{general_form})}.  
Then $S(w) = \{m_1, \dots, m_k \}$.

\smallskip

{\rm Fix a semi-standard Levi $M$ as above, and let $\nu\in W\mu_0$ be $M$-dominant. }

\smallskip

\noindent {\rm (b)} We have $w \in {\rm Adm}^G(\mathcal O_\nu)$ if and only if $m_i \in \mathcal O_\nu$ for all $i$.   Thus, ${\rm Adm}(\mathcal O_\nu) \subset \widetilde{W}_M := X_*(T) \rtimes W_M$, since $\bar w \in W_M$.  In fact, we have the equality
$$ {\rm Adm}^G(\mathcal O_\nu) = {\rm Adm}^M(\nu),$$
where the right hand side is the $\nu$-admissible subset for the Levi subgroup $M$ and the $M$-dominant coweight $\nu$, which is defined using the Bruhat order on $\widetilde{W}_M$ determined by the alcove in the $M$-{\em antidominant} chamber whose closure contains the origin.

\smallskip

\noindent {\rm (c)}  For $w \in Adm(\mu_0)$, we have ${\rm codim}(w) = |S(w)| -1$.

\smallskip

\noindent {\rm (d)}  For $w, y \in Adm(\mu_0)$, we have $y \leq w \Longleftrightarrow S(w) \subseteq S(y)$.

\smallskip

\noindent {\rm (e)}  There is an equality $$k_{\mathcal O_\nu} =  k^M_\nu,$$
where $k^M_\nu$ denotes the Kottwitz function associated to $M$ and the $M$-dominant coweight $\nu$.
\end{prop}

\begin{figure}
\includegraphics[width=10cm]{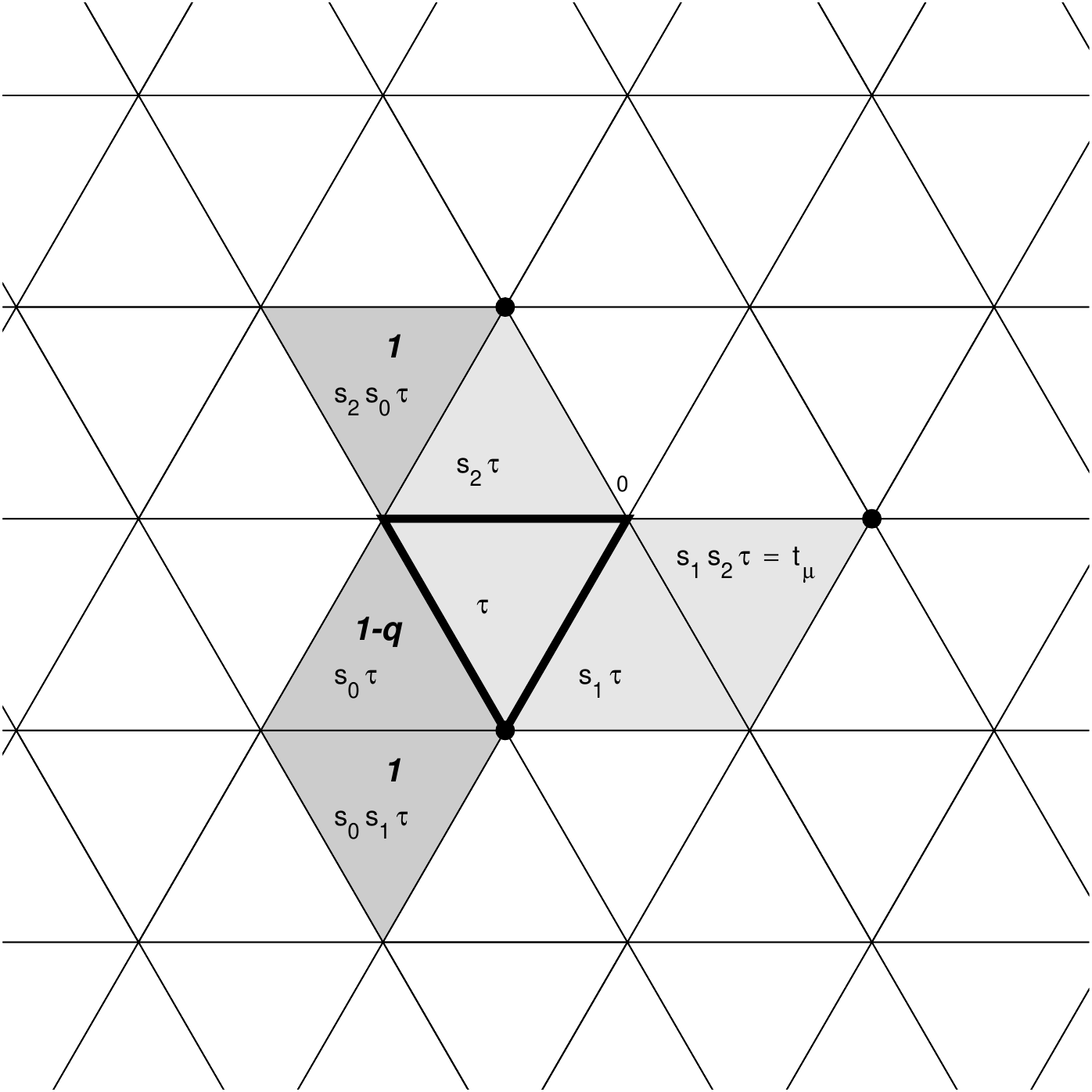}
\caption{\it Let $\mu = (1,0,0)$.  The three marked vertices represent the Weyl group orbit of $\mu$.  
The set ${\rm Adm}^{{\rm GL}_3}(\mu)$ labels the 7 gray alcoves, with $\tau$ labeling the base alcove.  
Consider the decomposition $\{1\} \amalg \{2, 3\}=\{1, 2, 3\}$, so that $M = {\mathbb G}_m \times {\rm GL}_2$. 
The values of the function $k_{{\mathcal O}_\nu}$ for $\nu = (0,1,0)$ are indicated on the dark gray alcoves comprising its support ${\rm Adm}^{{\rm GL}_3}(\mathcal O_\nu)$.}
\end{figure}

\begin{proof}
We start with part (a).  If $k=1$, then $w = t_{e_m}$ and clearly $S(w) = \{ m \}$.   So, from now on, we may assume $k >1$, i.e.,  $m_k > m_1$.  
For any $j \in \{ 1, \dots, d\}$, Lemma \ref{w_leq_e_j} shows that  
$$
j \in S(w) \Longleftrightarrow \bar{\omega}_{j} - \bar w(\bar{\omega}_{j}) = e_m - e_j.
$$
If $j > m_k$ or $j < m_1$ (resp. $j = m$), then the condition on the right fails (resp. holds), and so $j \notin S(w)$ (resp. $j=m \in S(w)$).  If $m_1 \leq j < m_k$, let $n$ denote the unique index with $m_n \leq j < m_{n+1}$.  A calculation shows that
$$
\bar{\omega}_{j} - \bar w(\bar{\omega}_{j}) = e_m - e_{m_{n}}.
$$
Thus 
$$
j \in S(w) \Longleftrightarrow j = m_{n}.
$$
Putting all this  together, we get that $S(w) = \{m_1, \dots, m_k \}$, proving (a).  Part (a) implies part (b), once we take into account Lemma \ref{general_form} for $G$ as well as its obvious analogue for $M$.

For parts (c) and (d), conjugation by an appropriate power of $\tau$ reduces us to the case where $w \leq t_{e_d}$, i.e. $\lambda(w) = e_d$ (Lemma \ref{w_leq_mu}).  For $1 \leq i \leq d-1$, let $s_i$ denote the simple reflection corresponding to the transposition $(i ~ i\!\!+\!\!1) \in S_d$.  We have the equalities
\begin{equation} \label{e_d_reduced}
t_{e_d} = \tau ~(1 ~ 2 \cdots d) = \tau s_{1} s_2 \cdots s_{d-1},
\end{equation}
where the expression on the right is reduced.  Since $w \leq   t_{e_d}$, there exists a strictly-increasing list of indices $i_1, \dots, i_q \in \{ 1, \dots, d-1 \}$ such that $w = \tau s_1 \cdots \hat{s}_{i_1} \cdots \hat{s}_{i_q} \cdots s_{d-1}$, where the hat means we omit the corresponding term in the reduced expression for $t_{e_d}$.  Then we have
\begin{align*}
w &= \tau  ~ s_{1} \cdots \hat{s}_{i_1} \cdots \hat{s}_{i_q} \cdots s_{d-1} \\
&= t_{e_d} ~ (d ~ d-1 \cdots 1) ~ s_{1} \cdots \hat{s}_{i_1} \cdots \hat{s}_{i_q} \cdots s_{d-1} \\
&= t_{e_d} ~ (d ~ i_q ~ \cdots ~ i_1).
\end{align*}
Thus, in view of (a), we have $S(w) = \{i_1,~ \cdots, i_q, ~ d \}$.  Parts (c) and (d) now follow easily.

For part (e), assume $w \in {\rm Adm}^G(\mathcal O_\nu)$.  Then,  in addition to the equality
$$
\ell(t_{\mu_0}) - \ell(w) = |S(w)|-1
$$
we just proved, we also have 
$$\ell^M(t_{\nu}) - \ell^M(w) = |S^M(w)|-1.$$ 
Here we use the alcove specified in (b) to define both the Bruhat order on $\widetilde{W}_M$ coming into the definition of the subset $S^M(w) \subseteq W_M \nu$, and the length function $\ell^M$ on $\widetilde{W}_M$.  Now since $w \in {\rm Adm}^G(\mathcal O_\nu)$, it is clear that $|S(w)| = |S^M(w)|$ hence
\begin{equation} \label{ell_diff}
\ell(t_{\mu_0}) - \ell(w) = \ell^M(t_\nu) - \ell^M(w).
\end{equation}
Then comparing the formulas (\ref{Ko_fcn_GLd}) for $G$ and $M$ proves the desired equality.
\end{proof}


\subsection{The tori $T_S$ and $^wT$} \label{T_S-wT_subsec}

For $w \in {\rm Adm}(\mu_0)$ write $S := S(w)$ for the set of critical indices for $w$.  If we write $w = t_{e_m}(m \, m_{k-1} \, \cdots \, m_1)$, as in Lemma \ref{general_form}, then Lemma \ref{main_comb_prop} (a) says that $S = \{m_1, m_2, \dots, m_{k-1}, m \}$.  We define the subtorus
$$
T_S \subset T
$$
by requiring that $(t_1, \cdots, t_d) \in T$ belongs to $T_S$ if and only if $t_i =1$ for all 
$i \notin S$.  Thus
$$
X_*(T_S) = \bigoplus_{j \in S} \mathbb Z e_j.
$$

We also define the free abelian subgroup $L_w \subset X_*(T)$ to be the subgroup generated by the elements of the form $w(\nu) - \nu$, for $\nu$ ranging over $X_*(T)$.  Here $w(\nu) : = e_m + \bar w(\nu)$, if $w = t_{e_m} \bar w $ for $\bar w \in W$.

\begin{lemma} \label{T_S_vs_wT}
The quotient $X_*(T)/L_w$ is torsion-free, hence there is a unique quotient torus $^wT$ of $T$ whose cocharacter lattice is given by
$$
X_*(\, ^wT) = X_*(T)/L_w.
$$
There is a natural identification $L_w = \bigoplus_{j \in S(w)} \mathbb Z e_j$.  Hence $^wT$ is canonically isomorphic to $T/T_{S(w)}$.
\end{lemma}

\begin{proof}
For $j \in S(w)$, we have $w(\bar{\omega}_{j}) - \bar{\omega}_{j} = e_j$ (Lemma \ref{w_leq_e_j}).  Hence we have a canonical identification
\begin{equation*}
\mathbb Z^d = \bigoplus_{j \notin S(w)} \mathbb Ze_j ~ \bigoplus ~ \bigoplus_{j \in S(w)} \mathbb Z(w(\bar{\omega}_{j})- \bar{\omega}_{j}).
\end{equation*}
and the second summand is contained in $L_w$.  

\medskip

\noindent {\bf Claim:} $L_w \cap \oplus_{j \notin S(w)} \mathbb Z e_j = (0)$.

\smallskip

The claim implies that $L_w = \oplus_{j \in S(w)} \mathbb Z(w(\bar{\omega}_{j}) - \bar{\omega}_{j})$, which in turn implies the lemma.

To prove the claim, write $\nu = (\nu_1, \cdots, \nu_d)$.  A calculation shows that 
$$
w(\nu) - \nu = (\nu_{m_2}-\nu_{m_1})e_{m_1} + (\nu_{m_3}- \nu_{m_2})e_{m_2} + \cdots + (1+ \nu_{m_1} - \nu_{m}) e_{m},
$$
and clearly a sum of such elements can belong to $\oplus_{j \notin \{m_1, \dots, m_{k-1}, m \}} \mathbb Z e_j$ only if it is zero.
\end{proof}
Thus, in summary, for $S=S(w)$:
\begin{equation} \label{X_*(T_S)=}
X_*(T_S) = \langle w(\lambda) - \lambda ~ | ~ \lambda \in X_*(T) \rangle = 
\bigoplus_{j \in S} \mathbb Ze_j.
\end{equation}

\subsection{The functions $\delta(w,\chi)$ and $\delta^1(w,\chi)$}\label{deltadef}

We fix a character $\chi = (\chi_1, \dots, \chi_d)$ on $T(\mathbb F_p)$, and an element $w \in {\rm Adm}(\mu_0)$.

Let $S = S(w) = \{ i_1, \dots, i_h \}$ as in subsection \ref{crit_def}, and define $T_S \subset T$ as in subsection \ref{T_S-wT_subsec}.  Let $T^1_S \subset T_S$ denote the kernel of
\begin{equation}\label{dethom}
\begin{aligned}
T_S &\rightarrow \mG_m \\
(t_{i_1}, \dots, t_{i_h}) &\mapsto t_{i_1} \cdots t_{i_h}.
\end{aligned}
\end{equation}

The homomorphism (\ref{dethom}) also defines the $T_S(\mathbb F_p)$-module  $\bar{\mathbb Q}_\ell[\mu_{p-1}]$. The character $\chi$ gives rise to a disjoint decomposition 
\begin{equation}\label{decchi}
\{ 1, 2, \dots, d \} = \mathcal O_1 \amalg \cdots \amalg \mathcal O_t,
\end{equation}
given by the equivalence relation $\sim$, where $i \sim j$ if and only if $\chi_i = \chi_j$. 

\begin{lemma} \label{delta_equiv}
The following conditions on a pair $(w,\chi)$ are equivalent:
\begin{enumerate}
\item[{\rm (i)}] the $\chi$-isotypical component of ${\rm Ind}^{T(\mathbb F_p)}_{T_S(\mathbb F_p)} \bar{\mathbb Q}_\ell[\mu_{p-1}]$ is non-zero;
\item[{\rm (ii)}] $\chi |_{T_S(\mathbb F_p)}$ appears in $\bar{\mathbb Q}_\ell[\mu_{p-1}]$;
\item[{\rm (iii)}] $\chi_{i_1} = \cdots = \chi_{i_h}$;
\item[{\rm (iv)}] $S(w) \subseteq \mathcal O_i$, for some $i$;
\item[{\rm (v)}] $\chi$ is trivial on $T^1_S(\mathbb F_p)$.
\end{enumerate}
\end{lemma}

\begin{proof} Frobenius reciprocity yields ${\rm (i)} \Leftrightarrow {\rm (ii)}$.  The other equivalences are immediate.
\end{proof}
We define $\delta(w,\chi) \in \{0,1\}$ by
\begin{equation} \label{delta_defn}
\delta(w, \chi) := \begin{cases} 1, & \mbox{if (i-v) above hold for } (w, \chi)\\
0, & \mbox{otherwise}. \end{cases}
\end{equation}

\begin{cor} \label{chi=chi^w}
If $\delta(w, \chi) = 1$, then $^{w}\chi = \chi$.
\end{cor}

\begin{proof}
This follows from (iii) in combination with Proposition \ref{main_comb_prop}, (b).
\end{proof}

Now we define 
\begin{equation} \label{delta1_defn}
\delta^1(w, \chi) := \begin{cases} 1, & \mbox{if $\chi_j = {\rm triv}, \, \forall j \in S(w)$} \\
0, & \mbox{otherwise}. \end{cases}
\end{equation}

The following statement is obvious. 

\begin{lemma} \label{chi_S_triv}
A character $\chi \in T(\mathbb F_p)^\vee$ satisfies $\delta^1(w,\chi) =1$ if and only if $\chi$ is trivial on $T_S(\mathbb F_p)$.
\end{lemma}

\subsection{An extension of $\chi$}

\begin{lemma} \label{chi_extn}
Suppose $\delta^1(w, \chi) = 1$, and let $k_{r'} \supseteq \mathbb F_p$ be any finite extension.  Then $\chi$ possesses an extension to a character $\widetilde{\chi}$ on $T(k_{r'}) \rtimes \langle w \rangle$ which is trivial on $T_S(k_{r'})$ and $\langle w\rangle$.
\end{lemma}

\begin{proof}
This is straightforward, using (\ref{X_*(T_S)=}),  along with Lemma \ref{chi_S_triv}.
\end{proof}

\section{Trace of Frobenius on nearby cycles}\label{tracefrob_sec}

\subsection{Definition of the functions $\phi_{r,0}$ and $\phi_{r,\chi}$}

We fix $r$ and write $q$ for $p^r$.  Recall $\mu^* = -w_0(\mu)$, where $\mu$ is the Drinfeld type cocharacter which corresponds to $\mu_0 = (1,0^{d-1})$ under the isomorphism $G_{\mathbb Q_p} = {\rm GL}_d \times \mG_m$.  

Let $\phi_{r,0} = k_{\mu^*,r}$, where $k_{\mu^*,r} := q^{\langle \rho, \mu^*\rangle} z_{\mu^*,r}$ is the Kottwitz function. Here $z_{\mu^*,r} \in \mathcal Z(G_r,I_r)$ is the Bernstein function as in the Appendix, section \ref{nonst}, and  $\rho$ is the half-sum of the $B$-positive roots, so that $\langle \rho, \mu^*\rangle = \frac{d-1}{2}$. 

By \cite{H05}, the function $\phi_{r,0}$ is known to be ``the'' test function for the $\A_0$ moduli problem, in the sense of the Introduction, i.e., 
$\phi_{r,0}$ satisfies the analogue of (\ref{count}) below,  relative to $R\Psi_0$ instead of $R\Psi_\chi$.  In fact, this follows  from section \ref{counting_pts_sec}, by taking $\chi={\rm triv}$ to be the trivial character.  The function $\phi_{r,0}$ also computes the semi-simple trace of Frobenius function for  the nearby cycles $R\Psi_0$, see Proposition \ref{A_0_tr_Fr} below.

At this point we introduce a change in the notation used in sections \ref{moduli_sec}--\ref{comb_sec}, by letting $T = \mathbb G_m^d \times \mathbb G_m$ denote the maximal torus in $G_{\mQ_p} = {\rm GL}_d \times \mathbb G_m$ whose first factor is the usual diagonal torus.  We will only consider  characters $\chi$ on $T(\mF_p)$ which are trivial on the $\mathbb G_m$-factor.  It is obvious that for such $\chi$ and for $w \in \widetilde{W}$, we may define the symbols $\delta(w, \chi)$ and $\delta^1(w,\chi)$ as in (\ref{delta_defn}) and (\ref{delta1_defn}).
 
Now fix a character $\chi \in T(\mF_p)^\vee$ which is trivial on the $\mathbb G_m$-factor.  Recall from  section \ref{notation_sec} the Hecke algebra $\mathcal H(G_r, I_r, \chi_r)$ and its center $\mathcal Z(G_r,I_r, \chi_r)$, where  $\chi_r=\chi\circ N_r$,  as in the Introduction. We want to define a function $\phi_{r,\chi} \in \mathcal Z(G_r,I_r, \chi_r)$ with  analogous properties to those of $\phi_{r, 0}$.  Recall that $I^+_r \subset I_r$ denotes the pro-unipotent radical of $I_r$, and that we have the decomposition 
$$T(k_r) \rtimes \widetilde{W} = I^+_r \backslash G_r /I^+_r.$$
Here, as in section \ref{notation_sec}, we embed $t_\lambda \bar w \in X_*(T) \rtimes W = \widetilde{W}$ into 
$N_G(T)(L_r)$ by using permutation  matrices to represent the elements ${\bar w}$ and by sending $t_\lambda$ to $\lambda(\varpi)$, where $\varpi = p$.  Also, we use Teichm\"{u}ller representatives to regard $T(k_r)$ as a subgroup of $T(\mathcal O_{L_r})$, and thus of $I_r$.  

The function $\phi_{r,\chi}$ will be defined by specifying its values on representatives $tw$, for $t \in T(k_r)$, $w \in \widetilde{W}$.  Since $t$ and $w$ do  not commute in general, this will only make sense in the situation where $\delta(w,\chi) = 1$.  The key point is the following lemma.

\begin{lemma} \label{fcn_on_tw}  Suppose $\delta(w, \chi) = 1$.  Having fixed the embedding $T(k_r) \rtimes \widetilde{W} \hookrightarrow G_r$ as above, for each $t$ there is a unique function $\psi \in \mathcal H(G_r,I_r, \chi_r)$ supported on $I_r tw I_r$ and having value 1 at $tw$.
\end{lemma}

\begin{proof}
On $I_r tw I_r$, we define $\psi$ by the formula $\psi(i_1 t w i_2) = \chi^{-1}(t_1)^{-1} \chi^{-1}(t_2)$, where $i_j = t_j i^+_j \in T(k_r)I^+_r$ for $j = 1,2$.  We need only check this is well-defined, and for this it is enough to check that $i_1 tw i_2 = tw$ implies that $\chi^{-1}(t_1)\chi^{-1}(t_2) = 1$.  But
$i_1 tw i_2 = tw$ only if $t_1 = \, ^wt_2^{-1}$, and so the result holds because $\chi(t_2) = \chi(\,^wt_2)$ (using Corollary \ref{chi=chi^w}).
\end{proof}


\begin{defn} \label{test_fcn}  
For any $r \geq 1$,  let $\phi_{r,\chi}$  be the unique function in $\mathcal H(G_r, I_r, \chi_r)$ determined by
$$
\phi_{r,\chi}(tw) = \delta^1(w^{-1},\chi) \, \chi^{-1}_r(t)\, k_{\mu^*,r}(w),
$$
for all $(t,w) \in T(k_r) \rtimes \widetilde{W}$. 
\end{defn} 
Note that $k_{\mu^*,r}(w)$ is supported on $w \in {\rm Adm}(\mu^*)$ (cf. Appendix), and for such elements $\delta^1(w^{-1},\chi)$ is defined. Therefore
the previous Lemma shows that the right hand side is well-defined. 

The Iwahori-level function $\phi_{r,0}$ can be seen as the  special case corresponding to  the trivial character $\chi={\rm triv}$.

Here is how we will proceed: at the end of the present section we will show that $\phi_{r,\chi}$ computes the semi-simple trace of Frobenius on $R\Psi_\chi$.  Then,  in Theorem \ref{counting_points}, we will show that this function is ``the'' test function for $R\Psi_\chi$, in the sense that $\phi_{r,\chi}$ satisfies (\ref{count}).   Finally, in Proposition \ref{test_fcn_image}, we shall identify $\phi_{r,\chi}$ using Hecke algebra isomorphisms,  and thereby show that it belongs to the center 
$\mathcal Z(G_r,I_r,\chi_r)$ of $\mathcal H(G_r, I_r,\chi_r)$.

\subsection{Trace of Frobenius on $R\Psi_0$}

We start by recalling some results on ${\rm Tr}^{ss}(\Phi^r_\fp, R\Psi_0)$ proved in \cite{H1b} and \cite{HN}.

\'{E}tale locally at a point in its special fiber, $\A_0$ is isomorphic to the local model \cite{RZ}, which we will denote by ${\bf M}^{\rm loc}$.  The $\mathbb Z_p$-scheme ${\bf M}^{\rm loc}$ can be identified with a closed subscheme of a certain deformation of the affine flag variety of ${\rm GL}_d \times \mathbb G_m$ to its affine Grassmannian (cf. \cite{HN}).  The special fiber of ${\bf M}^{\rm loc}$ is a union of the Iwahori-orbits ${\rm FL}_w$ in the affine flag variety associated to $w \in {\rm Adm}(\mu)$, 
\begin{equation}\label{dec_FL}
{\bf M}^{\rm loc} \otimes  {\mF}_p =  {\rm FL}_\mu := \coprod_{w \in {\rm Adm}(\mu)} {\rm FL}_w.
\end{equation}
A point $x \in \A_0$ belongs to the KR-stratum $\A_{0,w}$ if and only if it is associated via the {\em local model diagram} to a point $\tilde{x}$ lying in $ {\rm FL}_w$.  By  \cite{H1b} or \cite{HN}, the nearby cycles sheaf $R\Psi^{\rm loc} := R\Psi^{{\bf M}^{\rm loc}}(\bar{\mathbb Q}_\ell)$ is Iwahori-equivariant and corresponds to the function $q^{\langle \rho, \mu \rangle } z_{\mu}$ under the function-sheaf dictionary.  
This means that for $x \in \A_{0,w}(k_r)$ we have
\begin{equation} \label{RPsi_0_fs}
{\rm Tr}^{ss}(\Phi^r_{\mathfrak p}, (R\Psi_0)_x) = 
{\rm Tr}^{ss}(\Phi^r_{\mathfrak p}, R\Psi^{\rm loc}_{\tilde{x}}) = q^{\langle \rho, \mu \rangle} z_{\mu,r}(w).
\end{equation}

On the other hand, there is a general identity $z_{\mu^*,r}(w^{-1}) = z_{\mu,r}(w)$, where $\mu^* = -w_0(\mu)$ for any dominant $\mu$ and any $w \in \widetilde{W}$ (cf. \cite{HKP}). Thus we have

\begin{prop} \label{A_0_tr_Fr}  For $x \in \A_{0,w}(k_r)$, 
$$
{\rm Tr}^{ss}(\Phi^r_{\mathfrak p}, (R\Psi_0)_x) = q^{\langle \rho, \mu^* \rangle} z_{\mu^*,r}(w^{-1}) = \phi_{r,0}(w^{-1}).
$$\qed
\end{prop}

\subsection{Trace of Frobenius on $R\Psi_\chi$}

We will now prove the analogue of Proposition \ref{A_0_tr_Fr} for arbitrary $\chi$ (which is trivial on the $\mathbb G_m$-factor).  First, we need a preliminary construction.  Let $x \in \A_{0,w}(k_r)$.  We will construct from $x$ an element $t_x \in T^S(\mathbb F_p)$, where $S = S(w)$.  (In keeping with the notation change made at the beginning of this section, one should think of $T^S$ as the quotient of $T = \mathbb G_m^d \times \mathbb G_m$ by the torus $(\prod_{i \in S}\mathbb G_m) \times \mathbb G_m$.)

Choose a sufficiently large extension $k_{r'} \supset k_r$ and a point $x' \in \mathcal A_{1,w}(k_{r'})$ lying above 
$x$.  Then its Frobenius translate $\Phi^r_{\mathfrak p}(x')$ also lies above $x$.  
Since $\mathcal A_{1,w} \rightarrow \mathcal A_{0,w}$ is a $T^S(\mathbb F_p)$-torsor, we have
\begin{equation} \label{t_x_def_eqn}
\Phi_{\mathfrak p}^r(x') = x't_x
\end{equation}
for a unique element $t_x \in T^S(\mathbb F_p)$.  It is easy to check that $t_x$ is independent of the choice of $x'$ and $r'$.  

The element $t_x$ enters into the computation of ${\rm Tr}^{ss}(\Phi^r_\fp, R\Psi_{\chi,x})$ via the following lemma, cf.~\cite{D}, $\S1.4$.

\begin{lemma} \label{sommes_trig}   Suppose that $\delta^1(w,\chi) = 1$ and let $\bar \chi : T^S(\mathbb F_p) \rightarrow \bar{\mathbb Q}_\ell^\times$ denote the character derived from $\chi$.  Let $\bar{\mathbb Q}_{\ell, \bar \chi} = (\pi_{S,*}(\bar{\mathbb Q}_\ell))_{\chi}$ denote the push-out of $\pi_{S,*}(\bar{\mathbb Q}_\ell)$ along $\bar \chi$ {\em (cf. Theorem \ref{main_geom_thm})}.  Then 
$${\rm Tr}(\Phi^r_{\mathfrak p}, (\bar{\mathbb Q}_{\ell, \bar \chi})_x) = \chi(t_x).$$\qed
\end{lemma}

Putting together Theorem \ref{main_geom_thm}, Proposition \ref{A_0_tr_Fr}, and Lemma \ref{sommes_trig}, we get
\begin{align*}
{\rm Tr}^{ss}(\Phi^r_\fp, (i^*_SR\Psi_{\chi})_x) &= \delta^1(w,\chi) \cdot \chi(t_x) \cdot {\rm Tr}^{ss}(\Phi^r_\fp, i^*_S(R\Psi_0)_x) \\
&= \delta^1(w, \chi) \cdot \chi_r(t) \cdot k_{\mu^*,r}(w^{-1})\ ,
\end{align*}
where $t \in T(k_r)$ is any element such that $N_r(t)$ projects to $t_x \in T^S(\mF_p)$.
This implies the following result.
\begin{thm} \label{chi_tr_Fr} Let $x \in \A_{0,w}(k_r)$ and write $S = S(w)$.  
Then for any element $t \in T(k_r)$ such that $N_r(t)$ projects to $t_x \in T^S(\mF_p)$, we have
$$
{\rm Tr}^{ss}(\Phi^r_\fp, (R\Psi_{\chi})_x) = \phi_{r,\chi}(t^{-1}w^{-1}).$$\qed
\end{thm}

\section{Counting points formula} \label{counting_pts_sec}

\subsection{Statement of the formula}

Fix $r \geq 1$ and let $k_r = \mathbb F_{p^r}$.  Fix $\chi \in T(\mathbb F_p)^\vee$.  We recall that $\mathcal A_0 = \mathcal A_{0,K^p}$, where $K^p$ is a sufficiently small compact open subgroup of $G(\mathbb A^p_f)$, so that $\mathcal A_0$ is a generically smooth $\mathbb Z_p$-scheme.  Our goal is to prove a formula for the semi-simple Lefschetz number
\begin{equation}
{\rm Lef}^{ss}(\Phi_{\mathfrak p}^r, \chi) := \sum_{x \in \mathcal A_0(k_r)} {\rm Tr}^{ss}(\Phi^r_{\mathfrak p}, R\Psi_{\chi, x}).
\end{equation}

We shall follow the strategy of Kottwitz \cite{K90}, \cite{K92}.  

\begin{thm}  \label{counting_points}  Let $\phi_{r,\chi}$ be the function in $\mathcal H(G_r, I_r, \chi_r)$ defined in {\em Definition \ref{test_fcn}}.  Then
\begin{equation} \label{count}
{\rm Lef}^{ss}(\Phi_\mathfrak p^r, \chi) =  \sum_{(\gamma_0; \gamma, \delta)} c(\gamma_0;\gamma,\delta) \, {\rm O}_\gamma(f^p) \, {\rm TO}_{\delta \sigma}(\phi_{r,\chi}).
\end{equation}
\end{thm}
The terms on the right hand side have the same meaning as in \cite{K92}, p.~442\footnote{The factor $|{\rm ker}^1(\mathbb Q,G)|$ does not appear here, since our assumptions on $G$ guarantee that this number coincides with $|{\rm ker}^1(\mathbb Q,Z(G))|$ and that this number is 1, cf.\ loc.~cit., \S7.}.  Recall that $f^p = 1_{K^p}$, and that the twisted orbital integral is defined as
$$
{\rm TO}_{\delta \sigma}(\phi) = \int_{G_{\delta \sigma}\backslash G(L_r)} \phi(x^{-1} \delta \sigma(x)) \, d\overline{x}$$ 
where $G_{\delta \sigma} = \{ g \in G(L_r) ~ | ~ g^{-1}\delta \sigma(g) = \delta \}$ and where the quotient measure $d\overline{x}$ is determined as in \cite{K92}, except that here we use the measure on $G(L_r)$ which gives $I_r$ measure 1.  

Exactly as in \cite{K90}, the sum in (\ref{count}) is indexed by equivalence classes of triples $(\gamma_0; \gamma, \delta) \in G(\mathbb Q) \times G(\mathbb A^p_f) \times G(L_r)$ for which the Kottwitz invariant $\alpha(\gamma_0; \gamma, \delta)$ is defined and trivial (see \cite{K92}, p.~441).  The term $c(\gamma_0;\gamma,\delta)$ is defined as follows.  Let $I$ denote\footnote{There should be no confusion arising from the fact that elsewhere in the paper the symbol $I$ denotes an Iwahori subgroup. } the inner $\mathbb Q$-form of $I_0 := G_{\gamma_0}$ specified in \cite{K90}, $\S3$.  (In what follows, $I$ will be realized as the group of self-$\mathbb Q$-isogenies of a triple $(A',\lambda', i')$ coming from a point in $\mathcal A_0(k_r)$.)    We let $c(\gamma_0;\gamma,\delta) = {\rm vol}(I(\mathbb Q)\backslash I(\mathbb A_f)) \, |{\rm ker}[{\rm ker}^1(\mathbb Q,I_0) \rightarrow {\rm ker}^1(\mathbb Q, G)]|$.   The measures on $I(\mathbb A^p_f)$ and $I(\mathbb Q_p)$ used to define ${\rm vol}(I(\mathbb Q)\backslash I(\mathbb A_f))$ are compatible with the ones on $G_{\gamma}(\mathbb A^p_f)$ and $G_{\delta \sigma}$ used to form ${\rm O}_{\gamma}(f^p)$ and ${\rm TO}_{\delta \sigma}(\phi_{r,\chi})$, in the sense of \cite{K90}, $\S3$.  The expression (\ref{count}) is thus a {\em rational number} independent of the choices of these measures.

\subsection{Description of the $k_r$-points in a $\mathbb Q$-isogeny class in $\mathcal A_0$}



We will prove Theorem \ref{counting_points} by determining the contributions of the individual  $\mathbb Q$-isogeny classes to ${\rm Lef}^{ss}(\Phi_{\mathfrak p}^r, \chi)$.  We will perform the calculation on each KR-stratum $\mathcal A_{0,w}$ separately.  We will first describe the $k_r$-points of $\mathcal A_0$ lying in a single $\mathbb Q$-isogeny class.

Fix a point $(A'_\bullet, \lambda', i', \overline{\eta}') \in \mathcal A_0(k_r)$.  Set $(A',\lambda',i') := (A'_0,\lambda'_0, i'_0)$, and regard this as a {\em $c$-polarized virtual $B$-abelian variety over $k_r$ up to isogeny}, using the terminology of \cite{K92}, $\S14$.  Let $I$ denote the $\mathbb Q$-group of automorphisms of $(A', \lambda', i')$.  Thus, using the notation  $k = \bar{k_r}$ and $\overline{A}' = A'\otimes_{k_r}k$, then $I(\mathbb Q)$ consists of the $\mathbb Q$-isogenies $\overline{A}' \rightarrow \overline{A}'$ which commute with $i'$, preserve $\lambda'$ up to a scalar in $\QQ^\times$, and commute with $\pi_{A'}$, the absolute Frobenius morphism of $A'$ relative to $k_r$ (see \cite{K92}, \S10).  Let ${\rm Isog}_0(A',\lambda',i')$ (resp. ${\rm Isog}_{0,w}(A',\lambda',i')$) denote the set of points $(A_\bullet, \lambda, i, \overline{\eta}) \in \mathcal A_0(k_r)$ (resp. $\mathcal A_{0,w}(k_r)$) such that $(A_0, \lambda, i)$ is isogenous to $(A',\lambda', i')$.  Following the usual strategy, we need to prove a bijection
\begin{equation} \label{Isog_0=}
{\rm Isog}_0(A',\lambda',i') = I(\mathbb Q) \backslash[Y^p \times Y_p]
\end{equation}
for some appropriate sets $Y^p$ and $Y_p$.

For the moduli problem $\mathcal A_0$, such a bijection is explained in \cite{H05}, $\S11$, which we now review.  Associated to $(A',\lambda',i')$ is an $L_r$-isocrystal $(H'_{L_r},\Phi)$.  The precise definition of $H' = H(A')$ is given in \cite{K92}, $\S10$, but roughly, $H'=H(A')$ is the 
$W(k_r)$-dual of $H^1_{\rm crys}(A'/W(k_r))$ and $\Phi$ is the $\sigma$-linear bijection on $H'_{L_r}$ such that $p^{-1}H' \supset \Phi H' \supset H'$ (i.e. $\Phi = V^{-1}$, the inverse of the Verschiebung operator on $H(A')$ which is induced by duality from the Frobenius on $H^1_{\rm crys}(A'/W(k_r))$; see the formulas (14.4.2-3) in \cite{H05}).  We have isomorphisms of skew-Hermitian $\mathcal O_B \otimes \mathbb A^p_f$- (resp. $\mathcal O_B \otimes L_r$-) modules
\begin{align} 
V \otimes \mathbb A^p_f &= H_1(A', \mathbb A^p_f) \label{V_A^p_f} \\
V \otimes  L_r &= H(A')_{L_r} \label{V_L_r}
\end{align}
which we fix (the end result of our calculation will be independent of these choices). 

Let us describe the set $Y^p$.  Using (\ref{V_A^p_f}), transport the action of $\pi^{-1}_{A'}$ on $H_1(A',\mathbb A^p_f)$ over to the action of an element $\gamma \in G(\mathbb A^p_f)$ on $V \otimes \mathbb A^p_f$.  A $\mathbb Q$-isogeny $\xi \in {\rm Isom}((A_0,\lambda,i),(A',\lambda',i'))_\mQ$ takes a $K^p$-level structure on $A_0$ to an element $yK^p \in G(\mathbb A^p_f)/K^p$.  More precisely, having fixed (\ref{V_A^p_f}), there is a map ${\rm Isog}_0(A',\lambda',i') \rightarrow I(\mathbb Q)\backslash G(\mathbb A^p_f)/K^p$ of the form 
$$
(A_\bullet, \lambda, i, \overline{\eta}) \mapsto \xi(\overline{\eta})$$ 
(use (\ref{V_A^p_f}) to regard $\xi(\overline{\eta}) \in G(\mathbb A^p_f)/K^p$).  
The level structure is Galois-invariant if and only if $y^{-1}\gamma y \in K^p$.  Thus we define
$$
Y^p := \{ y \in G(\mathbb A^p_f)/K^p ~ | ~ y^{-1} \gamma y \in K^p \}.
$$
A theorem of Tate gives an isomorphism $I(\mathbb A^p_f) \cong G(\mathbb A^p_f)_\gamma$ (cf.~\cite{K92}, Lemma 10.7), and the latter acts on $Y^p$ by left multiplications.  So $I(\mathbb Q)$ acts on $Y^p$ via $I(\mathbb Q) \hookrightarrow I(\mathbb A^p_f)$.

Next we turn to the set $Y_p$.  Using the terminology of \cite{RZ}, we shall describe it as the set of multichains $H_\bullet$ of $\mathcal O_B \otimes W(k_r)$-lattices in $H(A')_{L_r} = V_{L_r}$.  Let 
$$
\widetilde{\Lambda}_\bullet = \Lambda_\bullet \oplus \Lambda^*_\bullet
$$
denote the ``standard'' self-dual multichain of $\mathcal O_B \otimes \mathbb Z_p$-lattices, in the sense of \cite{RZ}, which is constructed in \cite{H05}, $\S5.2.3$.  We may assume that $\widetilde{\Lambda}_\bullet$ corresponds to our base alcove ${\bf a}$ in the building of $G(\mQ_p)$. 
  Let $Y_p^\circ$ denote the set of all multichains  of $\mathcal O_B \otimes W(k_r)$-lattices $H_\bullet$  contained in $V_{L_r}$, of type $(\widetilde{\Lambda}_\bullet)$  and self-dual\footnote{In particular, these multichains are {\em polarized} in the sense of \cite{RZ}, Def.~3.14.} up to a scalar in $L_r^\times$.  The group $G(L_r)$ acts transitively on these objects, and thus we can identify $Y_p^\circ$ with $G(L_r)/I_r$.  To see this,  we use \cite{RZ}, Theorem 3.11, 3.16, along with \cite{P} (see \cite{H05}, Thm.~6.3).

Consider the set
$$
\{(A_\bullet, \lambda, i, \xi) \}
$$
consisting of chains of polarized $\mathcal O_B$-abelian varieties over $k_r$, up to $\mathbb Z_{(p)}$-isogeny, equipped with a $\mathbb Q$-isogeny of polarized $\mathcal O_B$-abelian varieties $\xi: A_0 \rightarrow A'$.  The group $I(\mathbb Q)$ acts on these objects by acting on the isogenies $\xi$.  The covariant functor $A \mapsto H(A)$ takes this set to the set of chains of $\mathcal O_B \otimes W(k_r)$-lattices in $H(A')_{L_r}$, equipped with Frobenius and Verschiebung endomorphisms.  In fact, $(A_\bullet, \lambda, i, \xi) \mapsto \xi(H(A_\bullet))$ gives an isomorphism
$$
\{(A_\bullet, \lambda, i, \xi) \} ~ \widetilde{\rightarrow} ~ Y_p,
$$
where by definition $Y_p$ is the set of $H_\bullet \in Y^\circ_p$ such that $p^{-1}H_i \supset \Phi H_i \supset H_i$ for each $i$, and $\sigma^{-1}(\Phi H_i/H_i) $ satisfies the determinant 
condition.  The determinant condition arises because we imposed it on ${\rm Lie}(A_i)$ in (\ref{kocon}), and we have $\Phi H(A_i)/H(A_i) = \sigma({\rm Lie}(A_i))$ (cf. \cite{H05}, $\S14$).   The group $I(\mathbb Q)$ acts on $Y_p$ in a natural way. 

Let us rephrase the determinant condition.  Using (\ref{V_L_r}) we write $\Phi = \delta \sigma$ for $\delta \in G(L_r)$.  Lemma 10.8 of \cite{K92} shows that (\ref{V_L_r}) induces an isomorphism 
$I(\mathbb Q_p) \cong G_{\delta \sigma}$, and hence an embedding $I(\mathbb Q) \hookrightarrow G(L_r)$.  
Thus for any element $g \in I(\mathbb Q) \backslash G(L_r)/I_r$, the double coset $I_r g^{-1} \delta \sigma(g) I_r$ is well-defined.  

We use the symbol $x$ to abbreviate $(A_\bullet, \lambda,i,\overline{\eta}) \in {\rm Isog}_0(A',\lambda',i')$.  Choose any $\mathbb Q$-isogeny $\xi \in {\rm Isom}((A_0,\lambda,i),(A',\lambda',i'))_\mQ$.  Write
\begin{equation} \label{1st_x_to_g}
\xi(H(A_\bullet)) = g\widetilde{\Lambda}_{\bullet, W(k_r)}
\end{equation}
for some $g \in G(L_r)/I_r$.  The image of $g$ in $I(\mathbb Q) \backslash G(L_r)/I_r$ is independent of the choice of $\xi$.  Thus we have a well-defined map
\begin{align} \label{x_to_g}
{\rm Isog}_0(A',\lambda',i') ~ &\rightarrow ~ I(\mathbb Q) \backslash G(L_r)/I_r \\
x \,\,\, &\mapsto \,\,\,[g]. \notag 
\end{align}
For each index $i$ of the chain $\widetilde{\Lambda}_\bullet$, let $K_i$ denote the fixer of $\widetilde{\Lambda}_i$ in $G(L_r)$.  The determinant condition (\ref{kocon}) at $i$ gives the relative position of the lattices $H_i$ and $\Phi H_i$, and may be written
\begin{equation} \label{inv=mu*}
{\rm inv}_{K_i}(\widetilde{\Lambda}_{i, W(k_r)}, \, g^{-1}\delta\sigma(g) \widetilde{\Lambda}_{i,W(k_r)}) = \mu^*,
\end{equation}
where $\mu^* := -w_0(\mu)$ corresponds under Morita equivalence to $\mu_0^* := -w_0(\mu_0) = (0^{d-1},-1)$ (cf. \cite{H05}, $\S11.1$).  This says that $g^{-1}\delta \sigma(g) \in I_r w^{-1} I_r$ for some $w \in {\rm Perm}^G(\mu)$.  The equality ${\rm Adm}^G(\mu) = {\rm Perm}^G(\mu)$ holds (this translates under Morita equivalence to the analogous equality for ${\rm GL}_n$, which is known by \cite{KR}).  

For $w \in {\rm Adm}^G(\mu)$ define
\begin{equation}\begin{aligned} \label{defM_w}
{\bf M}_w(k_r) &= I_r w^{-1} I_r/I_r \\
{\bf M}_\mu(k_r) &= \coprod_{w \in {\rm Adm}^G(\mu)} {\bf M}_w(k_r).  
\end{aligned}
\end{equation}
(Note that ${\bf M}_w(k_r)={\rm FL}_{w^{-1}}(k_r)$, in the notation of (\ref{dec_FL}).)
The determinant condition on $H_\bullet$ can now be interpreted as:
\begin{equation}
g^{-1}\delta \sigma(g)\widetilde{\Lambda}_{\bullet,W(k_r)} \in {\bf M}_\mu(k_r).
\end{equation} 

We can thus identify
\begin{equation}
Y_p = \{ gI_r \in G(L_r)/I_r ~ | ~ g^{-1}\delta \sigma(g) I_r \in {\bf M}_\mu(k_r) \}.
\end{equation}
 Also,  for $w \in {\rm Adm}^G(\mu)$, define
\begin{equation}
Y_{p,w} = \{ gI_r \in G(L_r)/I_r ~ | ~ g^{-1} \delta \sigma(g) I_r \in {\bf M}_w(k_r) \}.
\end{equation}

Let us summarize:

\begin{lemma}  The map $(A_\bullet, \lambda, i, \overline{\eta}) \mapsto [\xi(\overline{\eta}), \xi(H(A_\bullet))]$ determines a bijection
\begin{equation}
{\rm Isog}_0(A',\lambda',i') ~ \widetilde{\rightarrow} ~ I(\mathbb Q)\backslash [Y^p \times Y_p]
\end{equation}
(here $\xi$ is any choice of $\mQ$-isogeny as above).   This map restricts to give a bijection
\begin{equation} \label{Isog_0w=}
{\rm Isog}_{0,w}(A',\lambda',i') ~ \widetilde{\rightarrow} ~ I(\mathbb Q)\backslash [Y^p \times Y_{p,w}].
\end{equation}
\end{lemma}
The final statement comes from an analysis of how the KR-stratum of $x$ can be recovered from the element $g^{-1}\delta \sigma(g)$.  This is explained in \cite{H05}, Lemma 11.1, noting that here $w^{-1}$ appears instead of $w$ because of differing conventions (here $\mu_0 = (1,0^{d-1})$; in loc.~cit. $\mu_0 = (0^{d-1},-1)$). We state this as the following lemma. 

\begin{lemma} \label{dR_vs_crys}{\em (\cite{H05}, Lemma 11.1)}
Fix $w \in {\rm Adm}(\mu)$.  Let $(A_\bullet, \lambda, i, \overline{\eta}) \in {\rm Isog}_0(A',\lambda',i')$.  Suppose that $\xi: A_0 \rightarrow A'$ is a choice of $\mathbb Q$-isogeny as above and set $\xi(H(A_\bullet)) = g\widetilde{\Lambda}_{\bullet, W(k_r)}$ for $g \in G(L_r)/I_r$.  Then the image $g \in I(\mathbb Q)\backslash G(L_r)/I_r$ is well-defined (independent of $\xi$) and $(A_\bullet, \lambda,i, \overline{\eta}) \in {\rm Isog}_{0,w}(A',\lambda',i')$ if and only if
$$
I_r g^{-1} \delta \sigma(g) I_r = I_r w^{-1} I_r.
$$\qed
\end{lemma} 

\subsection{The group-theoretic analogue $t_g$ of $t_x$}

Let $x = (A_\bullet, \lambda, i, \overline{\eta}) \in {\rm Isog}_{0,w}(A',\lambda',i')$.  Recall from (\ref{t_x_def_eqn}) that  $x$ determines  an element $t_x \in T^S(\mathbb F_p)$, where $S = S(w)$.

Now we need to give a more group-theoretic description of $t_x$, which will be denoted $t_g$, where $g \in I(\mathbb Q) \backslash G(L_r)/I_r$ is defined in (\ref{x_to_g}).  We first need a lemma.  Recall that Morita equivalence  is realized here by multiplying by the idempotent $e_{11} \in {\rm M}_d(\mathbb Z_p)$.  Below, we shall abuse notation,  in that operators induced via Morita equivalence (e.g. $\delta \sigma$) will be denoted by the same symbols.  Further, in our standard lattice chain $\widetilde{\Lambda}_{i} = \Lambda_i \oplus \Lambda^*_i$ we have (\cite{H05}, 5.2.3)
\begin{equation}
\Lambda_i = {\rm diag}((p^{-1})^i, 1^{d-i}) \, {\rm M}_d(\mathbb Z_p)
\end{equation}
and we use the same symbol to denote its Morita equivalent $e_{11}\Lambda_i$, namely
\begin{equation}
\Lambda_i = (p^{-1}\mathbb Z_p)^i \oplus \mathbb Z_p^{d-i}.
\end{equation}

\begin{lemma} \label{v_i_exists}  Suppose $g \in I(\mathbb Q) \backslash G(L_r)/I_r$ satisfies $g^{-1}\delta \sigma(g) \in I_r w^{-1} I_r$, and write $S(w) = S$.   Let $\widetilde{g} \in G(L_r)/I^+_r$ be any lift of $g$, and set for any $r'$ with $k_{r'} \supset k_r$, $M_{\bullet, {r'}} := e_{11}(\widetilde{g} \widetilde{\Lambda}_{\bullet, W(k_{r'})}) = \widetilde{g}\Lambda_{\bullet, W(k_{r'})}$.   
Then,  for any sufficiently large extension $k_{r'} \supset k_r$ and,  for each non-critical index $i$ (i.e., $i \notin S$),  there exists a non-zero $v_i \in M_{i,r'}/M_{i-1,r'}$ such that $(\delta \sigma)^{-1} (v_i) = v_i$  (and $v_i$ is uniquely 
determined up to a scalar in $\mathbb F_p^\times$).
\end{lemma}
The identity $(\delta \sigma)^{-1} (v_i) = v_i$ is meant here in the following sense: $(\delta \sigma)^{-1}$ preserves each $M_i$ and  for non-critical $i$ induces a bijection
$$
(\delta \sigma)^{-1}: M_i/M_{i-1} ~ \widetilde{\rightarrow} ~ (\delta \sigma)^{-1} M_i/(\delta \sigma)^{-1} M_{i-1} ~ \widetilde{\rightarrow} ~ M_i/M_{i-1}
$$
(the second arrow being induced by inclusion $(\delta \sigma)^{-1}M_i \subset M_i$), which sends $v_i$ to 
itself.

\begin{proof}
There exists $(A_\bullet, \lambda, i, \overline{\eta}) \in {\rm Isog}_{0,w}(A',\lambda',i')$ such that $M_\bullet = e_{11}H(A_\bullet) = M(X_\bullet)$, where $X_\bullet$ is the chain of $p$-divisible groups (\ref{chainpdiv}) and $M(X_i)$ is the covariant Dieudonn\'{e} module of $X_i$.  We regard this as an equality of chains in $e_{11}(V \otimes L_r) = L^d_r$  (via (\ref{V_L_r}) and Morita).  We may identify $M_{i,r'}/M_{i-1,r'}$ with the covariant Dieudonn\'{e} module $M(G_{i,k_{r'}})$.  Choose $k_{r'}$ such that for each non-critical index $i$, we have $G_{i,k_{r'}} \cong (\mathbb Z/p\mathbb Z)_{k_{r'}}$. The  covariant Dieudonn\'{e} module of $(\mathbb Z/p\mathbb Z)_{k_{r'}}$ is $(k_{r'}, F = 0, V = \sigma^{-1})$.  The Verschiebung operator $V = (\delta \sigma)^{-1}$ on $H_{L_r}$ induces the operator $\sigma^{-1}$ on $M(G_{i,k_{r'}}) = k_{r'}$, which contains a unique $\mathbb F_p$-line of $\sigma^{-1}$-fixed points.
\end{proof}

Now fix $g,\, \widetilde{g}, \, r'$, and $v_\bullet$ as in the lemma.    Let $e_i$ denote the $i$-th standard basic vector in $\mathbb Z_p^d$, and let $e_{i,r'}$ denote the image of $p^{-1}e_i$ in the quotient $\Lambda_{i, W(k_{r'})}/\Lambda_{i-1,W(k_{r'})}$.   Write
\begin{equation} \label{M=gt'}
(M_\bullet,v_\bullet) = \widetilde{g} t'(\Lambda_{\bullet, W(k_{r'})}, e_{\bullet,r'}),
\end{equation}
for a unique element $t' \in T^S(k_{r'})$.  Here and in what follows we are using Teichm\"{u}ller lifts to regard $T^S(k_{r'})$ as a subgroup of $T^S(\mathcal O_{L_{r'}})$.  

  By (\ref{M=gt'}), we have an equality (in the sense analogous to the previous identity $(\delta \sigma)^{-1} v_\bullet = v_\bullet$):
\begin{equation} \label{trans_Vv=v}
(\widetilde{g}t')^{-1} \, (\delta \sigma)^{-1} \, (\widetilde{g} t') \, e_{\bullet, r'} = e_{\bullet, r'}.
\end{equation}
Now we may define
\begin{equation} \label{def_t_g}
t_g := (\widetilde{g}t')^{-1} \, \sigma^r(\widetilde{g}t') = t'^{-1} \, \sigma^r(t') \in T^S(k_{r'}).
\end{equation}
It follows easily that $t_g$ depends only on $g \in I(\mathbb Q) \backslash G(L_r)/I_r$ and not on the choice of $r'$, $\widetilde{g}$,  or $v_\bullet$.

\begin{lemma} \label{t_g_lemma} Let $g$ be as in Lemma \ref{v_i_exists}.

\noindent {\rm (i)} We have $t_g \in T^S(\mathbb F_p)$.

\noindent {\rm (ii)} If $g$ comes from $x$ via (\ref{x_to_g}), then $t_x = t_g$.

\noindent {\rm (iii)}  Choose any lift $\widetilde{g}$  in $G(L_r)$ of $g \in I(\QQ)\backslash G(L_r)/I_r$;  if $\widetilde{g}^{-1}\delta \sigma(\widetilde{g}) \in I^+_r t w^{-1} I^+_r$, where $tw^{-1} \in T(k_r) \rtimes \widetilde{W}$ is embedded into $G(L_r)$ as in section \ref{notation_sec}, then $N_r(t)$ projects to $t_g^{-1}$ in $T^S(\mF_p)$.
\end{lemma}

\begin{proof}
Part (i).  Heuristically, $\widetilde{g}^{-1}\, \delta \sigma(\widetilde{g})$ is the inverse of ``Verschiebung in the fiber over $gI_{k_{r'}}$'', and $t_g$ measures the difference between two points in the fiber over $g \in G(L_r)/I_r$ in the moduli problem
$$
(\mbox{lattice chain; generators in non-critical indices, fixed by Verschiebung over the chain}).
$$
The group of deck transformations of the forgetful morphism -- the one forgetting the generators -- is the torus $T^S(\mathbb F_p)$.  Thus $t_g \in T^S(\mathbb F_p)$.  Rather than making these notions precise, we will take a more pedestrian approach.  

In view of (\ref{trans_Vv=v}), to show that  $t_g \in T^S(k_{r'})$ belongs to $T^S(\mathbb F_p)$ it suffices to show 
\begin{equation} \label{ratl_crit}
(\delta \sigma)^{-1}( \widetilde{g}t' t_g) e_{\bullet,r'} = (\widetilde{g}t' t_g) e_{\bullet, r'}.
\end{equation}
However, applying $\sigma^r$ to (\ref{trans_Vv=v}) and using $[\sigma^r, \delta \sigma ] = 0$ we obtain
$$
(\delta \sigma)^{-1}( \sigma^r(\widetilde{g}t')) e_{\bullet, r'} = \sigma^r(\widetilde{g}t') \,
e_{\bullet, r'},
$$
which implies (\ref{ratl_crit}) by (\ref{def_t_g}).

\smallskip

Part (ii).  This follows, using Morita equivalence, from subsection \ref{OT_Dieu_sec} below.

\smallskip

Part (iii).  It is enough to prove the following equivalent statement. Recall $\chi_r=\chi\circ N_r$. 

\smallskip

\noindent ${\rm (iii')}$: {\it For each $\chi$ with $\delta^1(w,\chi) =1 $, if $\widetilde{g}^{-1} \delta \sigma(\widetilde{g}) \in I^+_r t w^{-1} I^+_r$ then $\chi_r(t)= \chi(t_g^{-1})$.}

\smallskip

We first claim that this condition is independent of the choice of lift $\widetilde{g}$.  Write $\widetilde{g}^{-1} \delta \sigma(\widetilde{g}) = i^+_1 t w^{-1} i^+_2$ for some $i^+_1, i^+_2 \in I^+_r$.  Changing the choice of the lift $\widetilde{g} \in G(L_r)$ replaces $t$ by an element of the form $t_1t\, ^{w^{-1}}\!\!\!\sigma(t_1)^{-1}$ for some $t_1 \in T(k_r)$.   But then $\chi_r(t)$ is replaced by
$$
\chi_r(t_1 t \, ^{w^{-1}}\!\!\!\sigma(t_1)^{-1}) = \chi_r(t) \cdot \chi_r(t_1) \cdot \chi_r(\, ^{w^{-1}}\!\!\!\sigma(t_1))^{-1}.
$$
Since $^w\chi = \chi$ (Corollary \ref{chi=chi^w}) we also have $^w\chi_r = \chi_r$, and thus the right hand side is $\chi_r(t)$, and the claim is proved.  Thus it makes sense to {\em define}
$$
\chi_r(g^{-1}\delta \sigma(g)) = \chi_r(t)\ ,
$$
where $\widetilde{g}^{-1}\delta\sigma(\widetilde{g})\in I_r^+tw^{-1}I_r^+$. Now to compute $\chi_r(g^{-1}\delta \sigma(g))$ for $g$ as in (iii), we again fix any representative $\widetilde{g} \in G(L_r)$ for $g$ and write $\widetilde{g}^{-1} (\delta \sigma) \widetilde{g} \in I^+_r t w^{-1} I^+_r \sigma$.  The identity
\begin{equation*}  
t'^{-1} [\widetilde{g}^{-1} (\delta \sigma)^{-1} \widetilde{g}] t' \, e_{\bullet, r'} = e_{\bullet, r'}
\end{equation*}
implies that
\begin{equation} \label{twsigma_e=e}
[t'^{-1} (t w^{-1} \sigma) t'^{-1}]^{-1} \, e_{\bullet, r'} = e_{\bullet, r'}.
\end{equation}
(We require the $-1$ exponent on $[ \cdots ]$ here for the same reason it is required to make sense of $(\delta \sigma)^{-1}v_\bullet = v_\bullet$.)  ``Raising (\ref{twsigma_e=e}) to the $r$th power'' yields
\begin{equation} \label{N^w_r(t)e=e}
^{(w\sigma)^{r}}[t'^{-1} \, N^{w^{-1}}_r(t) \,\, ^{w^{-r}}\!\!\!\sigma^r(t')]^{-1} \, e_{\bullet, r'} = e_{\bullet, r'},
\end{equation}
where $N^{w^{-1}}_r(t) := t \cdot \, ^{w^{-1}\sigma}t \cdots \, ^{(w^{-1}\sigma)^{r-1}}t$.  Here we have used the fact that $w^{r}$, and thus also $(w\sigma)^{r}$, fixes $e_{\bullet, r'}$.  

We claim that in fact  $w \, e_{\bullet, r'} = e_{\bullet, r'}$.  First recall (Lemma \ref{general_form}, Proposition 
\ref{main_comb_prop}) that $w$ corresponds under Morita equivalence to a matrix of the form ${\rm diag}(1, \dots, p, \dots, 1) E$, where $E$ is an elementary matrix which fixes every $e_i$ with $i \notin S = S(w)$ and where $p$ appears in the $m$-th place for some $m \notin S$.  Then it is clear that for $i \notin S$, we have $w(p^{-1}e_i) = pe_m + p^{-1}e_i$, which is congruent modulo $\Lambda_{i-1}$ to $p^{-1}e_i$.  This proves our claim.

From (\ref{N^w_r(t)e=e}) and the fact that conjugation by $w\sigma$ preserves $T_S \subset T$ (see Lemma \ref{T_S_vs_wT})  we deduce
\begin{equation} \label{N^w_r(t)_in_T_S}
t'^{-1} \, N^{w^{-1}}_r(t) \,\, ^{w^{-r}}\!\!\!\sigma^r(t') \in T_S(k_{r'}).
\end{equation}

Now the fact that $\delta^1(w, \chi) =1$ means that $\chi \in T(\mathbb F_p)^\vee$ can be extended to a character $\widetilde{\chi}$ on $T(k_{r'}) \rtimes \langle w \rangle$ which is trivial on $T_S(k_{r'})$ and on $\langle w\rangle$ (Lemma \ref{chi_extn}).  Applying $\widetilde{\chi}$ to (\ref{N^w_r(t)_in_T_S}) gives
$$
\widetilde{\chi}(N^{w^{-1}}_r(t)) = \widetilde{\chi}(t'^{-1}\, ^{w^{-r}}\!\!\!\sigma^r(t'))^{-1}\ ,
$$
and hence the desired equality
$$
\chi_r(t) = \chi(t^{-1}_g).
$$
\end{proof}

\subsection{Oort-Tate generators in terms of Dieudonn\'e modules} \label{OT_Dieu_sec}

We pause to fill in an ingredient used in the proof of Lemma \ref{t_g_lemma}, (ii).

Let $\mathcal G$ denote an \'{e}tale finite group scheme over $k_r$ having order $p$.  For example, we could take  $\mathcal G = G_i := {\rm ker}(X_{i-1} \rightarrow X_{i})$ for a non-critical index $i$ in the chain of $p$-divisible groups (\ref{chainpdiv}).  Suppose $k_{r'} \supset k_r$ is an extension field such that $\mathcal G_{k_{r'}} \cong (\mathbb Z/p\mathbb Z)_{k_{r'}}$.  An Oort-Tate generator $\eta \in \mathcal G^\times(k_{r'})$ is then given by the choice of an isomorphism
$$
\eta: (\mathbb Z/p\mathbb Z)_{k_{r'}} ~ \widetilde{\rightarrow} ~ \mathcal G_{k_{r'}}.
$$
The group $\mu_{p-1}(k_{r'}) = \mathbb F_p^\times$ acts on $\mathcal G^\times(k_{r'})$: it acts on the set of $\eta$ via its natural action on $\mathbb Z/p\mathbb Z$.  Also, the group ${\rm Gal}({k_{r'}}/k_r)$ acts on $\mathcal G^\times(k_{r'})$. 

We want to describe these data in terms of Dieudonn\'{e} modules.   Let $M(\mathcal G)$ denote the covariant Dieudonn\'{e} module of $\mathcal G$; this is a $W(k_r)$-module equipped with Frobenius operator $F$ and Verschiebung operator $V$.  We have $M((\mathbb Z/p\mathbb Z)_{k_r}) = (k_r, F = 0, V = \sigma^{-1})$.  

Applying the functor $M(\cdot)$ to $\eta: (\mathbb Z/p\mathbb Z)_{k_{r'}} ~ \widetilde{\rightarrow} ~ \mathcal G_{k_{r'}}$ yields an isomorphism compatible with the Verschiebung operators on each side, 
$$M(\eta):  k_{r'}=M((\mathbb Z/p\mathbb Z)_{k_{r'}})\to M(\mathcal G_{k_{r'}}) .$$   
Since $\mathcal G$ is defined over $k_r$, $M(\mathcal G_{k_{r'}})$ carries an action of ${\rm Gal}(k_{r'}/k_r)$, in particular of $\sigma^r$.  Also, there is an obvious action of $\mathbb F_p^\times \subset k_{r'}^\times$ on $M(\mathcal G_{k_r'})$.

\begin{lemma}
Let $V$ denote the Verschiebung operator on $M(\mathcal G_{k_r})$.  The association $\eta \mapsto M(\eta)(1) = v$ sets up a bijection
$$
\left\{ \mbox{Oort-Tate generators $\eta \in \mathcal G^\times(k_{r'})$} \right\} \longleftrightarrow \left\{\mbox{non-zero vectors $v \in M(\mathcal G_{k_{r'}})$ with $Vv = v$} \right\}.
$$
Moreover, this bijection is equivariant for the actions of $\sigma^r$ and $\mu_{p-1}(k_{r'}) = \mathbb F_p^\times$ on both sides.\qed
\end{lemma}

\subsection{Contribution of ${\rm Isog}_{0}(A',\lambda',i')$ to the counting points formula}

Recall from \cite{K90}, $\S14$,  the map $(A',\lambda',i') \mapsto (\gamma_0;\gamma, \delta)$.  Above we recalled the construction of $\gamma$ and $\delta$.  Once $\gamma$ and $\delta$ are given, $\gamma_0$ is constructed from them as in loc.~cit., and we shall not review that here.  For the remainder of this section, we fix $(A',\lambda',i')$ and the associated triple $(\gamma_0;\gamma, \delta)$.

Temporarily fix $w \in {\rm Adm}^G(\mu)$, and consider points $x = (A_\bullet, \lambda, i, \eta) \in 
{\rm Isog}_{0,w}(A',\lambda',i')$.  For each $\chi \in T(\mathbb F_p)^\vee$ we must calculate
$$
{\rm Tr}^{ss}(\Phi^r_{\mathfrak p}, (R\Psi_{\chi})_x).
$$
By Theorem \ref{main_geom_thm}, this vanishes unless $\delta^1(w, \chi) = 1$, in which case it is
$$
{\rm Tr}^{ss}(\Phi^r_{\mathfrak p}, (\bar{\mathbb Q}_{\ell, \bar \chi} \otimes R\Psi_0)_x).
$$
The nearby cycles sheaf $R\Psi_0$ is constant along the KR-stratum $\mathcal A_{0,w}$, and if $x \in \mathcal A_{0,w}(k_r)$ maps to $g$ under (\ref{x_to_g}), then by Proposition \ref{A_0_tr_Fr} 
\begin{equation}
{\rm Tr}^{ss}(\Phi^r_{\mathfrak p}, (R\Psi_0)_x) = \phi_{r,0}(w^{-1}) = \phi_{r,0}(g^{-1}\delta \sigma(g)).
\end{equation}
By Lemma \ref{sommes_trig}, we have
$$
{\rm Tr}^{ss}(\Phi^r_{\mathfrak p}, (\mathbb Q_{\ell, \bar \chi})_x) = \chi(t_x) = \chi(t_g),
$$
using Lemma \ref{t_g_lemma} (ii) for the last equality.  By Definition \ref{test_fcn} and Lemma \ref{t_g_lemma} (iii) we have 
$$
\phi_{r,\chi}(g^{-1}\delta \sigma(g)) = \delta^1(w,\chi) \, \chi(t_g) \, \phi_{r,0}(w^{-1}).
$$
Putting these facts together with our parametrization (\ref{Isog_0w=}) of ${\rm Isog}_{0,w}(A',\lambda', i')$, we see that for each $w \in {\rm Adm}^G(\mu)$ and $\chi \in T(\mathbb F_p)^\vee$,
\begin{equation*}
\sum_{x \in {\rm Isog}_{0,w}(A',\lambda',i')} {\rm Tr}^{ss}(\Phi^r_{\mathfrak p}, (R\Psi_\chi)_x)
\end{equation*}
is given by
\begin{equation*}
{\rm vol}(I(\mathbb Q)\backslash I(\mathbb A_f)) ~ {\rm O}_{\gamma}(f^p) ~ {\rm TO}_{\delta \sigma}[\phi_{r,\chi} \cdot 1_{I_r w^{-1} I_r}].
\end{equation*}
Now summing over all $w \in {\rm Adm}^G(\mu)$, we obtain the following.

\begin{prop}\label{counting_Isog_0} For any $\chi \in T(\mathbb F_p)^\vee$, the contribution of $(A',\lambda',i')$ to the counting points formula is given by
\begin{equation}
\sum_{x \in {\rm Isog}_0(A',\lambda',i')} {\rm Tr}^{ss}(\Phi^r_{\mathfrak p}, (R\Psi_\chi)_x) = 
{\rm vol}(I(\mathbb Q)\backslash I(\mathbb A_f)) ~ {\rm O}_{\gamma}(f^p) ~ 
{\rm TO}_{\delta \sigma}(\phi_{r,\chi}),
\end{equation}
where $\phi_{r,\chi}$ is the function of {\em Definition \ref{test_fcn}}.
\end{prop}

\subsection{Conclusion of the proof of the counting points formula} 

By Kottwitz \cite{K92}, $(A',\lambda',i') \mapsto (\gamma_0; \gamma, \delta)$ maps to the set of triples for which the Kottwitz invariant $\alpha(\gamma_0; \gamma, \delta)$ is defined and trivial.  The image consists of triples for which $\gamma_0^* \gamma_0 = c$ for a positive rational number $c = p^r c_0$, for $c_0$ a $p$-adic unit, and for which there exists a lattice $\Lambda$ in $V_{L_r}$ with $(\delta \sigma)\Lambda \supset \Lambda$ (see \cite{K92}, Lemma 18.1).  In loc.~cit.,~p.~441,  Kottwitz determines the cardinality of the fiber above $(\gamma_0;\gamma, \delta)$ to be $|{\rm ker}^1(\mathbb Q,I_0)|$.  Further, he shows that in the counting points formula, we may drop 
the conditions in loc.~cit., Lemma 18.1, imposed on $(\gamma_0;\gamma,\delta)$ other than $\alpha(\gamma_0;\gamma,\delta) = 1$.
Recall that in our situation, where $|{\rm ker}^1(\mathbb Q,G)| = 1$, we have $c(\gamma_0;\gamma,\delta) = {\rm vol}(I(\mathbb Q)\backslash I(\mathbb A_f)) ~ |{\rm ker}^1(\mathbb Q,I_0)|$.  Therefore,  exactly as in loc.~cit., we derive Theorem \ref{counting_points} from Proposition \ref{counting_Isog_0}.

\section{Hecke algebra isomorphisms associated to depth-zero characters} \label{HA_isom_sec}

\subsection{Preliminaries}
In this section we temporarily change our notation.  We let $F$ denote an arbitrary $p$-adic field with ring of integers $\mathcal O$, and residue field $k_F$.  Let $q$ denote the cardinality of $k_F$.  Write $\varpi$ for a uniformizer (some constructions will depend upon this choice of uniformizer).  

Let $G$ denote a connected reductive group, defined and split over $\mathcal O$.  Fix an $F$-split maximal torus $T$ and a Borel subgroup $B$ containing $T$; assume $T$ and $B$ are defined over $\mathcal O$.   Let $^\circ T = T(\mathcal O)$ denote the maximal compact subgroup of $T(F)$.   Let $\Phi \subset X^*(T)$,  resp. $\Phi^\vee \subset X_*(T)$, denote the set of roots, resp. coroots, for $G,T$.  Let $U$, resp. $\overline{U}$,  denote the unipotent radical of $B$, resp. the Borel subgroup $\overline{B} \supset T$ opposite to $B$.

Let $N$ denote the normalizer of $T$ in $G$, let $W = N/T$ denote the Weyl group, and write $\widetilde{W} = N(F)/\, ^\circ T$ for the Iwahori-Weyl group, cf.~\cite{HR}.  There is a canonical isomorphism $X_*(T) = T(F)/\, ^\circ T, \spa \lambda \mapsto \varpi^\lambda := 
\lambda(\varpi)$ (independent of the choice of $\varpi$).  The canonical homomorphism $N(F)/\, ^\circ T = \widetilde{W} \rightarrow W = N(F)/T(F)$ has a  section, hence there is a (non-canonical) isomorphism with the extended affine Weyl group $\widetilde{W} \cong X_*(T) \rtimes W$.  

Next, write $\chi$ for a {\em depth-zero character} of $T$, that is, a homomorphism $^\circ T \rightarrow \bar{\QQ}_\ell^\times$ which is trivial on the kernel $T_1$ of the canonical homomorphism $^\circ T = T(\mathcal O) \rightarrow T(k_F)$.  Clearly $N(F)$, $\widetilde{W}$ and $W$ act on the set of depth-zero characters.  We define
$$
W_\chi = \{ w \in W ~ | ~ ^w\chi = \chi \}.
$$
Consider the natural surjective homomorphisms $N \rightarrow \widetilde{W} \rightarrow W$.  We denote by $N_\chi$ (resp. $\widetilde{W}_\chi$) the inverse image in $N$ (resp. $\widetilde{W}$) of $W_\chi$.  Then there are obvious surjective homomorphisms $N_\chi \rightarrow \widetilde{W}_\chi \rightarrow W_\chi$.  


Define $\Phi_\chi$ (resp. $\Phi^\vee_\chi$, resp. $\Phi_{\chi, \rm aff}$) to be the set of roots $\alpha \in \Phi$ (resp. coroots $\alpha^\vee \in \Phi^\vee$, resp. affine roots $a = \alpha + k$, where $\alpha \in \Phi$, $k \in \mathbb Z$) such that $\chi \circ \alpha^\vee|_{\mathcal O_F^\times} = 1$.  Note that $\widetilde{W}_\chi$ acts in an obvious way on $\Phi_{\chi, \rm aff}$.  Define the following subgroups of the group of affine-linear automorphisms of $V := X_*(T) \otimes \mathbb R$:
\begin{align*}
W^\circ_\chi &= \langle s_\alpha ~ | ~ \alpha \in \Phi_\chi \rangle \\
W_{\chi, \rm aff} &= \langle s_a ~ | ~ a \in \Phi_{\chi, \rm aff} \rangle.
\end{align*}
Here $s_a$ and $s_\alpha$ are the reflections on $V$ corresponding to $a$ and $\alpha$.  

It is clear that $\Phi_\chi$ is a root system with Weyl group $W^\circ_\chi$, and that $W^\circ_\chi \subseteq W_\chi$. Let $\Phi^+$ denote the $B$-positive roots in $\Phi$, and set $\Phi^+_\chi = \Phi_\chi \cap \Phi^+$.  Then $\Phi^+_\chi$ is a system of positive roots for $\Phi_\chi$, and we denote by  $\Pi_\chi$  the corresponding set of simple  roots. 

Let $\mathcal C_\chi$ resp. ${\bf a}_\chi$ denote the subsets in $V$ defined by 
\begin{align*}
\mathcal C_\chi &= \{ v \in V ~ | ~ 0 < \alpha(v), \,\, \forall \alpha \in \Phi^+_\chi \}, \spa \mbox{resp.} 
\\
{\bf a}_\chi &= \{ v \in V ~ | ~ 0 < \alpha(v) < 1, \,\, \forall \alpha \in \Phi^+_\chi \}.
\end{align*}

For $a \in \Phi_{\chi, \rm aff}$ we write $a > 0$ if $a(v) > 0$ for all $v \in {\bf a}_\chi$.    Let 
$$\Pi_{\chi, \rm aff} = \{ a \in \Phi_{\chi, \rm aff} ~ | ~ a \,\, \mbox{is a minimal positive element} \},
$$ for the obvious  ordering on the set $\Phi_{\chi, \rm aff}$.  Define
\begin{align*}
S_{\chi, \rm aff} &= \{ s_a ~ | ~ a \in \Pi_{\chi, \rm aff} \} \\  
\Omega_\chi &= \{ w \in \widetilde{W}_\chi ~ | ~ w{\bf a}_\chi = {\bf a}_\chi \}.
\end{align*}

  In general, $W_\chi$ can be larger that $W^\circ_\chi$ and is not even a Weyl group (see Example 8.3 in \cite{Ro} and Remark \ref{W_chi} below).  The next result is contained in \cite{Ro}.

\begin{lemma} [Roche]  \label{W_chi_aff_Cox}

\noindent {\rm (1)} The group $W_{\chi,\rm aff}$ is a Coxeter group with system of generators $S_{\chi, \rm aff}$;

\smallskip

\noindent {\rm (2)} there is a canonical decomposition $\widetilde{W}_\chi = W_{\chi, \rm aff} \rtimes \Omega_\chi$, and the Bruhat order $\leq_\chi$ and the length function $\ell_\chi$ on $W_{\chi, \rm aff}$ can be extended in an obvious way to $\widetilde{W}_\chi$ such that $\Omega_\chi$ consists of the length-zero elements;

\smallskip

\noindent {\rm (3)} if $W^\circ_\chi = W_\chi$, then ${W}_{\chi, \rm aff}$ (resp. $\widetilde{W}_\chi$) is the affine (resp. extended affine) Weyl group associated to the root system $\Phi_\chi $. Further, $\mathcal C_\chi$ is the dominant Weyl chamber corresponding to a set of simple positive  roots which can be identified with $\Pi_{\chi}$. Similarly,  ${\bf a}_\chi$ is the base alcove in $V$ corresponding to a set of simple positive affine roots which can be identified with $\Pi_{\chi, \rm aff}$.
\end{lemma}

\begin{Remark} \label{W_chi}
In \cite{Ro}, pp. 393-6, Roche proves that $W^\circ_\chi = W_\chi$ at least when $G$ has connected center and when $p$ is not a torsion prime for $\Phi^\vee$ (see loc.~cit.,  p. 396).  It is easy to see that $W^\circ_\chi = W_\chi$ always holds when $G = {\rm GL}_d$ (with no restrictions on $p$).  
\end{Remark}

\subsection{Bernstein components and Bushnell-Kutzko types associated to depth-zero characters}

Here we recall some facts about the Bernstein decomposition of the category $\mathcal R(G)$ of smooth representations of $G(F)$ and Bushnell-Kutzko types (cf. \cite{BK}), specifically in relation to principal series representations associated to depth-zero characters.

Choose any extension $\widetilde{\chi} : T(F) \rightarrow \bar{\QQ}_\ell^\times$ of the depth zero character $\chi : {^\circ T}\rightarrow  \bar{\QQ}_\ell^\times$.  This gives an inertial equivalence class
$$
[T, \widetilde{\chi}]_G = \mathfrak s_\chi = \mathfrak s,
$$
which depends only on the $W$-orbit of $\chi$.  It indexes a component $\mathcal R_\mathfrak s(G)$ of the Bernstein decomposition of $\mathcal R(G)$.  Recall that $\mathcal R_\mathfrak s(G)$ is the full subcategory of $\mathcal R(G)$ whose objects are the smooth representations of $G(F)$,  each of whose irreducible subquotients is a subquotient of some normalized principal series $i^G_B(\widetilde{\chi} \, \eta) := {\rm Ind}^{G(F)}_{B(F)}(\delta_B^{1/2} \, \widetilde{\chi} \, \eta)$, where $\eta$ denotes an unramified $\bar{\QQ}_\ell^\times$-valued character of $T(F)$. 

Let $I$ denote the Iwahori subgroup of $G(F)$ corresponding to the alcove ${\bf a}$, and let $I^+ \subset I$ be its pro-unipotent radical, so that $I/I^+ = T(k_F)$.  Let $\rho=\rho_\chi:I \rightarrow \bar{\QQ}_\ell^\times$ be the character of $I$, trivial on $I^+$, given by the character that $\chi$ induces on $T(k_F)$.  Set $I_U = I \cap U$ and $I_{\overline{U}} = I \cap \overline{U}$.  In terms of the Iwahori-decomposition
\begin{equation*}
I = I_{\overline{U}} \cdot \, ^\circ T \, \cdot I_U,
\end{equation*}
$\rho$ is given by 
\begin{equation} \label{rho_from_chi}
\rho(\overline{u} \cdot t \cdot u) = \chi(t).  
\end{equation}

Let $\mathcal R_\rho(G)$ denote the full subcategory of $\mathcal R(G)$ whose objects $(\pi, V)$ are generated, as $G$-modules, by their $\rho$-isotypical components $V^\rho$.  Let $\mathcal H(G)$, resp.~$\mathcal H(G,\rho)$, denote the Hecke algebra of compactly supported smooth functions $f : G(F) \rightarrow \bar{\QQ}_\ell$, resp. those with  $f(i_1 g i_2) = \rho^{-1}(i_1)  f(g) \rho^{-1}(i_2)$ for all $i_1, i_2 \in I$ and $g \in G(F)$. Note that $\mathcal H(G,\rho_\chi)$ can be identified with $\mathcal H(G,I, \chi)$ defined in section \ref{notation_sec}. 

\begin{theorem} [Roche, Bushnell-Kutzko]
The pair $(I, \rho)$ is an $\mathfrak s$-{\em type} in the sense of Bushnell-Kutzko \cite{BK},  i.e.,  an irreducible smooth representation $\pi$ of $G(F)$ contains $\rho$ if and only if $\pi \in \mathcal R_\mathfrak s(G)$.  Furthermore, $\mathcal R_\mathfrak s(G) = \mathcal R_\rho(G)$ as subcategories of $\mathcal R(G)$.  There is an equivalence of categories
\begin{align*}
\mathcal R_\rho(G) ~ &\widetilde{\rightarrow} ~ \mathcal H(G,\rho)\mbox{-\rm Mod} \\
(\pi,V) ~ &\mapsto ~ V^\rho.
\end{align*}
\end{theorem}

\begin{proof}
This is contained in Theorems 7.7 and 7.5 of \cite{Ro}.  The first part is proved there using Bushnell-Kutzko's notion of $G$-cover (cf. \cite{BK}).  
However, it can also be proved more directly (for depth-zero characters $\chi$) by imitating Casselman's proof \cite{Ca} of Borel's theorem that an irreducible smooth representation of $G(F)$ is a constituent of an unramified principal series if and only if it is generated by its Iwahori-fixed vectors -- i.e, essentially the case $\chi = {\rm triv}$ of the theorem.  One input in that approach is the Hecke-algebra isomorphism discussed in the following subsection.
\end{proof}

The above theorem allows us to describe $\mathcal Z(G, \rho)$, the center of $\mathcal H(G,\rho)$, as follows.  The general theory of the Bernstein center (cf. \cite{BD}) identifies the center of the category $\mathcal R_\mathfrak s(G)$ (i.e. the endomorphism ring  of its identity functor) with the ring of regular functions of an affine algebraic variety $\mathfrak X_\mathfrak s = \Spec \mathcal O_{\mathfrak X_\mathfrak s}$
(over $\bar{\QQ}_\ell$).  Let us make this concrete for the inertial class $\mathfrak s = \mathfrak s_\chi$.  We can identify
\begin{equation} \label{X_mathfrak_s}
\mathfrak X_\mathfrak s = \{ (T, \xi)_G \},
\end{equation}
the set of $G$-equivalence classes $(T,\xi)_G$ of pairs $(T,\xi)$ where $\xi : T(F) \rightarrow \bar{\QQ}_\ell^\times$ is a character extending $\chi$.  Two such pairs are $G$-equivalent if they are $G(F)$-conjugate, in the obvious sense.  

In the following we identify the dual torus $\widehat{T} = \widehat{T}(\bar{\QQ}_\ell)$ with the group of unramified characters $\eta$ on $T(F)$.

\begin{lemma} \label{X_s_unifor}
Let $\widetilde{\chi}$ be a $W_\chi$-invariant extension of $\chi$ to a character of $T(F)$.  Then the map $\eta \mapsto (T, \eta \, \widetilde{\chi})_G$ determines an isomorphism of algebraic varieties
$$
\widehat{T}/W_\chi ~ \widetilde{\rightarrow} ~ \mathfrak X_\mathfrak s.
$$
Thus, there is an isomorphism of algebras
$$
\beta_{\widetilde{\chi}}: \bar{\QQ}_\ell[X_*(T)]^{W_\chi} = \mathcal O_{\mathfrak X_\mathfrak s} ~\widetilde{\rightarrow} ~  \mathcal Z(G, \rho)
$$
having the property that $z \in \mathcal Z(G, \rho)$ acts on $i^G_B(\widetilde{\chi}\eta)^\rho$ by the scalar 
$\beta^{-1}_{\widetilde{\chi}}(z)(\eta)$.  

{\em Here $\beta^{-1}_{\widetilde \chi}(z)$ is viewed as a regular function on the variety $\widehat{T}/W_\chi$, and $\eta$ is a point in $\widehat{T}/W_\chi$.}\qed
\end{lemma}

We note that $W_\chi$-invariant extensions of $\chi$ exist by Remark \ref{varpi_can_rem} below.  The next lemma will be useful in the following subsection.

\begin{lemma} \label{extn_lemma}
Let $\widetilde \chi$ denote a character on $T(F)$ which extends $\chi$.  Then $\widetilde \chi$ 
extends to a character $\breve \chi$ on $N_\chi$ if and only if $\widetilde \chi$ is $W_\chi$-invariant.
\end{lemma}

\begin{proof} 
First, $\chi$ can always be extended to a character $\overline{\chi}$ on $N_\chi(\mathcal O)=N_\chi\cap N(\mathcal O)$, cf.\ \cite{HL}, 6.11.  Choose any uniformizer $\varpi$ and note that $X_*(T) \rtimes N_\chi(\mathcal O) \cong N_\chi$ via $\spa (\nu,n_0) \mapsto \varpi^\nu n_0$.  Define a function $\breve \chi$ on $N_\chi$ by setting 
$$\breve \chi(\varpi^\nu n_0) = \widetilde \chi(\varpi^{\nu})\overline{\chi}(n_0).$$
It is then easy to check that $\breve \chi$ is a character if and only if $\widetilde \chi$ is $W_\chi$-invariant.
\end{proof}

\begin{Remark} \label{varpi_can_rem}
We single out a $W_\chi$-invariant extension $\widetilde{\chi}^\varpi$, which depends on our choice of uniformizer $\varpi$ and which we call the $\varpi$-{\em canonical} extension.  It is defined by  
\begin{equation} \label{varpi_canon}
\widetilde{\chi}^\varpi(\varpi^\nu \, t_0) = \chi(t_0), \spa \forall \nu \in X_*(T), \spa t_0 \in \, ^\circ T.
\end{equation}
In the notation of Lemma \ref{extn_lemma},  we can define an extension $\breve \chi$ of $\widetilde \chi$ to $N_\chi$ by 
$\breve \chi(\varpi^\nu n_0) = \overline{\chi}(n_0)$.
\end{Remark}

\subsection{Hecke algebra isomorphisms}

Fix $\chi$ and $\rho=\rho_\chi$ as above.  Fix a $W_\chi$-invariant character $\widetilde \chi$ on $T(F)$ which extends $\chi$.  Let $\breve \chi$ denote a character on $N_\chi$ extending $\widetilde \chi$.

Assume that $W^\circ_\chi = W_\chi$.  Let $H = H_\chi$ be an $F$-split connected reductive group containing $T$ as a maximal torus and having $(X^*(T),\Phi_\chi, X_*(T), \Phi^\vee_\chi, \Pi_\chi)$ as based root system.  
Recall (Lemma \ref{W_chi_aff_Cox} (3)) that $\widetilde{W}_\chi$ is the extended affine Weyl 
group attached to $H$.

 Let $B_H \supset T$ denote the Borel subgroup which corresponds to the set of simple positive roots $\Pi_\chi$.  Let $I_H$ denote the Iwahori subgroup of $H(F)$ corresponding to the base alcove ${\bf a}_\chi$ and $\mathcal H(H,I_H)$  the Iwahori-Hecke algebra of $H(F)$ relative to $I_H$.  

For $w \in \widetilde{W}_\chi$, choose any $n = n_w \in N_\chi$ mapping to $w$ under the projection $N_\chi \rightarrow \widetilde{W}_\chi$.  Define
\begin{equation} \label{InI_chi}
[InI]_{\breve{\chi}} \in \mathcal H(G,\rho)
\end{equation}
to be the unique element in $\mathcal H(G, \rho)$ which is supported on $InI$ and whose value at $n$ is $\breve{\chi}^{-1}(n)$.  Note that $[InI]_{\breve{\chi}}$ depends only on $w$, not on the choice of $n \in N_\chi$ mapping to  $w \in \widetilde{W}_\chi$.

The following theorem is essentially due to D.~Goldstein \cite{Go}.  Our precise statement is extracted from Morris \cite{Mor} and Roche \cite{Ro}.  The hypothesis $W^\circ_\chi = W_\chi$ is not necessary, but it is sufficient for our purposes and  simplifies the statement.  
Recall that $\ell $ resp. $\ell_\chi$ denotes the length function on $\widetilde{W}$ resp. $\widetilde{W}_\chi = \widetilde{W}(H)$,  defined using the alcove ${\bf a}$ resp. ${\bf a}_\chi$.

\begin{theorem} [Goldstein, Morris, Roche] \label{GMR} Assume $W^\circ_\chi = W_\chi$, and define $H = H_\chi$ as above.  Fix the extensions $\widetilde{\chi}$ and $\breve \chi$ as above.

\smallskip

\noindent {\rm (1)}
The functions $[In_wI]_{\breve{\chi}}$ for $w \in \widetilde{W}_\chi$  form a $\bar{\QQ}_\ell$-basis for the algebra $\mathcal H(G,\rho)$. 

\smallskip

\noindent {\rm (2)} There is an isomorphism of algebras  (depending on the choice of $\breve \chi$)
$$
\xymatrix{
\Psi_{\breve \chi}:\mathcal H(G,\rho) \ar[r]^{\sim} & \mathcal H(H,{I_H})}
$$
which sends $q^{-\ell(w)/2}[InI]_{\breve{\chi}}$ to $q^{-\ell_\chi(w)/2} [I_H n I_H]$ for any $n \mapsto w \in \widetilde{W}_\chi$.   In particular, $[In_wI]_{\breve \chi}$ is invertible in $\mathcal H(G, \rho)$ for every $w \in \widetilde{W}_\chi$.
\end{theorem}

\begin{proof}
The key step for part (1) is proved in \cite{Mor}, Lemma 5.5.  Part (2) is a special case of \cite{Ro}, Lemma 9.3.  Note, however, that we describe $\Psi$ as depending only on a choice of extension $\breve \chi$, instead of on the choice of the uniformizer $\varpi$ and other data as in \cite{Ro}, since this seems more natural to us.  
\end{proof}

A very special case of the theorem (which is easy to check directly) is that there is an algebra 
isomorphism
\begin{equation}
\xymatrix{
\psi_{\widetilde{\chi}}:\mathcal H(T,\chi) \ar[r]^{\sim} & \mathcal H(T, {^\circ T}) }
\end{equation}
 which sends $[t \, ^\circ T]_{\widetilde{\chi}}$ to $[t \, ^\circ T]$.

Next, we construct an algebra homomorphism
\begin{equation*}
\theta^{\chi}_B: \mathcal H(T,\chi) \rightarrow \mathcal H(G, \rho)
\end{equation*}
as follows.  Given $t \in T(F)$, we write $t \in T^{+}$ provided that $t = \varpi^\nu t_0$ where $t_0 \in \, ^\circ T$ and $\nu \in X_*(T)$ is $B$-dominant.  This notion is independent of the choice of uniformizer $\varpi$ we use.  Any $t \in T(F)$ may be written as $t = t_1 t_2^{-1}$, where $t_1, t_2 \in T^+$.  

We then define $\theta^{\chi}_B$ by putting
\begin{equation}
\theta^{\chi}_B([t \, ^\circ T]_{\widetilde{\chi}}) = \delta_B^{1/2}(t) \, [It_1 I]_{\breve \chi} [It_2I]_{\breve \chi}^{-1}.
\end{equation}

\begin{lemma} The homomorphism is well-defined, i.e.,  $\theta^{\chi}_B([t\, ^\circ T]_{\widetilde \chi})$ is independent of the representative $t$ and of the choice of $t_1,t_2 \in T^+$ such that $t = t_1 t_2^{-1}$.  As the notation suggests, $\theta^{\chi}_B$ depends on $\chi$ but not on the extension $\breve \chi$ of $\chi$.
\end{lemma}

\begin{proof}
The independence of the representative $t$ is clear.  Using the isomorphism $\Psi_{\breve \chi}$, one can translate the assertion concerning $t_1,t_2$ (associated to a fixed choice for $t$) into an analogous independence statement for the Iwahori-Hecke algebra $\mathcal H(H, {I_H})$, which is well-known (cf. \cite{Lus}).  The final assertion of the lemma is left to the reader.  
\end{proof}

Note that $\theta^1_{B_H}$ is the usual embedding of the commutative subalgebra $\bar{\QQ}_\ell[X_*(T)]$ into the Iwahori-Hecke algebra $\mathcal H(H, {I_H})$ (cf. \cite{Lus}, \cite{HKP}).

\begin{prop} \label{comm_diag_prop} With the same hypotheses as Theorem \ref{GMR}, 

\smallskip

\noindent {\rm (1)} the following diagram commutes:
\begin{equation} \label{comm_diagram}
\begin{aligned}
\xymatrix{ \mathcal H(G,\rho) \ar[r]^{\Psi_{\breve \chi}}_{\sim} & \mathcal H(H, {I_H})  \\
\mathcal H(T, \chi) \ar[u]^{\theta^{\chi}_B} \ar[r]^{\psi_{\widetilde \chi}}_{\sim} & \mathcal H(T, {^\circ T}) \ar[u]_{\theta^{1}_{B_H}};
}
\end{aligned}
\end{equation}

\smallskip

\noindent {\rm (2)}   the isomorphism $\Psi_{\breve \chi}$ sets up an equivalence of categories $\mathfrak R_{\mathfrak s_\chi}(G) \cong \mathfrak R_{\mathfrak s_1}(H)$ under which $i^G_B(\widetilde{\chi} \, \eta)$ corresponds to $i^H_{B_H}(\eta)$. 
\end{prop}

\begin{proof}
For (1), suppose $t \in T^+$ maps to $t_\nu \in \widetilde{W}_\chi$.  Then $\delta_B(t) = q^{-\ell(t_\nu)}$ and $\delta_{B_H}(t) = q^{-\ell_\chi(t_\nu)}$.  It is immediate from this and Theorem \ref{GMR} (2) that $\Psi_{\breve \chi} \circ \theta^{\chi}_B$ and $\theta^1_{B_H} \circ \psi_{\widetilde \chi}$ agree on $[t \, ^\circ T]_{\widetilde \chi}$.  The commutativity in (1) follows,  since such elements and their inverses generate the algebra $\mathcal H(T, \chi)$.

Part (2) is contained in \cite{Ro}, Theorem 9.4.  Indeed, it is checked there that (\ref{comm_diagram}) induces a diagram of functors (whose horizontal arrows are equivalences of categories and which commutes up to a natural isomorphism of functors)
\begin{equation}
\begin{aligned}
\xymatrix{ 
\mathcal H(G,\rho)\mbox{-\rm Mod} \ar[r]^{\Psi^\ast_{\breve \chi}}_{\sim} & \mathcal H(H, {I_H})\mbox{-\rm Mod}  \\
\mathcal H(T, \chi)\mbox{-\rm Mod} \ar[u]^{i^G_B} \ar[r]^{\psi^\ast_{\widetilde \chi}}_{\sim} & \mathcal H(T, {^\circ T})\mbox{-\rm Mod} \ar[u]_{i^H_{B_H}}\ .
}
\end{aligned}
\end{equation}\end{proof}

In the next proposition we again retain the hypotheses of Theorem \ref{GMR}.

\begin{prop} \label{Z_comm_diag_prop} 

\noindent {\rm (1)} Diagram (\ref{comm_diagram}) induces a commutative diagram
\begin{equation} \label{Z_comm_diagram}
\begin{aligned}
\xymatrix{ \mathcal Z(G,\rho) \ar[r]^{\Psi_{\breve \chi}}_{\sim} & \mathcal Z(H, {I_H})  \\
\mathcal H(T, \chi)^{W_\chi} \ar[u]^{\theta^{\chi}_B}_{\wr} \ar[r]^{\psi_{\widetilde \chi}}_{\sim} & \mathcal H(T, {^\circ T})^{W_\chi} \ar[u]_{\theta^{1}_{B_H}}^{\wr}.
}
\end{aligned}
\end{equation}
In particular, $\Psi_{\breve \chi}$ restricted to the center depends only on $\widetilde \chi$, not on $\breve\chi$.  

\smallskip

\noindent {\rm (2)} The following (equivalent) statements hold: 
\begin{itemize}
\item[{\rm (i)}]Let $z \in \mathcal Z(G,\rho)$.  Then $z$ acts on $i^G_B(\widetilde{\chi} \, \eta)^\rho$ by the scalar by which $\Psi_{\breve \chi}(z)$ acts on $i^H_{B_H}(\eta)^{I_H}$.
\item[{\rm (ii)}]  The isomorphism $\theta^{\chi}_B \circ \psi^{-1}_{\widetilde \chi}$
$$\bar{\QQ}_\ell[X_*(T)]^{W_\chi} = \mathcal H(T,  {^\circ T})^{W_\chi} ~ \widetilde{\rightarrow} ~ \mathcal Z(G,\rho)
$$
coincides with the isomorphism $\beta_{\widetilde \chi}$ of Lemma \ref{X_s_unifor} (hence in particular is independent of the choice of Borel subgroup $B \supset T$).
\end{itemize}  
\end{prop}

\begin{proof}
Concerning (\ref{Z_comm_diagram}), it is enough to recall that for Iwahori-Hecke algebras, $\theta^1_{B_H}$ induces an isomorphism $\bar{\QQ}_\ell[X_*(T)]^{W_\chi} ~ \widetilde{\rightarrow} ~ \mathcal Z(H,  {I_H})$ (\cite{Lus}, \cite{HKP}).  This proves (1).  We leave (2) as an exercise for the reader (use Proposition \ref{comm_diag_prop} (2)).
\end{proof}

In light of (1), when $\Psi_{\breve \chi}$ is restricted to the center we will also denote it by the symbol $\Psi_{\widetilde \chi}$.

\section{The base change homomorphism for depth-zero principal series} \label{BC_sec}

\subsection{Properties of the base change homomorphism}

We retain the notation and hypotheses of the previous section, except the hypothesis $W^\circ_\chi = W_\chi$, and we add some new notation.  For $r \geq 1$, let $F_r/F$ denote the unramified extension of $F$ in $\bar{F}$ having degree $r$.  Let $\mathcal O_r \subset F_r$ be its ring of integers, with residue field $k_r$.  Set $G_r := G(F_r)$, $B_r := B(F_r)$, and $T_r := T(F_r)$.  Write $I_r$ for the Iwahori subgroup of $G_r$ corresponding to $I$.  Use the symbol $N_r$ to denote the norm homomorphisms $T(F_r) \rightarrow T(F)$ and $T(\mathcal O_r) \twoheadrightarrow T(\mathcal O)$,  as well as the norm homomorphism $T(k_r) \twoheadrightarrow T(k)$.  

Let $\chi_r := \chi \circ N_r$, which we regard as a depth-zero $\bar{\QQ}_\ell^\times$-valued character on $^\circ T_r := T(\mathcal O_r)$,  or as the induced character on $T(k_r)$.  Note that we can identify $W_{\chi} = W_{\chi_r}$ (as subgroups of the absolute Weyl group $W$), and $\widetilde{W}_{\chi} = \widetilde{W}_{\chi_r}$ (as subgroups of $\widetilde{W}$ once it is identified with $X_*(T) \rtimes W$).  

Let $\widetilde{\chi}_r$ denote a character on $T_r$ extending $\chi_r$.  
Write $\mathfrak s_r := \mathfrak s_{\chi_r} = [T_r, \widetilde{\chi}_r]_{G_r}$, an inertial equivalence class for the category $\mathfrak R(G_r)$, which depends only on $\chi_r$.  

Write $\mathcal H(G_r, \rho_r)$ (resp. $\mathcal Z(G_r, \rho_r)$) for the analogue of $\mathcal H(G, \rho)$ (resp. $\mathcal Z(G, \rho)$).  

\smallskip

There is a canonical morphism of algebraic varieties
\begin{equation}
\begin{aligned}
b^\ast_r ~ : ~\mathfrak X_{\mathfrak s} & \rightarrow \mathfrak X_{\mathfrak s_r} \\
(T, \xi)_G & \mapsto (T_r, \xi \!\circ\! N_r)_{G_r}.
\end{aligned}
\end{equation}
Here $\xi$ ranges over characters on $T(F)$ extending $\chi$.  Let us describe $b^\ast_r$ explicitly.  Suppose $\widetilde{\chi}$ is a $W_\chi$-invariant character on $T(F)$ extending $\chi$.  Then $\widetilde{\chi}_r := \widetilde{\chi} \circ N_r$ is $W_\chi$-invariant and extends $\chi_r$.  Use $\widetilde{\chi}_r$, resp. $\widetilde \chi$, to identify $\widehat{T}/W_\chi$ with $\mathfrak X_{\mathfrak s_r}$, resp. $\mathfrak X_{\mathfrak s}$, as in Lemma \ref{X_s_unifor}.   Having chosen $\widetilde{\chi}$ and $\widetilde{\chi}_r$, we can identify $b^\ast_r$ with the morphism
\begin{equation}
\begin{aligned}
b^\ast_r ~ : ~ \widehat{T}/W_\chi & \rightarrow \widehat{T}/W_\chi \\
t & \mapsto N_r(t) = t^r.
\end{aligned}
\end{equation}

 We define the  {\em base change homomorphism} 
\begin{equation}\label{norm_map}
b_r : \mathcal Z(G_r, \rho_r) \rightarrow \mathcal Z(G,\rho)
\end{equation}
to be the canonical algebra homomorphism of rings of regular functions $\mathcal O_{\mathfrak X_{\mathfrak s_r}} \rightarrow \mathcal O_{\mathfrak X_{\mathfrak s}}$ which corresponds to $b^\ast_r : \mathfrak X_{\mathfrak s}  \rightarrow \mathfrak X_{\mathfrak s_r} $.  After choosing the extensions $\widetilde \chi$ and $\widetilde{\chi}_r$ as above, we can identify $b_r$ with the homomorphism
\begin{equation}
\begin{aligned}
  b_r : \bar{\QQ}_\ell[X_*(T)]^{W_\chi} & \rightarrow \bar{\QQ}_\ell[X_*(T)]^{W_\chi} \\
\sum_{\mu \in X_+} a_\mu m_\mu &\mapsto \sum_{\mu \in X_+} a_\mu m_{r\mu}.
\end{aligned}
\end{equation}
Here $X_+$ denotes a set of representatives for the $W_\chi$-orbits in $X_*(T)$.  Further, $a_\mu \in \bar{\QQ}_\ell$, and $m_\mu$ denotes the monomial symmetric element $m_\mu := \sum_{\lambda \in W_\chi \mu} t_\lambda$.

The homomorphism $b_r$ is analogous to the base-change homomorphism for centers of parahoric Hecke algebras, defined in \cite{H09}.  Furthermore, these base-change homomorphisms are compatible with the Hecke algebra isomorphisms of section \ref{HA_isom_sec}, as follows.   Choose a $W_\chi$-invariant character $\widetilde \chi$ on $T(F)$ extending $\chi$, and set $\widetilde{\chi}_r := \widetilde{\chi} \circ N_r$.

\begin{lemma} \label{HA_comp_lemma}  Assume that $W_\chi = W^\circ_\chi$.  Consider the group $H = H_\chi$ from section \ref{HA_isom_sec}, and set $H_r := H(F_r)$, etc.  

\smallskip

\noindent {\rm (1)} The following diagram commutes
$$
\xymatrix{
\mathcal Z(G_r, \rho_r) \ar[r]^{\Psi_{\widetilde{\chi}_r}} \ar[d]_{b_r} & \mathcal Z(H_r, {I_{H_r}})\ar[d]^{b^H_r} \\
\mathcal Z(G, \rho) \ar[r]^{\Psi_{\widetilde \chi}} & \mathcal Z(H, {I_H}),}$$
where $b^H_r$ is the base-change homomorphism defined for the groups $H_r,H$ and the trivial character $\chi = {\rm triv}$.

\smallskip

\noindent {\rm (2)} Let $\phi_r \in \mathcal Z(G_r, \rho_r)$.  For any unramified character $\eta$ on $T(F)$, set $\eta_r := \eta \circ N_r$.  Then the following five scalars (in $\bar{\QQ}_\ell$) coincide for any $\eta \in \widehat{T}$:
\begin{itemize}
\item[{\rm (i)}] the scalar by which $\phi_r$ acts on $i^{G_r}_{B_r}(\widetilde{\chi}_r \, \eta_r)^{\rho_r}$;
\item[{\rm (ii)}] the scalar by which $\Psi_{\widetilde{\chi}_r}(\phi_r)$ acts on $i^{H_r}_{B_{H_r}}(\eta_r)^{I_{H_r}}$;
\item[{\rm (iii)}] the scalar by which $b_r(\phi_r)$ acts on $i^{G}_B(\widetilde{\chi} \, \eta)^\rho$;
\item[{\rm (iv)}] the scalar by which $b^H_r(\Psi_{\widetilde{\chi}_r}(\phi_r))$ acts on $i^H_{B_H}(\eta)^{I_H}$;
\item[{\rm (v)}] the scalar by which $\Psi_{\widetilde \chi}(b_r(\phi_r))$ acts on $i^H_{B_H}(\eta)^{I_H}$.
\end{itemize}
\end{lemma}

\begin{proof}
Part (2) follows from the definition of $b_r$ and Proposition \ref{Z_comm_diag_prop} (2).  Part (1) follows from part (2).
\end{proof}

\subsection{The base change fundamental lemma for depth-zero principal series} 

Retain the notation and assumptions of the previous subsection.  Further, let $\sigma \in {\rm Gal}(F_r/F)$ denote any generator of that cyclic group.

We will apply the following stable base change fundamental lemma in the present context.  To state it, we will use the standard notions of stable (twisted) orbital integrals ${\rm SO}_{\delta \sigma}(\phi)$, associated pairs of functions $(\phi,f)$,  and the norm map $\mathcal N$.  For further information, we refer the reader to \cite{K82}, \cite{K86}, and \cite{H09}. The proof of the next theorem will appear in \cite{H10}, where it is proved for arbitrary unramified groups.

\begin{theorem} \label{fl}
For every $\phi_r \in \mathcal Z(G_r, \rho_r)$, the functions $\phi_r$ and $b_r(\phi_r)$ are associated.  That is, for every semi-simple $\gamma \in G(F)$, we have
\begin{equation*}
{\rm SO}_\gamma(b_r (\phi_r)) = \begin{cases} {\rm SO}_{\delta \sigma}(\phi_r) , \spa \mbox{if $\gamma = \mathcal N \delta$ for $\delta \in G(F_r)$}, \\
0 , \spa \mbox{if $\gamma$ is not a norm from $G(F_r)$.}
\end{cases}
\end{equation*}
\end{theorem}

\begin{Remark}
We can view $\chi' \in T(k_r)^\vee$ as a depth-zero character on $T(\mathcal O_r)$ and thus define the corresponding algebra $\mathcal Z(G_r, I_r, \chi')$.   The complete set of idempotents $e_{\rho_{\chi'}}$, $\chi' \in T(k_r)^\vee$, gives rise to an algebra monomorphism 
$$
\mathcal Z(G_r, I^+_r) \hookrightarrow \prod_{\chi' \in T(k_r)^\vee} \mathcal Z(G_r, I_r, \chi').
$$
In \cite{H10}, $\S10$, the base change homomorphisms $b_r: \mathcal Z(G_r, I_r, \chi_r) \rightarrow \mathcal Z(G,I,\chi)$ and the above remark are used to define the base change homomorphism
\begin{equation} \label{full_bc}
b_r: \mathcal Z(G_r, I^+_r) \rightarrow \mathcal Z(G,I^+).
\end{equation}
We refer to loc.~cit.~for the construction and for the proof that the analogue of Theorem \ref{fl} holds for (\ref{full_bc}). 
\end{Remark}

\section{Image of the test function under Hecke algebra isomorphism} \label{image_sec}

\subsection{Notation} Now we specialize to the case of $F = \QQ_p$ and, as in section \ref{tracefrob_sec}, write $G = {\rm GL}_d \times \mathbb G_m$ for the localization at $p$ of our usual global group.  Let $T$, resp. $B$, denote the maximal torus, resp. Borel subgroup, in $G$ whose first factor is the usual diagonal torus, resp. upper triangular Borel subgroup, in ${\rm GL}_d$.  

We only need to consider characters $\chi$ on $T$ which are trivial on the $\mathbb G_m$-factor.  For this reason, in what follows we may as well assume $G = {\rm GL}_d$, $T = \mathbb G_m^d$, and fix $\chi = (\chi_1, \dots, \chi_d)$, a  character of $T(\ZZ_p)$ (or $T(\FF_p)$).  Note that we may identify $H=H_\chi$ with the semi-standard Levi subgroup $M = M_\chi$ associated to the decomposition (\ref{decchi}) of $d$ by the ``level sets" of $\chi$.  Recall that the Iwahori subgroup $I \subset G(\mathbb Q_p)$ fixes the base alcove ${\bf a}$ which we identify with the lattice chain $\Lambda_\bullet$ defined by $\Lambda_i = (p^{-1}\mathbb Z_p)^i \oplus \mathbb Z^{d-i}_p$.  Throughout this section, let $\varpi = p$, viewed as a uniformizer in $\QQ_p$ or $L_r=\mathbb Q_{p^r}$.

Definition \ref{test_fcn} describes  our test function $\phi_{r,\chi}$ as an element in the Hecke algebra $\mathcal H(G_r, I_r, \chi_r) = \mathcal H(G_r, \rho_r)$.  Implicit in this definition is the choice of uniformizer $\varpi=p$ used to define the embedding $i_{\varpi,K}: \widetilde{W} \hookrightarrow G(F)$ via $(\nu, \bar w) \mapsto \varpi^\nu n_{\bar w}$, cf. section \ref{notation_sec}.  In the present case we take $K = G(\mathcal O)$, and take $n_{\bar w} \in N(\mathcal O)$ to be the lift of $\bar w \in W =S_d$ given by identifying it with a $d \times d$ permutation matrix.  For $w \in \widetilde{W}$, denote its image under this embedding by $n_w$. 

For the choice of uniformizer $\varpi = p$, we fix the $\varpi$-canonical extension $\widetilde\chi^\varpi$ of $\chi$ (sometimes denoted below by $\widetilde \chi$).  We will choose an extension $\breve \chi$ of $\widetilde\chi^\varpi$ as in Remark \ref{varpi_can_rem}, but in a particular way.  Note that the isomorphism $N(\mathcal O) \cong T(\mathcal O) \rtimes S_d$ (given as above using permutation matrices) induces an isomorphism $N_\chi(\mathcal O) \cong T(\mathcal O) \rtimes W_\chi$.  Hence we may define $\overline{\chi}$ to be trivial on the elements in $W_\chi \hookrightarrow N_\chi(\mathcal O)$, and use this $\overline{\chi}$ and $\widetilde{\chi}^\varpi$ to define $\breve \chi$ as in Remark \ref{varpi_can_rem}.   Set $\chi_r := \chi \circ N_r$, and $\widetilde \chi_r := \widetilde{\chi} \circ N_r$.  Using $N_\chi(\mathcal O_r) \cong T(\mathcal O_r) \rtimes W_\chi$, define $\breve{\chi}_r$ analogously to $\breve \chi$; it is a character on $N_{\chi_r}(F_r) = N_\chi(F_r)$ extending $\widetilde{\chi}_r$.  We use ${\breve \chi}_r$ in the  construction of  the Hecke algebra isomorphism $\Psi_{{\breve \chi}_r}$ as in Theorem \ref{GMR}.

\subsection{Centrality of $\phi_{r,\chi}$ and its image in the Iwahori Hecke algebra of $M_\chi$} Let $\mathcal O^{\rm triv}_\chi$ denote the set of indices $j \in \{1,2, \dots, d \}$ such that $\chi_j = {\rm triv}$.  If $\mathcal O^{\rm triv}_\chi \neq \emptyset$, let $\mu^1_\chi$ denote the $B_M$-dominant element in the set of coweights ${e}_j \spa (j \in \mathcal O^{\rm triv}_\chi)$.  Recall the set ${\rm Adm}^G(\mathcal O^{\rm triv}_\chi)$ defined in section \ref{main_comb_subsec}.  By Proposition \ref{main_comb_prop} (b) we have
\begin{equation} \label{AdmG=AdmM}
{\rm Adm}^G(\mathcal O^{\rm triv}_\chi) = {\rm Adm}^M(\mu^1_\chi),
\end{equation}
where by convention both sides are empty if $\mathcal O^{\rm triv}_\chi = \emptyset$.  Also, $\delta^1(w, \chi) = 1$ iff $w \in {\rm Adm}^G(\mathcal O^{\rm triv}_\chi)$. 

Thus, we may write Definition \ref{test_fcn} as
\begin{equation}
\phi_{r,\chi} = \sum_{w \in {\rm Adm}^G(\mathcal O^{\rm triv}_\chi)} \phi_{r,0}(w^{-1}) \spa 
[I_r n_{w^{-1}} I_r]_{\breve{\chi}_r}.
\end{equation}

Recall that $\phi_{r,0} = k^{G}_{\mu^*_0,r}$ and $k^G_{\mu^*_0}(w^{-1}) = k^G_{\mu_0}(w)$.  Together with (\ref{AdmG=AdmM}) this yields
\begin{equation}
\phi_{r,\chi} = \sum_{w \in {\rm Adm}^{M}(\mu^1_\chi)} k^G_{\mu_0, r}(w) \spa 
[I_r n_{w^{-1}} I_r]_{\breve{\chi}_r}.
\end{equation}

Writing $\ell^M(\cdot)$ in place of $\ell_\chi(\cdot)$,  we see,  using Theorem \ref{GMR} (2),  that 

\begin{equation}
\Psi_{{\breve \chi}_r}(\phi_{r,\chi}) = \sum_{w \in {\rm Adm}^M(\mu^1_\chi)} k^{G}_{\mu_0, r} (w) \spa q^{\ell(w)/2 - \ell^M(w)/2} \spa [I_{M_r} n_{w^{-1}} I_{M_r}] .
\end{equation}

Here $q = p^r$.  Using (\ref{ell_diff}) and Proposition \ref{main_comb_prop} (e), we may rewrite the right hand side as
\begin{equation}
q^{\ell(t_{\mu_0})/2 - \ell^M(t_{\mu^1_\chi})/2} \sum_{w \in {\rm Adm}^M(\mu^1_\chi)} \spa 
k^M_{\mu^1_\chi, r}(w) \spa [I_{M_r} n_{w^{-1}} I_{M_r}] .
\end{equation}

Again, if $\mu^{1*}_\chi$ denotes the ``dual'' of $\mu^1_\chi$ (the analogue of $\mu_0^*$), we have $k^M_{\mu^{1*}_\chi,r}(w^{-1}) = k^M_{\mu^1_\chi, r}(w)$.  We have proved the following result.

\begin{prop} \label{test_fcn_image} Let $\varpi = p$.  If $\mathcal O^{\rm triv}_\chi = \emptyset$, then $\phi_{r, \chi} = 0$.  If $\mathcal O^{\rm triv}_\chi \neq \emptyset$, then $\phi_{r, \chi} \in \mathcal Z(G_r, I_r, \chi_r)$, and the isomorphism $\Psi_{\widetilde{\chi}}$ sends $\phi_{r, \chi}$ to the element 
$$
q^{\langle \rho_B, \mu^* \rangle - \langle \rho_{B_M}, \mu^{1*}_\chi \rangle} \spa k^M_{\mu^{1*}_\chi, r}
$$
in the center $\mathcal Z(M_r, I_{M_r})$ of the Iwahori-Hecke algebra for the Levi subgroup $M$ associated to $\chi$.
\end{prop}

\section{Summing up the test functions $\phi_r(R\Psi_\chi)$} \label{sum_up_sec}

\subsection{Proof that the test function $\phi_{r,1}$ is central}

The previous section proved that each test function $\phi_{r, \chi}$ belongs to the center $\mathcal Z(G_r, I_r, \chi_r)$.  In the present section we deduce that the test function $\phi_{r, 1}=\phi_r(R\Psi_1)$ is central in $\mathcal H(G_r, I_r^+)$.
  
\begin{prop} \label{central_prop} The $I^+$-level test function $\phi_{r,1} =  [I_r:I^+_r]^{-1} \, \phi_r(\pi_*(R\Psi_1))$ given by 
\begin{equation} \label{phi_r_as_sum} 
\phi_{r,1} = [I_r:I^+_r]^{-1} \, \sum_{\chi \in T(\mathbb F_p)^\vee} \phi_{r, \chi} ~ \in \mathcal H(G_r, I^+_r)
\end{equation}
lies in $\mathcal Z(G_r, I^+_r)$.   
\end{prop}

Later,  in Proposition \ref{exp_test_fcn},  we shall also find an explicit formula for $\phi_{r,1}$.   Recall that the test function $\phi_{r,1} = \phi_r(R\Psi_1)$ resp. $\phi_{r,\chi} = \phi_r(R\Psi_\chi)$ is normalized using the Haar measure which gives $I^+_r$ resp. $I_r$ volume 1.  This explains why we divide out by the factor $[I_r:I^+_r] = (q-1)^d$ on the right side of (\ref{phi_r_as_sum}).

\begin{proof}
We define convolution on $\mathcal H(G_r,I^+_r)$ using the Haar measure $dx$ which gives $I^+_r$ volume 1, so that $e_{I^+_r} := {\rm char}(I^+_r)$ is the identity element.  For each $\xi \in T(\mathbb F_{p^r})^\vee$, define $e_\xi \in \mathcal H(G_r,I_r, \xi) \subset \mathcal H(G_r,I^+_r)$ by
$$
e_\xi(x) := \begin{cases} {\rm vol}_{dx}(I_r)^{-1} \, \rho_\xi(x^{-1}), \,\,\,\, \mbox{if $x \in I_r$} \\ 0, \,\,\,\,\mbox{otherwise}. \end{cases}
$$
Here $\rho_\xi : I_r \rightarrow \bar{\QQ}_\ell^\times$ is determined from $\xi: T(\mathbb F_{p^r}) \rightarrow \bar{\QQ}_\ell^\times$ just as in (\ref{rho_from_chi}).

The  set $\{ e_\xi ~ | ~ \xi \in T(\mathbb F_{p^r})^\vee \}$ forms a complete set of orthogonal idempotents, i.e.,
\begin{align*}
e_{I^+_r} &= \sum_\xi e_\xi\\
e_{\xi} \, e_{\xi'} &= \delta_{\xi, \xi'} \, e_\xi, \spa \forall \xi, \xi' \in T(\mathbb F_{p^r})^\vee.
\end{align*}
The first equation is clear and the second follows from the Schur orthogonality relations.   

Let $1_{wt}$ denote the characteristic function of $I_r^+wtI_r^+$, for $w \in \widetilde{W} \hookrightarrow G(F_r)$ and $t \in T(\mathbb F_{p^r})$.   A calculation shows that for any $\xi$,
\begin{equation} \label{calculation}
1_{wt} \, e_\xi = e_{^w\xi} \, 1_{wt} \ .
\end{equation}
  Thus, if $[\xi]$ denotes the $W$-orbit of $\xi \in T(\mathbb F_{p^r})^\vee$, and if we set $e_{[\xi]} := \sum_{\xi' \in [\xi]} e_{\xi'}$, then $e_{[\xi]} \in \mathcal Z(G_r,I^+_r)$ and the following result holds.

\begin{lemma} \label{HA_decomp}  For $[\xi] \in W\backslash T(\mathbb F_{p^r})^\vee$, the functions $e_{[\xi]}$ form a complete set of {\em central} idempotents of $\mathcal H(G_r,I_r^+)$.  Hence the $e_{[\xi]}$ give the idempotents in the Bernstein center which project the category $\mathcal R(G_r)$ onto the various Bernstein components $\mathcal R_{\mathfrak s_{\xi}}(G_r)$ relevant for $\mathcal H(G_r,I^+_r)$.  That is, there is a canonical isomorphism of algebras
\begin{equation} 
\mathcal H(G_r, I^+_r) = \prod_{\xi \in W \backslash T(\mathbb F_{p^r})^\vee} \mathcal H(G_r, I^+_r) e_{[\xi]}
\end{equation} and,  for any smooth representation $(\pi, V) \in \mathcal R(G_r)$, the $G_r$-module spanned by $e_{[\xi]} V$ is the component of $V$ lying in the subcategory $\mathcal R_{\mathfrak s_\xi}(G_r)$.\qed
\end{lemma}

For each $[\chi]  \in W\backslash T(\mathbb F_p)^\vee$ the function $\phi_{r, [\chi]} := \sum_{\chi' \in [\chi]} \phi_{r, \chi'}$ is $[I_r\!:\!I^+_r]$ times the projection $e_{[\chi_r]}\phi_{r,1}$ of $\phi_{r,1}$ onto the Hecke algebra associated to the Bernstein component corresponding to the inertial class $s_{\chi_r}$.   To prove that $\phi_{r,1}$ is central is equivalent to showing  that each $\phi_{r, [\chi]}$ is central.

Fix $\chi \in T(\mathbb F_p)^\vee$ resp. its Weyl-orbit $[\chi]$.  Let $(\pi, V)$ be any irreducible smooth representation of $G_r$. 

\begin{lemma}
Let $t \in T(\mathbb F_{p^r})$ and $w \in \widetilde{W}$.  Set $\chi' := \, ^w\chi$.   Write  $\rho_r$ resp. $\rho'_r$ for $\rho_{\chi_r}$ resp.
 $\rho_{\chi'_r}$.   Then the following diagram commutes:
$$
\xymatrix{
V^{\rho_r} \ar[r]^{1_{wt}} \ar[d]_{\phi_{r,\chi}} & V^{\rho'_r} \ar[d]^{\phi_{r,\chi'}} \\
V^{\rho_r} \ar[r]^{1_{wt}} & V^{\rho'_r}.}
$$
\end{lemma}

\begin{proof}
We may assume $V^{\rho_r} \neq 0$, in which case we may view $V$ as a subquotient of $i^{G_r}_{B_r}(\widetilde{\chi}_r \, \eta)$, for some unramified character $\eta$ on $T(F_r)$. 

Recall that we chose the uniformizer $\varpi=p$, and that we agreed to abbreviate the $\varpi$-canonical extension $\widetilde{\chi}_r^\varpi$ resp. $\widetilde{\chi'}_r^\varpi$ by $\widetilde{\chi}_r$ resp. $\widetilde{\chi}'_r$.  Using the corresponding Hecke algebra isomorphism $\Psi_{\widetilde{\chi}_r}$ resp. $\Psi_{\widetilde{\chi}'_r}$ we derive the following equalities:
\begin{align*}
 & \mbox{the scalar by which $\phi_{r,\chi}$ acts on $i^{G_r}_{B_r}(\widetilde{\chi}_r \, \eta)^{\rho_r}$} 
\\
= ~ & \mbox{the scalar by which $\Psi_{\widetilde{\chi}_r}(\phi_{r,\chi})$ acts on $i^{H_{\chi_r}}_{B_{H_{\chi_r}}}(\eta)^{\rm triv}$} \\
= ~& \mbox{the scalar by which $\Psi_{\widetilde{\chi'}_r}(\phi_{r,\chi'})$ acts on 
$i^{H_{\chi'_r}}_{B_{H_{\chi'_r}}}(\, ^w\eta)^{\rm triv}$} \\
= ~ & \mbox{the scalar by which $\phi_{r,\chi'}$ acts on 
$i^{G_r}_{B_r}(\widetilde{\chi'}_r \, ^w \eta)^{\rho'_r}$} \\
= ~
& \mbox{the scalar by which $\phi_{r,\chi'}$ acts on 
$i^{G_r}_{B_r}(\widetilde{\chi}_r \, \eta)^{\rho'_r}$}.
\end{align*}
The first and third equalities follow from Proposition \ref{Z_comm_diag_prop}.  The second equality follows, 
via Proposition \ref{test_fcn_image} and the Bernstein isomorphism, from the equality 
$$
\sum_{\nu \in W_{\chi}\mu^{1*}_\chi} \eta(\varpi^\nu) = \sum_{\nu' \in W_{\chi'}\mu^{1*}_{\chi'}} \,^w\eta(\varpi^{\nu'})
$$
(note that the correspondence $\nu \mapsto \nu' := \, ^w\nu$ gives an equality term-by-term).  The fourth equality is obvious.
\end{proof}

The lemma implies easily that the function
$$
1_{wt} ~ \phi_{r, [\chi]} - \phi_{r, [\chi]}~ 1_{wt}
$$
annihilates every irreducible smooth representation of $G_r$, hence by the separation lemma (\cite{Be}, Lemma 9), the function is zero.  Since this holds for every $w,t$, we have proved that $\phi_{r, [\chi]} \in \mathcal Z(G_r, I^+_r)$, as desired.
\end{proof}

We are about to use (\ref{phi_r_as_sum}) to give a very explicit formula for $\phi_{r,1}$, but the centrality of $\phi_{r,1}$ is still far from obvious from the explicit formula itself.

\subsection{Explicit formula for the test function $\phi_{r,1}$}

\begin{prop} \label{exp_test_fcn}
Set $q := p^r$.  With respect to the Haar measure $dx$ on $G_r$ which gives $I^+_r$ volume 1, the test function $\phi_{r,1}$ is given by the formula ($t \in T(\mathbb F_q)$, $\, w \in \widetilde{W} \hookrightarrow G_r$)
\begin{equation}
\phi_{r,1}(I^+_r \, t w^{-1} \, I^+_r) = \begin{cases} 0,\,\,\,\,\, \mbox{if $w \notin {\rm Adm}^{G}(\mu)$} \\
0, \,\,\,\,\, \mbox{if $w \in {\rm Adm}^{G}(\mu)$ but $N_r(t) \notin T_{S(w)}(\mathbb F_p)$} \\ 
(-1)^d \, (p-1)^{d-|S(w)|} \, (1-q)^{|S(w)|-d-1},  \,\,\, \mbox{otherwise}. \end{cases}
\end{equation}
\end{prop} 
  
\begin{proof}
We shall use the explicit formulas for the $\phi_{r,\chi}$ given in Definition \ref{test_fcn},  along with the formula  
$$
\phi_{r,0}(w^{-1}) = k_{\mu,r}(w) = \begin{cases} (1-q)^{\ell(t_{\mu})-\ell(w)} = (1-q)^{|S(w)|-1}, \spa \mbox{if $w \in {\rm Adm}^G(\mu)$} \\
0, \spa \mbox{otherwise } \end{cases}
$$
(cf. Proposition \ref{main_comb_prop} (c)).

By Definition \ref{test_fcn} and Lemma \ref{chi_S_triv}, $\phi_{r,\chi}(tw^{-1})$ is zero, unless $w \in {\rm Adm}(\mu)$ and $\chi$ factors through the quotient $^wT(\mathbb F_p) = T(\mathbb F_p)/T_{S(w)}(\mathbb F_p)$ (cf. Lemma \ref{T_S_vs_wT})), in which case it is given by 
$$
\phi_{r,\chi}(tw^{-1}) = \chi_r^{-1}(t) ~ (1-q)^{|S(w)|-1}.
$$
Thus, for $w \in {\rm Adm}(\mu)$, equation (\ref{phi_r_as_sum}) implies
\begin{align*}
\phi_{r,1}(I^+_r \, tw^{-1} \, I^+_r) &= (q-1)^{-d} \, \sum_{\chi \in \, ^wT(\mathbb F_p)^\vee} \chi^{-1}_r(t) ~ (1-q)^{|S(w)|-1}\\
&= \begin{cases} 0, \spa \spa \mbox{if $N_r(t) \notin T_{S(w)}(\mathbb F_p)$} \\
(p-1)^{d-|S(w)|} ~ (1-q)^{|S(w)|-1} ~ (q-1)^{-d}, \spa \spa \mbox{otherwise}.
\end{cases}
\end{align*}
The proposition follows.
\end{proof}

\subsection{Compatibility with change of level}

Convolution with $e_{I_r}$ gives an algebra homomorphism $\mathcal Z(G_r, I^+_r) \rightarrow \mathcal Z(G_r,I_r)$.

\begin{cor} Let $\phi_{r,0}$ denote the Iwahori-level test function (computed using the measure which gives $I_r$ volume 1).  Then we have
$$e_{I_r} \phi_{r,1} = [I_r:I^+_r]^{-1} \, \phi_{r,0}.$$
\end{cor}

\begin{proof} Recall (Definition \ref{test_fcn}) that $\phi_{r, 0}$ coincides with $\phi_{r, {\rm triv}}$, where ${\rm triv}$ denotes  the trivial character.  Now multiply (\ref{phi_r_as_sum}) by $e_{I_r}$.
\end{proof}

\section{Pseudo-stabilization of the counting points formula}\label{pseud_stab_sec}

\subsection{Parabolic induction and the local Langlands correspondence}

First we let  $G$ denote any split connected reductive group over a $p$-adic field $F$.  We assume that the local Langlands correspondence holds for $G(F)$ and all of its $F$-Levi subgroups.  Let $W'_F = W_F \ltimes \bar{\QQ}_\ell$ denote the Weil-Deligne group, and for an irreducible smooth representation $\pi$ of $G(F)$ denote the Langlands parameter associated to $\pi$ by
$$
\varphi'_\pi: W'_F \rightarrow \widehat{G}(\bar{\QQ}_\ell).
$$
In what follows we will write $\widehat{G}$ instead of $\widehat{G}(\bar{\QQ}_\ell)$.   We are working with $\bar{\QQ}_\ell$ instead of $\mathbb C$, so we need to specify how the correspondence should be normalized.  Following Kottwitz \cite{K92b}, we choose a square root $\sqrt{p}$ of $p$ in $\bar{\QQ}_\ell$, and use it to define the square-root $| \cdot |_F^{1/2}: F^\times \rightarrow \bar{\QQ}_\ell^\times$ of the normalized absolute value $|\cdot|_F: F^\times \rightarrow \QQ^\times$.  Use this to define (unitarily normalized) induced representations such as $i^G_B(\xi)$ and to normalize the local Langlands correspondence.

Let $P = MN$ be an $F$-parabolic subgroup, with $F$-Levi factor $M$ and unipotent radical $N$.  Suppose $\pi$ resp. $\sigma$ is an irreducible smooth representation of $G(F)$ resp. $M(F)$.  There is an embedding $\widehat{M} \hookrightarrow \widehat{G}$, which is well-defined up to conjugation in $\widehat{G}$.  Then if $\pi$ is a subquotient of $i^G_P(\sigma)$, we expect that 
$$
\varphi'_\pi\vert_{W_F}: W_F \rightarrow \widehat{G}
$$
is $\widehat{G}$-conjugate to
$$
\varphi'_\sigma\vert_{W_F}: W_F \rightarrow \widehat{M} \hookrightarrow \widehat{G}.
$$

\noindent {\em CAUTION:} The expectation concerns {\em the restrictions of the parameters to the subgroup $W_F$ of $W'_F$}; obviously it is not true for the parameters themselves.

\begin{Remark}
The above expectation is true when $G = {\rm GL}_d$.  This  is a consequence of the way the local Langlands correspondence for supercuspidal representations is extended to all irreducible representations, cf. \cite{Rod}, $\S$4.4, which we now review.  Any representation $\varphi': W'_F \rightarrow {\rm GL}_d(\bar{\mathbb Q}_\ell)$  can be written as a sum of indecomposable representations
$$
\varphi' = (\tau_1 \otimes {\rm sp}(n_1)) \oplus \cdots \oplus (\tau_r \otimes {\rm sp}(n_r))\ ,
$$ where each $\tau_i$ is an irreducible  representation of $W'_F$ of dimension $d_i$, and 
 ${\rm sp}(n_i)$ is the special representation of dimension $n_i$.   If $\pi(\tau_i)$ is the supercuspidal representation of ${\rm GL}_{d_i}$ corresponding to $\tau_i$, then the local Langlands correspondence associates to $\varphi'$ the Bernstein-Zelevinski quotient 
$L(\Delta_1, \dots, \Delta_r)$, where $\Delta_i$ is the segment $[\pi(\tau_i), \pi(\tau_i(n_i-1))]$, cf.~\cite{Rod}.   This is because of Jacquet's calculation in \cite{Jac},  showing that the $L$- and $\epsilon$- factors of $\varphi'$ and of $L(\Delta_1, \dots, \Delta_r)$ agree, knowing the corresponding fact for each $\tau_i$ and $\pi(\tau_i)$.   Moreover, it is clear that any representation with the same supercuspidal support as $L(\Delta_1, \cdots, \Delta_r)$ has the same $\varphi$-parameter (up to permutation of the factors $\tau_i\, |\cdot|^j$, that is, up to $\widehat{G}$-conjugacy).    
\end{Remark}

\begin{Remark} \label{LLC_comp_rem}  We use the above compatibility only in the special case where $G= {\rm GL}_d$ and $M = T$ is the diagonal torus.  Recall that, for any $F$-split torus $T$, the Langlands parameter $\varphi'_{\xi}: W_F \rightarrow \widehat{T}(\bar{\QQ}_\ell)$ associated to a smooth character $\xi: T(F) \rightarrow \bar{\QQ}_\ell^\times $ (here $\widehat{T}(\bar{\QQ}_\ell)$ and $\bar{\QQ}_\ell^\times$ have the discrete topology) is given by the Langlands isomorphism $\xi \leftrightarrow \varphi'_\xi$
\begin{equation} \label{LLC_split_tori}
{\rm Hom}_{\rm conts}(T(F), \bar{\mathbb Q}_\ell^\times) = H^1_{\rm conts}(W_F, \widehat{T}) = 
{\rm Hom}_{\rm conts}(W_F, \widehat{T}),
\end{equation}
which is defined by requiring, for every $\nu \in X_*(T) = X^*(\widehat{T})$ and every $w \in W_F$,  the equality
\begin{equation} \label{LLC_comp_eq1}
\nu(\varphi'_\xi(w)) = \xi(\nu(\tau_F(w))).
\end{equation}
Here $\tau_F : W_F \rightarrow F^\times$ induces the isomorphism ${\rm Art}_F^{-1}: W^{\rm ab}_F ~ \widetilde{\rightarrow} ~ F^\times$ of local class field theory, and is normalized so that a geometric Frobenius $\Phi \in W_F$ is sent to a uniformizer in $F$.  

Suppose $\Phi \in W_F$ is a geometric Frobenius such that $\tau_F(\Phi) = \varpi$.  The compatibility above then asserts that, if $\pi$ is a subquotient of $i^G_B(\xi)$, then for any $\nu \in X_*(T) = X^*(\widehat{T})$ we have
\begin{equation} \label{LLC_comp_eq2}
\nu(\varphi'_{\pi}(\Phi)) = \xi(\varpi^\nu).
\end{equation} 
\end{Remark}

\subsection{A spectral lemma}

We now turn to the spectral characterization of the image of the test functions $\phi_{r, \chi}$ under the base change homomorphism (\ref{norm_map}). So here the local field is $F = \mathbb Q_p$, and the  reductive group $G$ over $F$ is the localization at $p$ of our global group.  We will often write $G$ in place of $G(\mathbb Q_p)$, etc.

Fix $\chi$ as before and set $f_{r, \chi} = b_r(\phi_{r, \chi}) \in \mathcal Z(G,I,\chi)$.  Given an irreducible smooth representation $\pi$ of $G(\mathbb Q_p)$, consider the homomorphism
$$
\varphi_\pi: W_{{\mathbb Q}_p} \times {\rm SL}_2(\bar{\QQ}_\ell) \rightarrow \widehat{G}(\bar{\QQ}_\ell)
$$
which is constructed from $\varphi'_\pi$ using the Jacobson-Morozov theorem.  Note that the $\widehat{G}$-conjugacy class of $\varphi_\pi$ is well-defined.  Let $\Phi \in W_{{\mathbb Q}_p}$ denote any geometric Frobenius element.  Let $I_p \subset W_{{\mathbb Q}_p}$ denote the inertia subgroup.  

\begin{lemma} \label{spectral_lemma}  Let $\mu$ denote the cocharacter of $G_{\mathbb Q_p}$ which under the isomorphism $G_{\mathbb Q_p} \cong {\rm GL}_d \times {\mathbb G}_m$ corresponds to $(\mu_0,1)$.  Let $\mu^* = -w_0(\mu)$ denote its dual coweight.  Let $(V_{\mu^*},r_{\mu^*})$ denote the irreducible representation of $^L(G_{{\mathbb Q}_p})$ with extreme weight $\mu^*$.   Let $\pi$ be an irreducible smooth representation of $G(\mathbb Q_p)$.  Then $f_{r, \chi}  \in \mathcal Z(G,I,\chi)$ acts by zero on $\pi$,  unless $\pi$ belongs to the category $\mathcal R_{\mathfrak s_\chi}(G)$.  In that case, $f_{r, \chi} $ acts on $\pi^{\rho_\chi}$ by the scalar
\begin{equation} \label{Tr(Phi^r_V_mu*)}
p^{r \langle \rho, \mu^*\rangle }\, {\rm Tr}\big(r_{\mu^*} \circ \varphi_{\pi}(\Phi^r \times \begin{tiny} \begin{bmatrix} p^{-r/2} & 0 \\ 0 & p^{r/2} \end{bmatrix} \end{tiny}) , \spa V_{\mu^*}^{I_p}\big).
\end{equation}
\end{lemma}

Note that $p^{r \langle \rho, \mu^* \rangle }$ and ${\rm Tr}(r_{\mu^*}\circ \varphi_{\pi}(\cdots))$ each depend on the choice of $\sqrt{p} \in \bar{\QQ}_\ell^\times$ used to normalize the correspondence $\pi \leftrightarrow \varphi_\pi$, but their product is independent of that choice.

\begin{proof}
The first part of the statement is clear, so we may assume $\pi$ belongs to $\mathcal R_{\mathfrak s_\chi}(G)$. 

Without loss of generality, we may assume $\tau_{\mathbb Q_p}(\Phi) = \varpi = p$, viewed as a uniformizer in $\mathbb Z_p$.  Let $\widetilde{\chi}=\widetilde{\chi}^{\varpi}$ denote the $\varpi$-canonical extension of $\chi$.   There is an unramified character $\eta$ of $T(\mathbb Q_p))$ such that $\pi$ is an irreducible subquotient of $i^G_B(\widetilde{\chi}^\varpi \, \eta)$.  To ease notation, set $\xi := \widetilde{\chi}\, \eta$.

Now we consider the Langlands homomorphism $\varphi'_{\xi} : W_{\mathbb Q_p} \rightarrow \widehat{T}$.  Following Roche (\cite{Ro}, p.~394), we see that the restriction of $\varphi'_\xi$ to the inertia subgroup $I_p \subset W_{\mathbb Q_p}$ depends only on $\chi$, and we denote that restriction by $\tau_\chi$.  We consider its image ${\rm im}(\tau_\chi) = \tau_\chi(I_p) \subset \widehat{T}$.  

It is clear that $V_{\mu^*}^{I_p} = V_{\mu^*}^{{\rm im}(\tau_\chi)}$ is the sum of $\widehat{T}$-weight spaces 
$$
\sum_{\nu} V_{\mu^*}(\nu) ,
$$
where $\nu$ runs over the characters of $\widehat{T}$ such that $\nu({\rm im}(\tau_\chi)) = 1$ or, equivalently (cf. (\ref{LLC_comp_eq1})), such that 
$\chi \circ \nu(\mathbb Z_p^\times) = 1$ (viewing $\nu$ as a cocharacter).  

Now because $\chi$ is a character which is 
trivial on the $\mathbb G_m$-factor of $T = (\mathbb G_m^d) \times \mathbb G_m$, to continue we may ignore that $\mathbb G_m$-factor and proceed as if $G = {\rm GL}_d$.  In that case $(V_{\mu^*}, r_{\mu^*})$ is just the contragredient of the standard representation of $\widehat{G} = {\rm GL}_d(\bar{\QQ}_\ell)$ and so the weights of $V_{\mu^*}$ are just the vectors $-e_i,\,\,$ ($i = 1, \dots, d$).  If $\nu = -e_i$, we have $\chi \circ \nu(\mathbb Z_p^\times) = 1$ if and only if $\chi_i = {\rm triv}$.

So, if there is no index $i$ with $\chi_i = {\rm triv}$, then $V_{\mu^*}^{I_p} = 0$.  But then we also have $f_{r, \chi} = 0$ (cf. Definition \ref{test_fcn}), and so the lemma holds in this case.

Now assume there is such an index $i$, and as before denote by $\mu^1_\chi$ the $M_\chi$-dominant element of the $W_\chi$-orbit 
$$
\{ e_i ~ | ~ \chi \circ (-e_i)(\mathbb Z_p^\times) =1 \}.
$$ 
Then we find that (\ref{Tr(Phi^r_V_mu*)}) is equal to 
\begin{equation} \label{scalar}
p^{r \langle \rho, \mu^* \rangle} \, \sum_{\nu \in W_\chi \mu^{1*}_\chi} \nu(\varphi'_{\pi}(\Phi^r)) = 
p^{r \langle \rho, \mu^* \rangle} \, \sum_{\nu \in W_\chi \mu^{1*}_\chi} \eta(\varpi^{r\nu}),
\end{equation}
where we used (\ref{LLC_comp_eq2}),  together with $\widetilde{\chi} \eta(\varpi^{r\nu}) = \eta(\varpi^{r\nu})$.

On the other hand, using Proposition \ref{test_fcn_image}, the quantity (\ref{scalar}) is easily seen to be the scalar by which 
$\Psi_{\widetilde{\chi}}(f_{r, \chi} ) = b_r (\Psi_{\widetilde{\chi}}(\phi_{r,\chi}))$ acts on $i^{M_\chi}_{B_{M_\chi}}(\eta)^{I_{M_\chi}}$.  By Lemma \ref{HA_comp_lemma}, this is the scalar by which $f_{r, \chi} $ acts on $i^G_B(\widetilde{\chi} \, \eta)^{\rho_\chi}$.  This completes the proof. 
\end{proof}

\begin{cor} \label{spectral_cor}
Let $f_{r, 1} = [I:I^+]^{-1} \sum_{\chi \in T({\mathbb F}_p)^\vee}f_{r, \chi} =b_r(\phi_{r, 1})$.  Then for every irreducible smooth representation $\pi$ of $G(\mathbb Q_p)$, we have
$$
{\rm tr} \, \pi(f_{r, 1} ) = {\rm dim}(\pi^{I^+}) \, p^{r \langle \rho, \mu^* \rangle} \, {\rm Tr}\big(r_{\mu^*}\circ  \varphi_{\pi}(\Phi^r \times \begin{tiny} \begin{bmatrix} p^{-r/2} & 0 \\ 0 & p^{r/2} \end{bmatrix} \end{tiny}) ,\spa V_{\mu^*}^{I_p}\big).
$$
Here, the trace on the left hand side is computed using the Haar measure which gives $I^+$ volume 1.
\end{cor}

\begin{Remark}
The above is an equality of numbers in $\bar{\QQ}_\ell$.  We may also view it as an equality of numbers in $\mC$, when considering the usual framework of representations $\pi$ over $\mathbb C$,  and Langlands parameters taking values in $^LG = \widehat{G}(\mC) \rtimes W_{\mathbb Q_p}$.   We shall do this in what follows.
\end{Remark}

\subsection{Definition of the semi-simple local $L$-function} \label{L^ss_def}

We define the semi-simple local $L$-factor $L^{ss}(s,\pi_p,r_{\mu^*})$ by an expression like that on the right hand side of Corollary \ref{spectral_cor}, as follows.

Let $\pi_p$ denote an irreducible smooth representation of $G(\QQ_p)$ (on a $\mC$-vector space), and let $\varphi'_{\pi_p} : W'_{\QQ_p} \rightarrow \widehat{G}(\mC)$ denote its Langlands parameter.  Let $(V,r_V)$ denote a rational representation of $\widehat{G}(\mC)$.  We define the semi-simple trace of 
$$
\varphi'_{\pi_p}(\Phi^r) = \varphi_{\pi}(\Phi^r \times \begin{tiny} \begin{bmatrix} p^{-r/2} & 0 \\ 0 & p^{r/2} \end{bmatrix} \end{tiny})
$$
acting on $V$ by
$$
{\rm Tr}^{ss}(\varphi'_{\pi_p}(\Phi^r), V): = {\rm Tr}(\varphi'_{\pi_p}(\Phi^r),V^{I_p}).
$$
Here $V^{I_p}$ is an abbreviation for $V^{r_V\circ \varphi'_{\pi_p}(I_p)}$.  Now define the function $L^{ss}(s,\pi_p, r_V)$ of a complex variable $s$ by the identity
$$
{\rm log}(L^{ss}(s,\pi_p,r_V)) = \sum_{r=1}^\infty {\rm Tr}^{ss}(\varphi'_{\pi_p}(\Phi^r), V) \, \dfrac{p^{-rs}}{r}.
$$

\subsection{Determination of the semi-simple local zeta function} \label{L_factor_sec}

We now return to the notational set-up of section \ref{moduli_sec}. We make the additional assumption that  $D$ is a division algebra.  Then the group $G_{\rm der}$ is anisotropic over $\QQ$ and  $G$ has ``no endoscopy'' in the sense that the finite abelian groups $\mathfrak K(I_0/\mathbb Q)$ are all trivial, cf. \cite{K92b}. (Recall that $I_0 = G_{\gamma_0}$ for a semi-simple $\gamma_0 \in G(\QQ)$.)  
Furthermore, the moduli scheme 
$\A_1 = \A_{1,K^p}$ is proper over $\Spec \mZ_p$, cf. \cite{K92}. 
Our goal is to express $Z^{ss}_{\mathfrak p}(s, \A_{1, K^p})$, the local factor of the semi-simple zeta function, in terms of the local automorphic 
L-functions $L^{ss}(s, \pi_p, r_{\mu^*})$.  That is, we will prove Theorem \ref{main_tr_thm} and Corollary \ref{zeta_cor} in the Introduction.  Our strategy is exactly the same as that of Kottwitz \cite{K92b}, so we shall give only a brief outline after recalling the definitions.

In our case, where $E_{\mathfrak p} = \QQ_p$, the semi-simple local zeta function $Z^{ss}_{\mathfrak p}(s, \A_{1,K^p})$ is the function of a complex variable $s$ determined by the equality
\begin{align*}
{\rm log}(Z^{ss}_{\mathfrak p}(s,\A_{1,K^p})) &= \sum_{r=1}^\infty \Big(\sum_i (-1)^i {\rm Tr}^{ss}(\Phi^r_{\mathfrak p} \, , \, H^i(\A_{1,K^p} \otimes_E \bar{E}_\mathfrak p \, , \, \bar{\QQ}_\ell) \Big) \dfrac{p^{-rs}}{r} \\
&= \sum_{r=1}^\infty {\rm Lef}^{ss}(\Phi^r_\fp \, , \, R\Psi_1) ~ \dfrac{p^{-rs}}{r}.
\end{align*}
For the second equality we are using the identification
$$
H^i(\A_{1,K^p} \otimes_E \bar{E}_\fp \, , \, \bar{\QQ}_\ell) = H^i(\A_{1,K^p} \otimes \bar{\mF}_p \, ,\, R\Psi_1)
$$
and the Grothendieck-Lefschetz fixed-point theorem for semi-simple traces (see \cite{HN}).  The quantities
\begin{equation} \label{eq}
{\rm Lef}^{ss}(\Phi^r_\fp, R\Psi_1) = \sum_{(\gamma_0;\gamma, \delta)} c(\gamma_0;\gamma,\delta) \, {\rm O}_\gamma(f^p) \, {\rm TO}_{\delta \sigma}(\phi_{r,1}),
\end{equation}
which {\em a priori} belong only to $\bar{\QQ}_\ell$, actually belong to $\QQ$ because $\phi_{r,1}$ takes rational values (cf. Proposition \ref{exp_test_fcn}) and because of the choices for the various measures defining the right hand side (see \cite{K90}, $\S3$).  Thus $Z^{ss}_{\mathfrak p}(s, \A_{1,K^p})$ is a well-defined function of a complex variable $s$, independent of any choice of an isomorphism $\bar{\QQ}_\ell \cong \mC$.  

Let us simplify the right hand side of (\ref{eq}), following Kottwitz's method (and freely using his notation).  Arguing as in \cite{K90}, $\S4$ and 
\cite{K92b}, $\S5$, we see that $\mathfrak K(I_0/\QQ) = 1$ implies that the right hand side of (\ref{eq}) is equal to 
$$
\tau(G)\sum_{\gamma_0} \sum_{(\gamma,\delta)} e(\gamma,\delta) \cdot {\rm O}_\gamma(f^p) \cdot {\rm TO}_{\delta \sigma}(\phi_{r,1}) \cdot {\rm vol}(A_G(\mR)^\circ \backslash I(\mR))^{-1},
$$
where $I$ denotes the inner form of $I_0 = G_{\gamma_0}$ mentioned at the beginning of section \ref{counting_pts_sec}.  As in \cite{K92b}, p.662, $e(\gamma,\delta)$ is the product of the Kottwitz signs
$$
e(\gamma,\delta) = e(\gamma) e(\delta) e_\infty(\gamma_0)
$$
defined there.  The ``fundamental lemma'' (cf. Theorem \ref{fl}) asserts that for every semi-simple $\gamma_p \in G(\QQ_p)$,
\begin{equation} \label{fl_eq}
{\rm SO}_{\gamma_p}(f_{r, 1} ) = \sum_{\delta} e(\delta) \, {\rm TO}_{\delta \sigma}(\phi_{r,1}),
\end{equation}
where $\delta$ ranges over a set of representatives for the $\sigma$-conjugacy classes of elements $\delta \in G(L_r)$ such that $N_r(\delta)$ is conjugate to $\gamma_p$ in $G(\bar{\QQ}_p)$. Note that, since $G_{\mQ_p}\cong {\rm GL}_d\times \mG_m$, there is at most one summand on the right hand side of (\ref{fl_eq}). 

Letting $(-1)^{d-1}f_\infty$ denote a pseudo-coefficient of the packet $\Pi_\infty$, we have ${\rm SO}_{\gamma_0}(f_\infty) = 0$,  unless the semi-simple element $\gamma_0$ is $\mR$-elliptic, in which case 
$${\rm SO}_{\gamma_0}(f_\infty) = e_\infty(\gamma_0) \, {\rm vol}(A_G(\mR)^\circ \backslash I(\mR))^{-1} ,$$
 cf. \cite{K92b}, Lemma 3.1.  
Since $\gamma_0$ in (\ref{eq}) ranges only over $\mR$-elliptic elements, we  can put these remarks together to see
\begin{equation} \label{eq_f}
 {\rm Lef}^{ss}(\Phi^r_\fp, R\Psi_1) = \tau(G) \sum_{\gamma_0} {\rm SO}_{\gamma_0}( f^p \, f_{r, 1}\, f_\infty),
\end{equation}
where now $\gamma_0$ ranges over all stable conjugacy classes in $G(\QQ)$, not just the $\mR$-elliptic ones.

 Since $G_{\rm der}$ is anisotropic, the trace formula for any $f \in C^\infty_c(A_G(\mR)^\circ \backslash G(\mA))$ is given by
\begin{equation} \label{TF}
\sum_{\gamma} \tau(G_\gamma) \, {\rm O}_{\gamma}(f) = \sum_\pi m(\pi) \, {\rm tr} \, \pi(f),
\end{equation}
where $\gamma$ ranges over conjugacy classes in $G(\QQ)$ and $\pi$ ranges over irreducible representations in ${\rm L}^2(G(\QQ)A_G(\mR)^\circ \backslash G(\mA))$.   By \cite{K92b}, Lemma 4.1, the vanishing of all $\mathfrak K(I_0/\QQ)$ means that 
\begin{equation} \label{O=SO}
\sum_{\gamma} \tau(G_\gamma) \, {\rm O}_{\gamma}(f) = \tau(G) \sum_{\gamma_0} {\rm SO}_{\gamma_0}(f),
\end{equation}
where $\gamma_0$ ranges over a set of representatives for the stable conjugacy classes of $\gamma_0 \in G(\QQ)$.  Combining equations (\ref{eq_f}), (\ref{TF}), and (\ref{O=SO}), we get
\begin{equation} \label{tr_pf}
{\rm Lef}^{ss}(\Phi^r_\fp, R\Psi_1) = \sum_\pi m(\pi) \, {\rm tr} \, \pi(f^p f_{r, 1} f_\infty),
\end{equation}
which proves Theorem \ref{main_tr_thm}.

Now we can deduce Corollary \ref{zeta_cor}.  The right hand side of (\ref{tr_pf}) can be written as 
\begin{equation*}
\sum_{\pi_f} \sum_{\pi_\infty \in \Pi_\infty} m(\pi_f \otimes \pi_\infty) \cdot {\rm tr}\, \pi^p_f(f^p) \cdot {\rm tr} \, \pi_p(f_{r, 1} ) \cdot {\rm tr} \, \pi_\infty(f_\infty), 
\end{equation*}
that is, invoking Corollary \ref{spectral_cor}, as
\begin{equation*}
\sum_{\pi_f} a(\pi_f) \, {\rm dim}(\pi_f^K) \, p^{nr/2} \, {\rm Tr}^{ss}(\varphi'_{\pi_p}(\Phi^r) ,\ V_{\mu^*}).
 \end{equation*}
Corollary \ref{zeta_cor} follows easily from this and the definitions. \qed

\section{Appendix: Support of Bernstein functions}

\subsection{Nonstandard descriptions of the Bernstein functions}\label{nonst}

Let $\mathcal H$ denote the affine Hecke algebra associated to a based root system $\Sigma$, with parameters $v_s$ (for $s$ a simple affine reflection).  In the equal parameter case, each $v_s$ should be thought of as $q^{1/2}$.  Let $\widetilde{T}_{s} = v_s^{-1}T_s$.  For $w \in \widetilde{W}$ with reduced expression $w = s_1 \cdots s_r \tau$, write $\widetilde{T}_w := \widetilde{T}_{s_1} \cdots \widetilde{T}_{s_r} T_\tau$ (this element of 
$\mathcal H$ is independent of the choice of reduced expression for $w$).  For a translation element $t_\lambda \in \widetilde{W}$, write $\widetilde{T}_\lambda$ in place of $\widetilde{T}_{t_\lambda}$.    Setting $Q_s := v_s^{-1} - v_s$, we have the relation
$$
\widetilde{T}_s^{-1} = \widetilde{T}_s + Q_s.
$$

Suppose the simple affine reflections are the reflections through the walls of a (base) alcove ${\bf a} \subset X_* \otimes \mathbb R$.  
Denote the Weyl chamber that contains ${\bf a}$ by $\mathcal C_0$.  This determines the notion of positive root (those which are positive on $\mathcal C_0$) and dominant coweight (those which pair to $\mathbb R_{\geq 0}$ with any positive root or, equivalently, those belonging to the closure of $\mathcal C_0$).  For any $\lambda \in X_*$ we define 
$$
\Theta^{\mathcal C_0}_\lambda = \widetilde{T}_{\lambda_1} \widetilde{T}_{\lambda_2}^{-1},
$$
where $\lambda = \lambda_1 - \lambda_2$ and each $\lambda_i$ is {\em dominant}.  The element $\Theta^{\mathcal C_0}_\lambda$ is independent of the choice of the $\lambda_i$ (use that for dominant coweights $\nu,\mu$, we have $T_{\nu + \mu} = T_\nu T_\mu$).  

 Now let $\mathcal C$ be any Weyl chamber in the same apartment and having the same origin as $\mathcal C_0$.  We would like to define in an analogous manner an element $\Theta^{\mathcal C}_\lambda$ by setting
$$
\Theta^{\mathcal C}_\lambda := \widetilde{T}_{\lambda_1} \widetilde{T}^{-1}_{\lambda_2},
$$
where $\lambda = \lambda_1 - \lambda_2$ and each $\lambda_i$ is $\mathcal C$-{\em dominant}, i.e.,  it lies in the closure of the chamber $\mathcal C$.  However, we need to show that this definition is independent of the choice of the $\lambda_i$.  This follows from the previous independence statement and the following standard lemma.  Indeed, in using the lemma we take $v \in W$ to be the unique element such that $v\mathcal C_0 = \mathcal C$.

\begin{lemma}  For any dominant coweight $\lambda$ and any element $v \in W$, we have
$$
T_v T_\lambda T_v^{-1} = T_{v\lambda}.
$$
\end{lemma}

\begin{proof} 
First note that $vt_\lambda = t_{v\lambda}v$ and that the dominance of $\lambda$ implies $\ell(vt_\lambda) = \ell(v) + \ell(t_\lambda) = \ell(v) + \ell(t_{v\lambda})$.  These statements then imply $T_vT_{t_\lambda} = T_{vt_\lambda} = T_{t_{v\lambda}v} = T_{t_{v\lambda}} T_v$.  The desired formula follows.
\end{proof}

This proves that $\Theta^{\mathcal C}_\lambda$ is well-defined, and at the same time proves the following simple 
proposition.  We set $z_\mu^{\mathcal C} = \sum_{\lambda \in W\mu} \Theta^{\mathcal C}_\lambda$.  From now on we will write $\Theta_\lambda$ (resp. $z_\mu$) in place of $\Theta^{\mathcal C_0}_\lambda$ (resp. $z^{\mathcal C_0}_\mu$).   
Recall we know already that $z_\mu \in Z(\mathcal H)$.   We call $z_\mu$ the {\em Bernstein function} associated to $\mu$.

\begin{prop} \label{main_prop}  Let $v \in W$ be the unique element such that $v{\mathcal C}_0 = \mathcal C$.  Then 
$$
T_v \Theta_\lambda T^{-1}_v = \Theta^{\mathcal C}_{v\lambda},
$$
and thus
$$
z_{\mu}^{\mathcal C} = T_v^{-1} z_\mu T_{v} = z_\mu.
$$\qed
\end{prop}

\subsection{Alcove walk descriptions of Bernstein functions}

We apply A. Ram's theory of alcove walks \cite{Ram}, as developed by U. G\"{o}rtz \cite{G2}.

The theory of alcove walks,  with the orientation determined by the Weyl chamber $\overline{\mathcal C}$ (the one opposite $\mathcal C$),  gives alcove walk descriptions for the elements $\Theta^{\mathcal C}_\lambda$.  They take the following form.  Let ${\bf b}$ denote an alcove which is deep inside $\overline{\mathcal C}$.  For any expression 
$$
t_\lambda = s_{i_1} \cdots s_{i_k} \tau
$$
(which need not be reduced), we have the equality
$$
\Theta^{\mathcal C}_\lambda = \widetilde{T}^{\varepsilon_1}_{s_{i_1}} \cdots \widetilde{T}^{\varepsilon_k}_{s_{i_k}} T_\tau,
$$
where the signs $\varepsilon \in \{ \pm 1 \}$ are determined as follows.  For each $\nu$,  consider the alcove ${\bf c}_\nu := s_{i_1} \cdots s_{i_{\nu-1}}{\bf a}$, and denote by $H_\nu$ the affine root hyperplane containing the face shared by ${\bf c}_\nu$ and ${\bf c}_{\nu+1}$.  Set
$$
\varepsilon_\nu = \begin{cases} 1 \,\,\,\, \mbox{if ${\bf c}_\nu$ is on the same side of $H_\nu$ as ${\bf b}$} \\ -1 \,\,\,\, \mbox{otherwise}. \end{cases}
$$

We call such an expression for $\Theta^{\mathcal C}_\lambda$ an {\em alcove walk description}.  Clearly,  summing over $\lambda \in W\mu$ yields an alcove walk description for $z_\mu$.   Once we have fixed the expressions for each $t_\lambda$, each choice of Weyl chamber yields a different alcove walk description of $z_\mu$.  We call the ones that come from $\mathcal C \neq \mathcal C_0$ the {\em nonstandard alcove walk descriptions}.

\subsection{The support of $z_\mu$}

Recall that,  by definition,  ${\rm Adm}(\mu)$ consists of the finite set of elements $w \in \widetilde{W}$ such that $w \leq t_\lambda$ for some $\lambda \in W\mu$.  In this section we prove the following result.

\begin{prop}\label{suppz}
The support of $z_\mu$ is precisely the admissible set ${\rm Adm}(\mu)$.
\end{prop}

\begin{proof}
See \cite{H1a} for the inclusion ${\rm supp}(z_\mu) \subseteq {\rm Adm}(\mu)$.  (Alternatively: this inclusion is  obvious from the alcove walk description of the function $\Theta_\lambda$ which comes from a {\em reduced} expression for $t_\lambda$.)

Suppose $w \in {\rm Adm}(\mu)$.  Fix an element $\lambda \in W\mu$ such that $w \leq t_\lambda$ in the Bruhat order.  Now fix the Weyl chamber $\mathcal C$ such that the alcove $\lambda + {\bf a}$ belongs to the opposite Weyl chamber $\overline{\mathcal C}$.  

Note that $\lambda = 0-(-\lambda)$ and that $-\lambda$ is $\mathcal C$-dominant.  It follows from the definition of 
$\Theta^{\mathcal C}_\lambda$ that 
$$
\Theta^{\mathcal C}_\lambda = \widetilde{T}^{-1}_{-\lambda} = \widetilde{T}^{-1}_{s_{i_1}} \cdots \widetilde{T}^{-1}_{s_{i_k}} T_\tau,
$$
where $t_\lambda = s_{i_1} \cdots s_{i_k} \tau$ is any {\em reduced} expression for $t_\lambda$. 

But from \cite{H1b}, Lemma 2.5, any element $w \leq t_\lambda$ belongs to the support of the right hand side.  Furthermore, any element in the support of $\Theta^{\mathcal C}_\lambda$ also belongs to the support of $z_\mu$.  Indeed, it is enough to observe (\cite{H1a}, Lemma 5.1), that when a function of the form $\Theta^{\mathcal C}_\lambda = \widetilde{T}_{\lambda_1}\widetilde{T}^{-1}_{\lambda_2}$ is expressed as a linear combination of the basis elements $\widetilde{T}_w$ ($w \in \widetilde{W}$), then the coefficients which appear are polynomials in the renormalized parameters $Q_s$, {\em with non-negative integral coefficients}.  Thus, cancellation cannot occur when the various functions $\Theta^{\mathcal C}_\lambda$ are added in the formation of $z_\mu$.

This completes the proof of the proposition.
\end{proof}

\small
\bigskip
\obeylines
\noindent
University of Maryland
Department of Mathematics
College Park, MD 20742-4015 U.S.A.
email: tjh@math.umd.edu

\bigskip
\obeylines
Mathematisches Institut der Universit\"at Bonn  
Endenicher Allee 60 
53115 Bonn, Germany.
email: rapoport@math.uni-bonn.de


\end{document}